\documentclass[a4paper,10pt,twoside]{amsart}
\usepackage[english]{babel}
\usepackage[utf8]{inputenc}

\usepackage[a4paper,inner=2.2cm,outer=2.2cm,top=4cm,bottom=4cm,pdftex]{geometry}
\usepackage{fancyhdr}
\usepackage{color}
\usepackage{bold-extra}
\usepackage{mathrsfs}

\usepackage{comment}
\usepackage{graphics}
\usepackage{aliascnt}
\usepackage[pdftex,citecolor=green,linkcolor=red]{hyperref}
\usepackage{mathtools}

\usepackage{amsmath}
\usepackage{amsfonts}
\usepackage{amssymb}
\usepackage{amsthm}
\usepackage{enumerate}

\newtheorem{theorem}{Theorem}[section]
\newtheorem{corollary}[theorem]{Corollary}

\newtheorem{lemma}[theorem]{Lemma}
\newtheorem{proposition}[theorem]{Proposition}
\theoremstyle{definition}

\numberwithin{equation}{section}

\newcommand\E{\mathbb{E}}

\newcommand\eps{\varepsilon}
\renewcommand{\Re}{\textnormal{Re}}

\renewcommand{\pmod}[1]{\ (\mathrm{mod}\ #1)}
\renewcommand\d{\textnormal{ d}}

\newcommand{\sump}{\mathop{{\sum}^{\raisebox{-3pt}{\makebox[0pt][l]{$'$}}}}}

\begin{document}

\title[$\Omega$-Results for Exponential Sums Related to Maass Cusp Forms for $\mathrm{SL}_3(\mathbb Z)$]{$\Omega$-Results for Exponential Sums \\ Related to Maass Cusp Forms for $\mathrm{SL}_3(\mathbb Z)$}

\author{Jesse J\"a\"asaari}\address{Jesse J\"a\"asaari, Department of Mathematics and Statistics, University of Turku, 20014 Turku, Finland}\email{jesse.jaasaari@utu.fi}


\begin{abstract}   
We obtain $\Omega$-results for linear exponential sums with rational additive twists of small prime denominators weighted by Hecke eigenvalues of Maass cusp forms for the group $\mathrm{SL}_3(\mathbb Z)$. In particular, our $\Omega$-results match the expected conjectural upper bounds when the denominator of the twist is sufficiently small compared to the length of the sum. Non-trivial $\Omega$-results for sums over short segments are also obtained. Along the way we produce lower bounds for mean squares of the exponential sums in question and also improve the best known upper bound for these sums in some ranges of parameters.
\end{abstract}

\maketitle
     

\section{Introduction}

\noindent Hecke eigenvalues of automorphic forms are mysterious objects of arithmetic importance and thus it is highly desirable to understand how these numbers $a(m)$ are distributed. Generally one expects significant randomness in their distribution and a manifestation of this belief is that the convolution sums  
\begin{align}\label{convolution-sum}
\sum_{x\leq m\leq x+\Delta}a(m)b(m),
\end{align}
with $\Delta\ll x$, should exhibit cancellation for various sequences $\{b(m)\}_m$. There is an extensive literature concerning estimates for such sums in the case of classical modular forms, see e.g. \cite{Jutila1985, Jutila1987, Harcos, Ernvall--Hytonen-Karppinen, Fouvry-Ganguly2014, CPZ2020, He-Wang2023}, for a wide variety of sequences $\{b(m)\}_m$. However, such sums have been less studied for higher rank forms. Especially interesting situation is the case where\footnote{Throughout the paper $e(x)$ stands for $e^{2\pi ix}$.} $b(m)=e(m\alpha)$ for some fixed $\alpha\in\mathbb R$. These sums have long been connected to important questions in analytic number theory, e.g. the shifted convolution problem \cite{Harcos, Jutila1996, Munshi2013}, the subconvexity problem for twisted $L$-functions \cite{Blomer, Munshi2014}, and estimation of the second moment of automorphic $L$-functions in the $t$-aspect \cite{Miller}. Also, when $\Delta$ is small compared to $x$, the resulting short sums are closely related to the classical problems of studying various error terms, e.g. in the Dirichlet divisor problem or the Gauss circle problem, in short intervals. In the present article we shall investigate the extent of cancellation in sums (\ref{convolution-sum}) when $a(m)$ are chosen to be the Hecke eigenvalues\footnote{Fourier coefficients of $\mathrm{GL}_3$ Maass cusp forms are indexed by pairs of natural numbers and are typically denoted by $A(m,n)$. It is known that if the form is a Hecke eigenform and normalised so that $A(1,1)=1$, then the eigenvalue under the $m^{\text{th}}$ Hecke operator is given by $A(m,1)$.} $A(m,1)$ of a fixed Hecke--Maass cusp form for the group $\mathrm{SL}_3(\mathbb Z)$ and $b(m)$ are the exponential phases $e(m\alpha)$ with $\alpha\in\mathbb R$.

The relationship between exponential sums weighted by Hecke eigenvalues of automorphic forms and classical number theoretic error terms is particularly clear in the case where the twist $\alpha$ is close to a reduced fraction $h/k$ with a small denominator $k$. Furthermore, in this case the behaviour of such exponential sum is closely related to the behaviour of the sum at the fraction and hence it makes studying rationally additively twisted linear exponential sums, that is the setting where $b(m)=e(mh/k)$ for a reduced fraction $h/k$, particularly interesting. This is the set-up we restrict ourselves in the present work. Finally, on the technical side, good estimates for short rationally additively twisted exponential sums also provide a practical tool for reducing smoothing error (see e.g. \cite{Jutila1987, Vesalainen}), which is beneficial in variety of settings where linear exponential sums arise. 
               
For general linear twists the best known upper bound for long sums is
\begin{align*}
\sum_{m\leq x}A(m,1)e(m\alpha)\ll_\varepsilon x^{3/4+\varepsilon}
\end{align*}
uniformly in $\alpha\in\mathbb R$ due to Miller \cite{Miller} and sharper estimates are known when $\alpha$ is a rational number with a small denominator \cite{Jaasaari--Vesalainen1}. For shorter sums very little is known in the higher rank setting. The moment estimates in the $\mathrm{GL}_2$ situation \cite{Jutila, Ernvall-Hytonen2011, Ernvall-Hytonen2015, Vesalainen}, the square-root cancellation heuristics, and the shape of the truncated Voronoi identity for rationally additively twisted sums related to $\mathrm{SL}_3(\mathbb Z)$ Maass cusp forms derived in \cite{Jaasaari--Vesalainen1} give rise to the conjectural bounds
\begin{align}\label{conjecture}
\sum_{x\leqslant m\leqslant x+\Delta}A(m,1)e\left(\frac{mh}k\right)\ll_\varepsilon\text{min}\left(\Delta^{1/2}x^{\varepsilon},k^{1/2}x^{1/3+\varepsilon}\right).
\end{align}  

\noindent Our aim in this paper is to investigate what limitations there are for the extent of cancellation in rationally additively twisted sums attached to Maass cusp forms for the group $\mathrm{SL}_3(\mathbb Z)$ by giving stronger evidence towards the conjectural bounds (\ref{conjecture}). This is achieved in some ranges of parameters by establishing $\Omega$-results.

In this work we consider sums with a sharp cut-off $1\leq m\leq x$. Such sums occur naturally in many contexts throughout analytic number theory. One might also be interested in closely related sums with a smooth cut-off. These sums are technically easier to work with, but in certain ways their behaviour differs from the sums we are studying. For instance, if $w$ is a non-negative smooth function supported on the interval $[1/2,5/2]$ that is identically one on $[1,2]$, then a standard calculation using Mellin inversion formula and contour shifting shows that
\[
\sum_m A(m,1)w\left(\frac mx\right)\ll_A x^{-A}
\]
for any integer $A\geq 1$. This differs considerably from the conjectural bound
\[ 
\sum_{x\leq m\leq 2x}A(m,1)\ll_\eps x^{1/3+\eps}\]
for sums with a sharp cut-off, which is expected to be optimal. 

The key tools used to study rationally additively twisted exponential sums related to automorphic forms are the so-called Voronoi summation formulas. For example, classically (in the rank one setting) rationally additively twisted sums can be analysed using truncated Voronoi identities. These are reasonably sharp approximate formulations of full Voronoi summation formulas, which are essentially what one gets if one formally replaces the smooth cut-off function by a characteristic function of an interval in the Voronoi summation formula. Naturally, this formal substitution is analytically challenging as the Voronoi summation formulae typically require the test functions to be smooth enough. Nevertheless, in the classical setting twisted truncated Voronoi identities have been derived previously by Jutila \cite{Jutila1985} for the error term in the classical Dirichlet divisor problem and for sums involving Fourier coefficients of holomorphic cusp forms, and by Meurman \cite{Meurman2} for sums involving Fourier coefficients of $\mathrm{GL}_2$ Maass cusp forms. In higher rank a truncated Voronoi identity for the generalised divisor function $d_k(n)$ has been given in \cite[(3.23)]{Ivic3}, and a truncated Voronoi identity for plain sums of coefficients of fairly general $L$-functions has been obtained in \cite{Friedlander--Iwaniec}.

Rationally additively twisted Voronoi summation formulae have been implemented for $\mathrm{GL}_3$ by Miller and Schmid \cite{Miller--Schmid} using the framework of automorphic distributions and for $\mathrm{GL}_n$ in \cite{Goldfeld--Li1} by more classical means. However, these types of formulas have limitations in higher rank. Namely, the convergence problems indicated above concerning replacing the smooth cut-off with a sharp cut-off get more difficult as the rank increases and this leads to quite large error terms. For example, the truncated $\mathrm{GL}_3$ Voronoi identity in \cite{Jaasaari--Vesalainen1}, cited in a corrected form in the appendix below, has an error term which is actually larger than the expected optimal upper bound for the corresponding sum.

The main feature in truncated Voronoi identities is the interplay between the length of the sum on the dual side (the so-called truncated Voronoi series) and the size of the error term. When considering sums with a sharp cut-off, in typical situations requiring the Voronoi series to be short forces the error term to be larger and likewise smaller error term requires longer Dirichlet polynomial on the dual side. The main drawback in the Voronoi identity of \cite{Jaasaari--Vesalainen1} is that even under the Ramanujan--Petersson conjecture for $\mathrm{GL}_3$ Maass cusp forms the error term gives the dominant contribution (unless the length of the Voronoi series is taken to be large, which leads to other difficulties). This feature makes it hard to analyse exponential sums beneficially using such summation formulas. Indeed, the truncated Voronoi identities mentioned above are not suitable for good pointwise bounds and usually to evaluate moments precisely one needs the dual sum to be sufficiently short, in which case the error term becomes too large. In the higher rank setting this highlights the significant technical challenges faced in proving upper bounds of the right order of magnitude or $\Omega$-results for the sum in (\ref{conjecture}).

To circumvent these issues we take an alternate approach, building upon the paper \cite{Ivic--Matsumoto--Tanigawa} that studied a different question concerning obtaining upper bounds for the error term in the Rankin--Selberg problem. Essentially we relate sums with a sharp cut-off to very particularly smoothed sums (the so-called Riesz weighted sums), where the smoothing is tailored specifically so that it is possible to connect sums with a sharp cut-off to these smoothed sums with fairly simple arguments. For the latter sums very sharp Voronoi identities can be derived. Because of this, these new sums can be studied efficiently and this in turn yields information about the original exponential sum we are interested in.  

To be more precise, our approach relies on considering Riesz weighted exponential sums roughly\footnote{For an explicit expression for the residue terms, see (\ref{main-term}) below.} given by
\begin{align}\label{Riesz-sums}
\widetilde A_a\!\left(x;\frac hk\right):=\frac1{a!}\sump_{m\leqslant x}A(m,1)\,e\!\left(\frac{mh}k\right)\left(x-m\right)^a-\text{residue terms},
\end{align}
where $a\geqslant 0$ is an integer. Here and throughout the paper the primed summation notation means to take half the value at the endpoint in the case $x$ is an integer. This is a very classical form of smoothing. For sufficiently large order $a$, Voronoi identities can be derived for these sums and they are better behaved than the Voronoi identities for sums with a sharp cut-off we are actually interested in (which essentially corresponds to the case $a=0$). Obtaining these Voronoi identities in higher rank is of independent interest. It actually turns out that for Riesz weighted sums the Voronoi series on the dual side converges for $a\geq 2$ and so for these sums we do not need to truncate the Voronoi series unlike in the case $a=1$ (or $a=0$). 

For technical reasons we actually consider primed sums
\[ 
\sump_{m\leq x}A(m,1)e\left(\frac{mh}k\right).\]
The connection to ordinary sums is that 
\[ 
\sum_{m\leq x}A(m,1)e\left(\frac{mh}k\right)=\sump_{m\leq x}A(m,1)e\left(\frac{mh}k\right)+O_\eps\left(x^{\vartheta+\varepsilon}\right),
\]
where the error term arises when $x$ is an integer. Here $\vartheta\geqslant 0$ is the exponent towards the Ramanujan--Petersson conjecture for $\mathrm{SL}_3(\mathbb Z)$ Maass cusp forms. It is widely expected that $\vartheta=0$, but currently we only know that $\vartheta\leqslant 5/14$, see \cite[Appendix 2]{Kim--Sarnak}.

One strategy to obtain $\Omega$-results is to prove lower bounds for the second moment   
\begin{align}\label{exp-sum}
\int\limits_X^{2X}\left|\sump_{m\leq x}A(m,1)e\left(\frac{mh}k\right)\right|^2 \mathrm d x=\int\limits_X^{2X}\left|\widetilde A_0\left(x;\frac hk\right)\right|^2 \mathrm d x+O_\varepsilon\left(Xk^{3+\varepsilon}\right).
\end{align}
As discussed above, these moments are typically evaluated with the help of truncated Voronoi identities and this approach works well in the classical setting, but not so well in higher rank. Crucially for us, we will be able to evaluate second moments of $\widetilde A_a(x;h/k)$ for $a\geq 1$ using the Voronoi identities we establish, at least on average over $h$ (mod $k$). The asymptotic evaluation of the moments of higher order Riesz means can be converted into a lower bound for the second moment of ordinary exponential sums (\ref{exp-sum}) by elementary means explained in the following paragraph, which immediately implies an $\Omega$-result for these sums. 

The key property of the Riesz weighted sums is that 
\begin{align}\label{Key-property}
\int\limits_x^t \widetilde A_a\left(u;\frac hk\right)\,\mathrm d u=\widetilde A_{a+1}\left(t;\frac hk\right)-\widetilde A_{a+1}\left(x;\frac hk\right). 
\end{align}
This may be combined with the trivial identity
\begin{align*}
\widetilde A_a\left(x;\frac hk\right)=\frac1H\int\limits_x^{x+H}\widetilde A_a\left(x;\frac hk\right)\,\mathrm d t,
\end{align*}
and from these it quickly follows that 
\begin{align}\label{Key-relation}
\widetilde A_{a+1}\left(x;\frac hk\right)=\frac1H\int\limits_x^{x+H}\left( \widetilde A_{a+1}\left(t;\frac hk\right)-\int\limits_x^t\widetilde A_a\left(u;\frac hk\right)\,\mathrm d u\right)\,\mathrm d t.
\end{align}
An analogous identity for Riesz weighted sums attached to the Dirichlet series coefficients of the Rankin--Selberg $L$-function was used in \cite{Ivic--Matsumoto--Tanigawa} to relate upper bounds for the error term in the Rankin--Selberg problem to upper bounds for the error term in its first Riesz mean. The novelty of the present work is to observe that (\ref{Key-relation}), along with further arguments, may be used to relate moments of various Riesz weighted sums of different orders.
 
Indeed, one may e.g. essentially deduce (see the proof of Theorem \ref{lower bounds}) the estimate
\[ 
\int\limits_X^{2X}\left|\widetilde A_1\left(x;\frac hk\right)\right|^2\,\mathrm d x\ll \frac1{H^2}\int\limits_X^{2X}\left|\widetilde A_2\left(x;\frac hk\right)\right|^2\,\mathrm d x+H^2\int\limits_X^{3X}\left|\widetilde A_0\left(x;\frac hk\right)\right|^2\,\mathrm d x+\frac1H X^{10/3}k^5\]
for any $0<H\leq X$. The upshot is that, while we are unable to evaluate the mean square of $\widetilde A_0(x;h/k)$ directly, we may use our Voronoi identities to evaluate the other two moments asymptotically. Then choosing $H$ optimally leads to a lower bound for the second moment of $\widetilde A_0(x;h/k)$. One can also apply the identity (\ref{Key-relation}) directly to relate different sums $\widetilde A_a(x;h/k)$. An example of this is given e.g. in the proof Proposition \ref{A2toA1}. When studying short sums we need to work with even higher order Riesz weighted sums as for these sums we are unable to directly evaluate the second moment of $\widetilde A_1(x+\Delta;h/k)-\widetilde A_1(x;h/k)$ precisely as the error term in the Voronoi identity gives a larger contribution compared to expected size of the sum in this situation. 

To obtain Voronoi identities for the Riesz weighted exponential sums (\ref{Riesz-sums}) in the case $a\geq 2$ we use standard arguments based on Perron's formula and analysing the resulting integral involving quotients of gamma functions. However, we simplify this analysis by noting that these integrals are values of Meijer $G$-functions whose asymptotic behaviour is known. Establishing such identity in the case $a=1$ is more involved. The method used in the case $a\geq 2$ runs into difficulties due to lack of absolute convergence in certain sums and to circumvent this issue we need to implement some of the arguments in \cite{Jutila} to the rank two setting and develop them further. 

The problem of estimating moments of exponential sums alluded above is also a classical theme by its own right. Upper bounds for such moments have been obtained in many different settings, see for example \cite{Ivic2, Wu--Zhai, Jaasaari--Vesalainen1}. Moreover, in some cases even asymptotic behaviour is known \cite{Lester, Jaasaari} under additional hypothesis. While there are some results concerning upper bounds in higher rank settings, as far as the author is aware of, no asymptotics or lower bounds for these moments are known unconditionally in higher rank situations. There are essentially two reasons that make the higher rank cases difficult. Firstly, as indicated above, the dual sum in the truncated $\mathrm{GL}_3$ Voronoi identity has to be short enough so that the diagonal contribution to the second moment is not too large. This can be arranged, but unfortunately it forces the error term on the dual side to be too large in order to obtain good results. The second reason is the presence of Kloosterman sums that appear on the dual side of Voronoi identities when the sums are rationally additively twisted. In higher rank it is not even a priori clear that the expected main contribution in the second moment coming from the diagonal terms is positive. Note that such situation does not arise in the rank one setting where the exponential phases cancel each other. The strategy to overcome the first issue has been described above. The problem with Kloosterman sums is resolved by averaging over the numerators $h\in\mathbb Z_k^\times$. In this way we will obtain lower bounds for the averaged mean squares 
\begin{align*}
\sum_{h\in\mathbb Z_k^\times}\int\limits_X^{2X}\left|\sump_{x\leq m\leq x+\Delta}A(m,1)e\left(\frac{mh}k\right)\right|^2 \mathrm d x
\end{align*}
in certain ranges of parameters involved (in particular for $\Delta\asymp X$) when $k$ is a sufficiently small prime (or $k=1$). For simplicity we restrict ourselves to prime denominators in the twists throughout the paper, but it is plausible that the arguments could be made to work for more general $k$. 

\section{The main results}        

\noindent Throughout the paper we use standard asymptotic notation. If $f$ and $g$ are complex-valued functions defined on some set, say $\mathcal D$, then we write $f\ll g$ to  signify that $|f(x)|\leqslant C|g(x)|$ for all $x\in\mathcal D$ for some implicit constant $C\in\mathbb R_+$. The notation $O(g)$ denotes a quantity that is $\ll g$, and $f\asymp g$ means that both $f\ll g$ and $g\ll f$. We write $f=o(g)$ if $g$ never vanishes in $\mathcal D$ and $f(x)/g(x)\longrightarrow 0$ as $x\longrightarrow\infty$. The notation $f=\Omega(g)$ means that $f\neq o(g)$.

Our first main result shows the expected $\Omega$-result for rationally additively twisted long sums when the denominator of the twist is sufficiently small. As far as the author is aware of, this is the first time $\Omega$-results have been obtained for rationally additively twisted sums in a higher rank setting. Although it is not explicitly mentioned in what follows, we assume throughout the paper that the underlying Maass cusp form is a Hecke eigenform and normalised so that $A(1,1)=1$. Our results should generalise for arbitrary $\mathrm{GL}_3$ Maass cusp forms, but we restrict ourselves to Hecke eigenforms to lighten the notation. 

\begin{theorem}\label{Long-Omega}
Let $x\in[1,\infty[$ be sufficiently large and let $k$ prime so that $k\ll x^{1/3-\delta}$ for any sufficiently small fixed $\delta>0$. Then
\begin{align*}
\max_{h\in\mathbb Z_k^\times}\left|\sump_{m\leqslant x}A(m,1)\,e\!\left(\frac{mh}k\right)\right|=\Omega\left(k^{1/2}x^{1/3}\right),
\end{align*}
where the maximum is taken over all reduced residue classes modulo $k$.
\end{theorem}
\noindent Notice that $\ll_\varepsilon k^{1/2}x^{1/3+\varepsilon}$ is the conjectured upper bound in (\ref{conjecture}) for the range $k\ll x^{1/3}$. Thus the result of Theorem \ref{Long-Omega} is essentially optimal in the sense of conjecture (\ref{conjecture}). 

This follows immediately from the following mean square result.
\begin{theorem}\label{long_mean_square} (see Theorem \ref{lower bounds}) Let $X\in[1,\infty[$ be sufficiently large and let $k$ be a prime so that $k\ll X^{1/3-\delta}$ for any sufficiently small fixed $\delta>0$. Then 
\begin{align*}
\sum_{h\in\mathbb Z_k^\times}\int\limits_X^{3X}\left|\sump_{m\leq x}A(m,1)e\left(\frac{mh}k\right)\right|^2 \mathrm d x\gg X^{5/3}k^2.
\end{align*}
\end{theorem}
\noindent As sketched in the introduction, the proof of the previous result relies on Voronoi identities for both $\widetilde A_1(x;h/k)$ and $\widetilde A_2(x;h/k)$, but in Section $9$ we shall also present an alternative proof for Theorem \ref{Long-Omega} that sidesteps the use of a Voronoi identity for $\widetilde A_1(x;h/k)$. 



   
The third main theorem establishes a non-trivial $\Omega$-result for short sums of certain lengths. 

\begin{theorem}\label{Short-Omega}
Let $x\in[1,\infty[$ be sufficiently large, let $k$ prime and suppose that $k^{3/2}x^{1/2+\delta}\ll\Delta\ll kx^{2/3-\delta}$ for any sufficiently small fixed $\delta>0$. Then
\begin{align*}
\max_{h\in\mathbb Z_k^\times}\left|\sump_{x\leqslant m\leqslant x+\Delta}A(m,1)\,e\!\left(\frac{mh}k\right)\right|=\Omega\left(\Delta x^{-1/3}k^{-1/2}\right).
\end{align*}
\end{theorem}

\noindent This is a direct consequence of the following mean square result.

\begin{theorem}\label{Short-mean-square} (see Theorem \ref{ShortA_0-lower-bound})
Let $X\in[1,\infty[$ be sufficiently large and let $k$ be a prime so that $k^{3/2}X^{1/2+\delta}\ll\Delta\ll kX^{2/3-\delta}$ for any sufficiently small fixed $\delta>0$. Then 
\begin{align*}
\sum_{h\in\mathbb Z_k^\times}\int\limits_X^{4X}\left|\sump_{x\leqslant m\leqslant x+\Delta} A(m,1)\,e\!\left(\frac{mh}k\right)\right|^2\,\mathrm d x\gg\Delta^2X^{1/3}.
\end{align*}
\end{theorem}    
   
\noindent Notice that the result of Theorem \ref{Short-Omega} is of the right order of magnitude in the sense of the conjectural bound (\ref{conjecture}), up to $X^\delta$, when $\Delta\asymp kX^{2/3-\delta}$. However, for smaller $\Delta$ this $\Omega$-result is farther away from the expected upper bound. 

   
The restriction to prime denominators is made mainly for simplicity so we could handle the sum over the divisors of $k$ arising from the Voronoi identities for Riesz weighted exponential sums efficiently. Since we are interested in $\Omega$-results this is not such a serious restriction. Note also that theorems above hold in the case $k=1$, which can be easily seen by making simple cosmetic modifications to the proofs. 

Likewise, we are unable to obtain $\Omega$-results for an arbitrary numerator $h$ coprime to $k$. The averaging over the reduced residue classes modulo $k$ has to be included in order to be able to exract main terms when evaluating moments of the sums $\widetilde A_a(x;h/k)$. It seems that this cannot be achieved by our methods unless we perform such extra averaging. The presence of this is not so serious as we are mainly interested in limitations for obtaining bounds (\ref{conjecture}). We also point out that the primed sums can be replaced with ordinary sums if one assumes that $\vartheta<1/6$ (which is expected to be true, but currently we only know that $\vartheta\leq 5/14$).

Finally, as a by-product of our analysis we improve the best known upper bound for rationally additively twisted sums with a small denominator. 
\begin{theorem}\label{improved-upper-bound}
Let $x\in[1,\infty[$, and let $h$ and $k$ be coprime integers so that $1\leqslant k\ll x^{1/3}$. Then 
\begin{align*}
\sum_{m\leqslant x}A(m,1)\,e\!\left(\frac{mh}k\right)\ll_\eps k^{3/4}x^{1/2+\vartheta/2+\varepsilon}.
\end{align*}
\end{theorem}

\noindent Notice that this theorem does not require the assumption for the denominator $k$ being a prime. Theorem \ref{improved-upper-bound} improves the best previously known upper bound \cite[Corollary 3]{Jaasaari--Vesalainen1}, which gives for instance an upper bound $\ll_\varepsilon k^{3/4}x^{1/2+\vartheta/2+\varepsilon}+k^{9/8+3\vartheta/4}x^{1/4+3\vartheta^2/2+3\vartheta/4+\varepsilon}$ in the range $1\leq k\ll x^{1/3}$ under the assumption $\vartheta\leq 1/6$. 

This paper is organised as follows. In Sections $4$ and $6$ we gather basic facts concerning rationally additively twisted $L$-functions and linear exponential sums on $\mathrm{GL}_3$, respectively. In Section $5$ we introduce Meijer $G$-functions relevant for the present work and study their asymptotic behaviour. Then we proceed to derive Voronoi identities for the Riesz weighted sums $\widetilde A_a(x;h/k)$ in the next two Sections $7$ (for $a\geq 2)$ and $8$ (for $a=1$). In Section $9$ we evaluate the mean square of $\widetilde A_2(x;h/k)$ and give the first proof for Theorem \ref{Long-Omega}. In the next section the mean square of $\widetilde A_1(x;h/k)$ is evaluated and Theorem \ref{long_mean_square} (as well as Theorem \ref{Long-Omega}) is deduced as a consequence. Theorem \ref{improved-upper-bound} is proved in Section $11$. The final two sections are devoted to the proofs of Theorems \ref{Short-Omega} and \ref{Short-mean-square}. 

\section{Notation}

\noindent The letter $\varepsilon$ denotes a positive real number, whose value can be fixed to be arbitrarily small, and whose value can be different in different instances in a proof.  All implicit constants are allowed to depend on $\varepsilon$, on the implicit constants appearing in the assumptions of theorem statements, and on anything that has been fixed. When necessary, we will use subscripts $\ll_{\alpha,\beta,...},O_{\alpha,\beta,...}$, etc. to indicate when implicit constants are allowed to depend on quantities $\alpha,\beta,...$
  
As usual, complex variables are written in the form $s=\sigma+it$ with $\sigma$ and $t$ real, and we write $e(x)$ for $e^{2\pi ix}$. The subscript in the integral $\int_{(\sigma)}$ means that we integrate over the vertical line $\Re(s)=\sigma$. When $h\in\mathbb Z$ and $k\in\mathbb Z_+$ are coprime, then $\overline h$ denotes an integer such that $h\overline h\equiv 1\,(\text{mod }k)$. We also write $\langle \cdot\rangle$ for $(1+|\cdot|^2)^{1/2}$. The notation $\sum'_{m\leqslant x}$ means the sum up to $x$ with the last term halved if $x$ is an integer. Furthermore, $\sump_{x\leq m\leq x+\Delta}=\sump_{m\leq x+\Delta}-\sump_{m\leq x-1}$. The notation $d(n)$ denotes the ordinary divisor function. Let us furthermore write
\begin{align*}
1_{a\equiv b\,(\ell)}:=\begin{cases}
1&\text{if }a\equiv b\pmod\ell\\
0&\text{otherwise}
\end{cases}
\end{align*}
Finally, for a positive integer $k$ let us define the following averaging operator. Given a function $f\colon\mathbb Z_k^\times\longrightarrow\mathbb C$, we set 
\begin{align*}
\mathop{\text{\LARGE$\E$}}_{x\in\mathbb Z_k^\times}
f(x):=\frac1{\varphi(k)}\sum_{x\in\mathbb Z_k^\times}f(x),
\end{align*}
where $\varphi$ is Euler's totient function, and $\mathbb Z_k^\times$ denotes the reduces residue classes modulo $k$, or a set of representatives of them.

\section{Rationally additively twisted $L$-functions of $\mathrm{GL}_3$ Maass cusp forms}
    
\noindent We shall derive our Voronoi identities from the additively twisted $L$-function of the Maass cusp form in question. To this end, we quote some of its properties from Section 3 of \cite{Goldfeld--Li1}. 
We fix a Hecke--Maass cusp form for $\mathrm{SL}_3(\mathbb Z)$ with Fourier coefficients $A(m_1,m_2)$. Also, let us choose coprime integers $h$ and $k$ with $k$ positive, as well as an index $j\in\left\{0,1\right\}$. Then we have the rationally additively twisted version of the corresponding Godement--Jacquet $L$-function given by
\[L_j\!\left(s+j,\frac hk\right):=\sum_{m=1}^\infty\frac{A(m,1)}{m^s}\left(e\!\left(\frac{mh}k\right)\vphantom{\Bigg|}+\left(-1\right)^je\!\left(-\frac{mh}k\right)\right)\]
for $s\in\mathbb C$ with $\Re(s)>1$. This has an entire analytic extension and satisfies the functional equation
\[L_j\!\left(s+j,\frac hk\right)=i^{-j}\,k^{-3s+1}\,\pi^{3s-3/2}\,G_j(s+j)\,\widetilde L_j\!\left(1-s-j,\frac{\overline h}k\right).\]
The $\Gamma$-factors are clumped together into the factor $G_j(s+j)$ given by
\[G_j(s+j):=\frac{\displaystyle{\Gamma\!\left(\frac{1-s+j+\alpha}2\right)\Gamma\!\left(\frac{1-s+j+\beta}2\right)\Gamma\!\left(\frac{1-s+j+\gamma}2\right)}}{\displaystyle{\Gamma\!\left(\frac{s+j-\alpha}2\right)\Gamma\!\left(\frac{s+j-\beta}2\right)\Gamma\!\left(\frac{s+j-\gamma}2\right)}},\]
where in turn the complex constants $\alpha$, $\beta$ and $\gamma$ are the Langlands parameters of the underlying Maass cusp form. We know from \cite[Props.\ 6.3.1 \& 12.1.9]{Goldfeld} and \cite[Eq.\ (3.30)]{Goldfeld--Li1} that
\[\alpha+\beta+\gamma=0,\qquad\text{and}\qquad\max\left\{\left|\Re(\alpha)\right|,\left|\Re(\beta)\right|,\left|\Re(\gamma)\right|\right\}\leqslant\frac12.\]
Finally, the Dirichlet series on the right-hand side of the functional equation is given by
\[\widetilde L_j\!\left(1-s-j,\frac{\overline h}k\right)
:=\sum_{d\mid k}\sum_{m=1}^\infty\frac{A(d,m)}{d^{1-2s}\,m^{1-s}}
\left(S\!\left(\overline h,m;\frac kd\right)\vphantom{\Bigg|}+\left(-1\right)^jS\!\left(\overline h,-m;\frac kd\right)\right)\]
for $\Re(s)<0$, and by an entire analytic continuation elsewhere. Here $S(a,b;c)$ is the usual Kloosterman sum. 

From Stirling's formula, the functional equation above, and the Phragm\'en--Lindel\"of principle, we obtain the convexity estimate
\[L_j\!\left(s+j,\frac hk\right)\ll_\eps k^{3(1+\delta-\sigma)(1+\varepsilon)/2}\,\langle t\rangle^{3(1+\delta-\sigma)/2}\]
in the strip $-\delta\leqslant\sigma\leqslant 1+\delta$ (see \cite[p. 262]{Jaasaari--Vesalainen1}), where $\delta\in\mathbb R_+$ is arbitrary and fixed. Furthermore, it follows from this that, for each $j\in\left\{0,1\right\}$ and every $a\in\mathbb Z_+\cup\left\{0\right\}$, as well as for any coprime integers $h$ and $k$ with $k$ positive,
\[L_j\!\left(-a+j,\frac hk\right)\ll_{a,\eps} k^{3\left(1+2a\right)/2+\varepsilon}.\]
In particular,
\begin{align}\label{L-bounds}
L_j\!\left(0+j,\frac hk\right)\ll_\eps k^{3/2+\varepsilon},\quad
L_j\!\left(-1+j,\frac hk\right)\ll_\eps k^{9/2+\varepsilon},\quad
L_j\!\left(-2+j,\frac hk\right)\ll_\eps k^{15/2+\varepsilon},
\end{align}
and
\[L_j\!\left(-3+j,\frac hk\right)\ll_\eps k^{21/2+\varepsilon}.\]

\noindent When summing by parts we will repeatedly use the following consequences of the Rankin--Selberg theory saying that
\begin{align}\label{Average-estimates}
\sum_{m\leqslant x}\left|A(d,m)\right|\ll_\eps d^{\vartheta+\varepsilon}x
\qquad\text{and}\qquad
\sum_{m\leqslant x}\left|A(d,m)\right|^2\ll_\eps d^{2\vartheta+\varepsilon}x,
\end{align}
uniformly in $d\in\mathbb Z_+$ and $x\in\left[1,\infty\right[$ (see e.g.\ \cite{Jaasaari--Vesalainen3}).

\section{On certain Meijer $G$-functions and their asymptotics}\label{meijer-section}

\noindent Applying Perron's formula to the additively twisted $L$-function, shifting the line of integration to the left, and then applying the additively twisted functional equation leads to integrals involving quotients of gamma functions. We could derive the asymptotics we need in the same way as similar integrals were treated in earlier works, e.g \cite{Ivic, Miller, Li, Ernvall-Hytonen--Jaasaari--Vesalainen, Jaasaari--Vesalainen1}. However, here we prefer to follow the approach of \cite{Czarnecki}, which observed that these types of integrals are actually Meijer $G$-functions, and thus one obtains the relevant asymptotics directly from the well-known asymptotic properties of the latter. Our main references for Meijer $G$-functions are \cite{Luke1, Luke2}.

Let $\alpha$, $\beta$ and $\gamma$ be the complex numbers as in the functional equation of the additively twisted $L$-function, and let $y\in\mathbb Z_+$, $a\in\mathbb Z_+$ and $j\in\left\{0,1\right\}$.
We define $\mathscr J_{a,j}(y)$ to be the specific Meijer $G$-function of interest to us as
\begin{align*}
&\mathscr J_{a,j}(y)\\
&:=G^{5,0}_{2,8}\!\left(y;\begin{array}{l}\displaystyle{1,\frac12}\\[2mm]
\displaystyle{\frac{1+j+\alpha}2,\frac{1+j+\beta}2,
\frac{1+j+\gamma}2,-\frac a2,\frac12-\frac a2,1-\frac{j-\alpha}2,1-\frac{j-\beta}2,1-\frac{j-\gamma}2}\end{array}\right)\\[2mm]
&=\frac1{2\pi i}\int\limits_{\mathscr C}
Q(s)\,y^s\,\mathrm ds,
\end{align*}
where $Q(s)$ denotes the $\Gamma$-quotient
\[Q(s):=\frac{\displaystyle{\Gamma\!\left(\frac{1+j+\alpha}2-s\right)\Gamma\!\left(\frac{1+j+\beta}2-s\right)\Gamma\!\left(\frac{1+j+\gamma}2-s\right)\Gamma\!\left(-\frac a2-s\right)\Gamma\!\left(\frac12-\frac a2-s\right)}}{\displaystyle{\Gamma\!\left(s+\frac{j-\alpha}2\right)\Gamma\!\left(s+\frac{j-\beta}2\right)\Gamma\!\left(s+\frac{j-\gamma}2\right)\Gamma\!\left(1-s\right)\Gamma\!\left(\frac12-s\right)}}.\]
The contour of integration $\mathscr C$ can be, say, a simple polygonal chain, which begins at $\sigma_0-i\infty$ with a vertical half-line with abscissa $\sigma_0$ satisfying $\sigma_0>1/4-a/6$, ends at $\sigma_0+i\infty$ with another vertical half-line with the same abscissa $\sigma_0$, and having the property that all the poles of the $\Gamma$-factors in the numerator lie to the right of $\mathscr C$ \cite[Subsect. 5.3.1]{Luke2}. For definiteness, we may select $\mathscr C$ to be the contour $\mathscr C(\sigma_0,\sigma_1,\Lambda)$ consisting of line segments connecting the points $\sigma_0-i\infty$, $\sigma_0-i\Lambda$, $\sigma_1-i\Lambda$, $\sigma_1+i\Lambda$, $\sigma_0+i\Lambda$ and $\sigma_0+i\infty$, in this order, where the second abscissa $\sigma_1$ is a real number satisfying $\sigma_1<-a/2$ and $\Lambda$ is a positive real number such that $\Lambda$ is larger than the imaginary parts of $\alpha/2$, $\beta/2$ and $\gamma/2$.

We will not repeat the above parameters and instead write the above expression as $G^{5,0}_{2,8}(y;\ldots)$ with the understanding that the parameters are exactly as above. The asymptotic we wish to use is a special case of the case (4) of Theorem 2 in Section 5.10 of \cite{Luke1}. In our special case the result reads as follows.
\begin{lemma}\label{j-asymptotics}
Let $y\in\mathbb R_+$ be larger than some arbitrary fixed positive real constant, and let $j\in\left\{0,1\right\}$ and $a\in\mathbb Z_+$ be fixed. Then the above function $\mathscr J_{a,j}(y)$ has the asymptotic expansion
\[\mathscr J_{a,j}(y)
\sim-\frac{1}{2\,\sqrt{3\pi}}\sum_\pm\exp\!\left(\pm6i\,y^{1/6}\pm\frac{\pi i}2\left(a+j\right)\right)y^{\left(1-a\right)/6}\sum_{\ell=0}^\infty M_\ell\,i^{\pm\ell}\,y^{-\ell/6},\]
for some complex coefficients $M_0$, $M_1$, $M_2$, \dots\ for which $M_0=1$, and which otherwise depend on $\alpha$, $\beta$, $\gamma$, $a$, and $j$.
In particular,
\[
\mathscr J_{a,j}(y)
=-\frac1{\sqrt{3\pi}}\,y^{(1-a)/6}\,\cos\!\left(6\,y^{1/6}+\frac\pi2\left(a+j\right)\right)
+O(y^{-a/6}).
\]
\end{lemma}

\begin{proof}
We apply the asymptotics (4) of Theorem 2 in Section 5.10 of \cite{Luke1} which says that the Meijer $G$-function $G^{5,0}_{2,8}(y;\ldots)$ has the asymptotic expansion
\[G^{5,0}_{2,8}(y;\ldots)\sim
A^{5,0}_8\,H_{2,8}(y\,e^{3\pi i})+\smash{\overline A}\vphantom A^{\,5,0}_8\,H_{2,8}(y\,e^{-3\pi i}),\]
as $y\longrightarrow\infty$. In the notation of Section 5.7 in \cite{Luke1}, we have
\[\sigma=8-2=6\qquad\text{and}\qquad\nu=8-5-0=3.\]
We observe that our parameters satisfy the conditions (A) and (B) of Section 5.7 in \cite{Luke1} vacuously as in our case $n=0$, and they satisfy the condition (C) simply because $1-1/2=1/2\not\in\mathbb Z$.
In the notation of Theorem 5 in Section 5.7 of \cite{Luke1}, we have
\begin{align*}
\Xi_1 &=\frac{1+j+\alpha}2+\frac{1+j+\beta}2+\frac{1+j+\gamma}2-\frac a2+\frac12-\frac a2+1-\frac{j-\alpha}2+1-\frac{j-\beta}2+1-\frac{j-\gamma}2\\
&=5-a,
\end{align*}
and
\[\Lambda_1=1+\frac12=\frac32,\]
so that
\[\theta=\frac1\sigma\left(\frac12\left(1-\sigma\right)+\Xi_1-\Lambda_1\right)
=\frac{1-a}6.\]
The asymptotic behaviour of $H_{2,8}$ is thus given by
\[H_{2,8}(y\,e^{\pm3\pi i})\sim i^{\pm\left(1-a\right)}\,\frac{\left(2\pi\right)^{5/2}}{\sqrt6}\,\exp(\mp6i\,y^{1/6})\,y^{(1-a)/6}\sum_{\ell=0}^\infty M_\ell\,i^{\mp\ell}\,y^{-\ell/6},\]
as $y\longrightarrow\infty$,
for some complex constants $M_0$, $M_1$, $M_2,\dots$. In particular, $M_0=1$.
Finally, in the notation of Subsection 5.9.2 in \cite{Luke1},
\[A^{5,0}_8=\left(-1\right)^\nu\left(2\pi i\right)^{-\nu}\exp\!\left(i\pi\vphantom{\Bigg|}\left(0-1+\frac{j-\alpha}2-1+\frac{j-\beta}2-1+\frac{j-\gamma}2\right)\right)
=\left(2\pi\right)^{-3}i^{1-j},\]
and similarly,
\[\smash{\overline A}\vphantom A^{\,5,0}_8
=\left(2\pi\right)^{-3}i^{j-1}.\qedhere\]
\end{proof}

\noindent The Meijer $G$-function appears from the integral in Perron's formula. The following lemma connects the two.
\begin{lemma}\label{perron-meijer}
Let $a\in\mathbb Z_+$, let $y\in\mathbb R_+$, and let $\sigma_0$ and $\sigma_1$ be real numbers such that $\sigma_0>1/4-a/6$ and $\sigma_1<-a/2$. Then
\[\frac1{2\pi i}\int\limits_{\mathscr C(2\sigma_0,2\sigma_1,2\Lambda)}\frac{G_j(s+j)\,y^s\,\mathrm ds}{s\left(s+1\right)\cdots\left(s+a\right)}
=-\left(-2\right)^{-a}\mathscr J_{a,j}(y^2).\]
\end{lemma}

\begin{proof}
Observing 
\[\frac1{s\left(s+1\right)\cdots\left(s+a\right)}
=\frac{\displaystyle{\left(-\frac12\right)^{a+1}\Gamma\!\left(-\frac{s\vphantom1}2-\frac a2\right)\Gamma\!\left(-\frac{s\vphantom1}2-\frac a2+\frac12\right)}}{\displaystyle{\Gamma\!\left(-\frac{s\vphantom1}2+1\right)\Gamma\!\left(-\frac s2+\frac12\right)}},\]
we see that
\begin{align*}
\frac1{2\pi i}\int\limits_{\mathscr C(2\sigma_0,2\sigma_1,2\Lambda)}\frac{G_j(s+j)\,y^s\,\mathrm ds}{s\left(s+1\right)\cdots\left(s+a\right)}
&=-\left(-2\right)^{-a}\frac1{2\pi i}\int\limits_{\mathscr C(2\sigma_0,2\sigma_1,2\Lambda)}Q\!\left(\frac s2\right)\left(y^2\right)^{s/2}\frac{\mathrm ds}2\\
&=-\left(-2\right)^{-a}\mathscr J_{a,j}(y^2),
\end{align*}
as desired.
\end{proof}

\noindent In order to derive Voronoi identities with Riesz weights in the $a=1$ case, we wish to differentiate the formula of case $a=2$ termwise. The following lemma allows this.
\begin{lemma}\label{derivative-of-j}
Let $a\in\mathbb Z_+$, let $j\in\left\{0,1\right\}$, and let $y\in\mathbb R_+$. Then
\[\frac{\mathrm d}{\mathrm dy}\left(y^{a+1}\mathscr J_{a+1,j}(y^2)\right)
=-2\,y^a\mathscr J_{a,j}(y^2).\]
\end{lemma}
  
\begin{proof}
Let $\sigma_0$ and $\sigma_1$ be real numbers such that $\sigma_0>1/4-a/6$ and $\sigma_1<-a/2-1/2$.
Using Lemma \ref{perron-meijer} twice we may compute
\begin{align*}
&\frac{\mathrm d}{\mathrm dy}\left(y^{a+1}\mathscr J_{a+1,j}(y^2)\right)
=\frac{\mathrm d}{\mathrm dy}\left(-\left(-2\right)^{a+1}\frac1{2\pi i}\int\limits_{\mathscr C(2\sigma_0,2\sigma_1,2\Lambda)}\frac{G_j(s+j)\,y^{s+a+1}\,\mathrm ds}{s\left(s+1\right)\cdots\left(s+a\right)\left(s+a+1\right)}\right)\\
&\qquad=\left(-2\right)\cdot\left(-1\right)\left(-2\right)^{a}
\frac1{2\pi i}\int\limits_{\mathscr C(2\sigma_0,2\sigma_1,2\Lambda)}\frac{G_j(s+j)\,y^{s+a}\,\mathrm ds}{s\left(s+1\right)\cdots\left(s+a\right)}
=-2\,y^a\mathscr J_{a,j}(y^2),
\end{align*}
which completes the proof. 
\end{proof}

\section{Some known estimates for long linear exponential sums}

\noindent We shall need estimates for long linear exponential sums involving Fourier coefficients of a $\mathrm{GL}_3$ Maass cusp form. The following is Theorem 1.1 in \cite{Miller}.
\begin{lemma}\label{millers-bound}
Let $x\in\left[1,\infty\right[$. Then
\[\sum_{m\leqslant x}A(m,1)\,e(m\alpha)\ll_\eps x^{3/4+\varepsilon}\]
uniformly in $\alpha\in\mathbb R$.
\end{lemma}
\noindent On the other side of our Voronoi identities coefficients $A(d,m)$, where $d$ divides the denominator of the twist, will also appear. Fortunately, it is easy to extend Miller's estimate for these coefficients with uniformity over $d$.
\begin{lemma}\label{miller-improved}
Let $x\in\left[1,\infty\right[$, $d\in\mathbb Z_+$ and $\alpha\in\mathbb R$. Then
\[\sum_{m\leqslant x}A(m,d)\,e(m\alpha)\ll_\eps d^{\vartheta+\varepsilon}\,x^{3/4+\varepsilon},\]
uniformly in both $\alpha$ and $d$.
\end{lemma}

\begin{proof}
As the underlying form is a normalised Hecke eigenform we have
\[A(m,d)=\sum_{\ell\mid\left(d,m\right)}\mu\!\left(\ell\right)A\!\left(1,\frac d\ell\right)A\!\left(\frac{m\vphantom d}\ell,1\right)\]
for any $m,d\in\mathbb Z_+$, where $\mu$ is the M\"obius function. Using this, we may estimate
\begin{align*}
\sum_{m\leqslant x}A(m,d)\,e(m\alpha)
&=\sum_{m\leqslant x}\sum_{\ell\mid\left(d,m\right)}\mu\!\left(\ell\right)A\!\left(1,\frac d\ell\right)A\!\left(\frac{m\vphantom d}\ell,1\right)e(m\alpha)\\
&=\sum_{\ell\mid d}\mu\!\left(\ell\right)A\!\left(1,\frac d\ell\right)
\sum_{m\leqslant x/\ell}A(m,1)\,e\!\left(m\ell\alpha\right)\\
&\ll_\varepsilon\sum_{\ell\mid d}\left(\frac d\ell\right)^{\vartheta+\varepsilon}\left(\frac{x\vphantom d}\ell\right)^{3/4+\varepsilon}
\ll_\eps d^{\vartheta+\varepsilon}\,x^{3/4+\varepsilon},
\end{align*}
completing the proof. 
\end{proof}

\noindent When $\alpha$ is a fraction with a small denominator, we have more precise estimates, which depend on $\vartheta$. The following is Corollary 3 in \cite{Jaasaari--Vesalainen1}.
\begin{lemma}\label{corollary3}
Let $x\in\left[1,\infty\right[$, and let $h$ and $k$ be coprime integers such that $1\leqslant k\ll x^{2/3}$. Then
\[\sum_{m\leqslant x}A(m,1)\,e\!\left(\frac{mh}k\right)\ll_\eps k^{1/2+\varepsilon}\,x^{2/3}+k\,x^{1/3+\vartheta+\varepsilon}.\]
If in addition $\vartheta\leqslant1/3$ and $k\ll x^{2/3-2\vartheta}$, then
\[\sum_{m\leqslant x}A(m,1)\,e\!\left(\frac{mh}k\right)\ll_\eps k^{3/4}\,x^{1/2+\vartheta/2+\varepsilon}+k^{9/8+3\vartheta/4}\,x^{1/4+3\vartheta^2/2+3\vartheta/4+\varepsilon}.\]
In particular, if $\vartheta=0$ and $k\ll x^{2/3}$, then
\[\sum_{m\leqslant x}A(m,1)\,e\!\left(\frac{mh}k\right)\ll_\eps k^{3/4}\,x^{1/2+\varepsilon}.\]
\end{lemma}

\section{Additively twisted Voronoi identities for Riesz means: the case $a\geqslant2$}

\noindent In this section  we study the modified Riesz weighted sums
\begin{align*}
\widetilde A_{a,j}\!\left(x;\frac hk\right)&:=\frac1{a!}\sump_{m\leqslant x}A(m,1)\left(e\!\left(\frac{mh}k\right)\vphantom{\Bigg|}+\left(-1\right)^je\!\left(-\frac{mh}k\right)\right)\left(x-m\right)^a\\
&\qquad\qquad\qquad-\sum_{0\leqslant\nu\leqslant a}\frac{\left(-1\right)^\nu x^{a-\nu}}{\nu!\left(a-\nu\right)!}\,L_j\!\left(-\nu+j,\frac hk\right),
\end{align*}
where $x\in\left[0,\infty\right[$, $j\in\left\{0,1\right\}$, $a\in\mathbb Z_+\cup\left\{0\right\}$, and $h$ and $k$ are coprime integers with $k$ positive. Recall the key relation (see (\ref{Key-property}) above)
\begin{align*}
\int\limits_x^t \widetilde A_{a,j}\left(u;\frac hk\right)\,\mathrm d u=\widetilde A_{a+1,j}\left(t;\frac hk\right)-\widetilde A_{a+1,j}\left(x;\frac hk\right)
\end{align*}
for the Riesz means that we shall use repeatedly. Our goal in this section is to derive Voronoi-type identities for these sums when $a\geqslant2$.

The classical Perron's formula, which is e.g.\ Theorem 1.4.4 in \cite{Brudern}, has the following generalisation for Riesz means (see e.g.\ Chapter 5 in \cite{Montgomery--Vaughan}), which is the starting point for proving our Voronoi-type identities. 
\begin{lemma}\label{perron-formula-for-riesz-means}
Let $\sigma\in\mathbb R_+$, and let $c\colon\mathbb Z_+\longrightarrow\mathbb C$ be a sequence such that the Dirichlet series $\sum_{n=1}^\infty c(n)/n^\sigma$ converges absolutely. Then, for $x\in\mathbb R_+$, and for a non-negative integer $a$, we have
\[\frac1{a!}\sump_{n\leqslant x}\,c(n)\left(x-n\right)^a
=\frac1{2\pi i}\int\limits_{(\sigma)}\left(\sum_{n=1}^\infty\frac{c(n)}{n^s}\right)\frac{x^{s+a}\,\mathrm ds}{s\left(s+1\right)\cdots\left(s+a\right)},\]
where the integration is over the vertical line where the real part is $\sigma$. When $a=0$, the integral should be understood as the limit of $\int_{\sigma-iT}^{\sigma+iT}$ as $T\longrightarrow\infty$.
\end{lemma}

\noindent The main result of this section is the following proposition. 

\begin{proposition}\label{riesz-identity-with-large-exponent}
Let $x\in\mathbb R_+$, $j\in\left\{0,1\right\}$, $a\in\left\{2,3,\ldots\right\}$, and let $h$ and $k$ be coprime integers with $k$ positive. Then
\begin{multline*}
\widetilde A_{a,j}\!\left(x;\frac hk\right)
=
i^{-j}\,\pi^{-3/2}\,k\,x^a\left(-1\right)^{a+1}2^{-a}\\\times\sum_{d\mid k}\sum_{m=1}^\infty\frac{A(d,m)}{dm}\left(S\!\left(\overline h,m;\frac kd\right)\vphantom{\Bigg|}+\left(-1\right)^jS\!\left(\overline h,-m;\frac kd\right)\right)
\mathscr J_{a,j}\!\left(\frac{\pi^6\,d^4\,m^2\,x^2}{k^6}\right).
\end{multline*}
\end{proposition}

\begin{proof}
Let us first fix some $\delta\in\left]0,1/6\right[$. Perron's formula for Riesz means tells us that
\begin{align*}
&\frac1{a!}\sump_{m\leqslant x}A(m,1)\left(e\!\left(\frac{mh}k\right)\vphantom{\Bigg|}+\left(-1\right)^je\!\left(-\frac{mh}k\right)\right)\left(x-m\right)^a \\
&=\frac1{2\pi i}\int\limits_{\left(1+\delta\right)}L_j\!\left(s+j,\frac hk\right)\frac{x^{s+a}\,\mathrm ds}{s\left(s+1\right)\cdots\left(s+a\right)}.
\end{align*}
Let $\mathscr L$ be the polygonal chain $\mathscr C(-\delta,-a-\delta,2\Lambda)$ connecting the points $-\delta-i\infty$, $-\delta-2i\Lambda$, $-a-\delta-2i\Lambda$, $-a-\delta+2i\Lambda$, $-\delta+2i\Lambda$ and $-\delta+i\infty$, in this order, where $\Lambda$ is a fixed positive real number larger than the imaginary parts of $\alpha/2$, $\beta/2$ and $\gamma/2$. We shift the contour of integration from the vertical line $\Re(s)=1+\delta$ to $\mathscr L$. This is justified by the convexity bound for $L_j(s+j,h/k)$ and the assumption $a\geqslant 2$. This leads to residue terms from the simple poles at the points $-a$, $-a+1$, \dots, $0$. For each $\nu\in\left\{0,1,\ldots,a\right\}$ we get a residue term
\[\frac{\left(-1\right)^\nu\,x^{a-\nu}}{\nu!\left(a-\nu\right)!}\,L_j\!\left(-\nu+j,\frac hk\right).\]
Thus, we obtain
\[\widetilde A_{a,j}\!\left(x;\frac hk\right)
=\frac1{2\pi i}\int\limits_{\mathscr L}L_j\!\left(s+j,\frac hk\right)\frac{x^{s+a}\,\mathrm ds}{s\left(s+1\right)\cdots\left(s+a\right)}.\]
Applying the additively twisted functional equation gives
\[\widetilde A_{a,j}\!\left(x;\frac hk\right)
=\frac1{2\pi i}\int\limits_{\mathscr L}i^{-j}\,k^{-3s+1}\,\pi^{3s-3/2}\,G_j(s+j)\,\widetilde L_j\!\left(1-s-j;\frac{\overline h}k\right)\,\frac{x^{s+a}\,\mathrm ds}{s\left(s+1\right)\cdots\left(s+a\right)}.\]
Writing the $L$-function as a Dirichlet series and switching the order of summation and integration yields
\begin{multline*}  
\widetilde A_{a,j}\!\left(x;\frac hk\right)
=i^{-j}\,k\,\pi^{-3/2}\,x^a\sum_{d\mid k}\sum_{m=1}^\infty\frac{A(d,m)}{dm}\left(S\!\left(\overline h,m;\frac kd\right)\vphantom{\Bigg|}+\left(-1\right)^jS\!\left(\overline h,-m;\frac kd\right)\right)\\
\times\frac1{2\pi i}\int\limits_{\mathscr L}
G_j(s+j)\left(\frac{\pi^3\,d^2\,mx}{k^3}\right)^s\frac{\mathrm ds}{s\left(s+1\right)\cdots\left(s+a\right)}.
\end{multline*}
Invoking Lemma \ref{perron-meijer} finishes the proof.
\end{proof}
       
\noindent Let us then study the Riesz means of order $a\in\mathbb Z_+\cup\{0\}$ defined as
\begin{align}\label{main-term}
\widetilde A_a\!\left(x;\frac hk\right)&:=\frac12\left(\widetilde A_{a,0}\!\left(x;\frac hk\right)\vphantom{\Bigg|}+\widetilde A_{a,1}\!\left(x;\frac hk\right)\right) \nonumber \\
&=\frac1{a!}\sump_{m\leqslant x}A(m,1)\,e\!\left(\frac{mh}k\right)(x-m)^a \nonumber \\
&\qquad\qquad\qquad-\frac12\sum_{0\leqslant\nu\leqslant a}\frac{\left(-1\right)^\nu x^{a-\nu}}{\nu!\left(a-\nu\right)!}\left(L_0\!\left(-\nu,\frac hk\right)\vphantom{\Bigg|}+L_1\!\left(-\nu+1,\frac hk\right)\right),
\end{align}
where $x\in\left[0,\infty\right[$, and $h$ and $k$ are coprime integers with $k$ positive.
We observe at once that these sums have the pleasant property that
\begin{align*}
\int\limits_0^x\widetilde A_a\!\left(t;\frac hk\right)\mathrm dt
&=\widetilde A_{a+1}\!\left(x;\frac hk\right)-\widetilde A_{a+1}\!\left(0;\frac hk\right)\\
&=\widetilde A_{a+1}\!\left(x;\frac hk\right)+\frac{\left(-1\right)^{a+1}}{2\cdot\left(a+1\right)!}\left(L_0\!\left(-\left(a+1\right);\frac hk\right)\vphantom{\Bigg|}+L_1\!\left(-\left(a+1\right)+1;\frac hk\right)\right).
\end{align*}
This allows us to later relate Riesz means of different orders, and in particular to turn our better control of higher order Riesz means to information about lower order Riesz means. Of course, similar identities hold for $\widetilde A_{a,j}$ with $j\in\left\{0,1\right\}$.

The previous proposition, together with Lemma \ref{j-asymptotics}, immediately yields the following corollary using Weil's bound\footnote{Here $(a,b,c)$ denotes the greatest common divisor of integers $a$, $b$, and $c$.} $|S(a,b;k)|\leq d(k)k^{1/2}(a,b,k)^{1/2}$.
\begin{corollary}\label{approximate-voronoi-for-large-a}
Let $x\in\left[1,\infty\right[$, let $a\geqslant 2$ be an integer, and let $h$ and $k$ be coprime integers with $1\leqslant k\ll x^{1/3}$. Then 
\begin{multline*}
\widetilde A_a\!\left(x;\frac hk\right)\\
=\frac{\left(-1\right)^ak^a\,x^{(2a+1)/3}}{(2\pi)^{a+1}\sqrt3}\sum_{d|k}\frac1{d^{(2a+1)/3}}\sum_{m=1}^\infty\frac{A(d,m)}{m^{(a+2)/3}}\sum_\pm i^{\pm a}\,S\!\left(\overline h,\pm m;\frac kd\right)e\!\left(\pm\frac{3\,d^{2/3}\,m^{1/3}\,x^{1/3}}{k}\right)\\
+O\left(k^{(2a+3)/2}\,d(k)\,x^{2a/3}\right).
\end{multline*}
\end{corollary}

\noindent We also record here the following upper bound. 
\begin{corollary}\label{a2bound}
Let $x\in\left[1,\infty\right[$, $j\in\left\{0,1\right\}$, $a\in\left\{2,3,\ldots\right\}$, and let $h$ and $k$ be coprime integers with $1\leqslant k\ll x^{1/3}$. Then
\[\widetilde A_{a,j}\!\left(x;\frac hk\right)\ll k^{(2a+1)/2}\,d(k)\,x^{(2a+1)/3}.\]
In particular,
\[\widetilde A_{2,j}\!\left(x;\frac hk\right)\ll k^{5/2}\,d(k)\,x^{5/3}
\qquad\text{and}\qquad
\widetilde A_{3,j}\!\left(x;\frac hk\right)\ll k^{7/2}\,d(k)\,x^{7/3}.\]
\end{corollary}

\begin{proof}
Estimating the infinite series in Proposition \ref{riesz-identity-with-large-exponent} by absolute values using the asymptotics from Lemma \ref{j-asymptotics} and Weil's bound gives
\begin{align*}
\widetilde A_{a,j}\!\left(x;\frac hk\right)
&\ll k\,x^a\sum_{d\mid k}\sum_{m=1}^\infty\frac{\left|A(d,m)\right|}{dm}\left(\frac kd\right)^{1/2}d(k)\left(\frac{d^{2/3}\,m^{1/3}\,x^{1/3}}k\right)^{1-a}\\
&\ll k^{a+1/2}\,d(k)\,x^{2a/3+1/3}\sum_{d\mid k}d^{\vartheta-5/6-2a/3}
\ll k^{a+1/2}\,d(k)\,x^{2a/3+1/3},
\end{align*}
as required. 
\end{proof}

\section{Additively twisted Voronoi identities for Riesz means: the case $a=1$}

\noindent In the process of deriving a truncated Voronoi identity for $\widetilde A_1(x;h/k)$ we need a reasonable upper bound for $\widetilde A_{1,j}(x;h/k)$. After the Voronoi formula is derived, we will use it to deduce a better upper bound, which can then be fed back to the argument used to derive the Voronoi formula in the first place, leading to an improved error term in said formula. 
\begin{lemma}\label{simple-a1bound}
Let $x\in\left[1,\infty\right[$, let $j\in\left\{0,1\right\}$, and let $h$ and $k$ be coprime integers so that $1\leqslant k\ll x^{1/3}$. Then
\[\widetilde A_{1,j}\!\left(x;\frac hk\right)\ll_\eps k^{5/4}\,x^{29/24+\varepsilon}.\]
\end{lemma}

\begin{proof}
We choose $H=k^{5/4}\,x^{11/24}$. Then $1\leqslant H\ll x$, and using (\ref{Key-relation}), Lemma \ref{millers-bound}, and Corollary \ref{a2bound} gives
\begin{align*}
\widetilde A_{1,j}\!\left(x;\frac hk\right)
&=\frac1H\int\limits_x^{x+H}\widetilde A_{1,j}\!\left(t;\frac hk\right)\mathrm dt
-\frac1H\int\limits_x^{x+H}\int\limits_x^t\widetilde A_{0,j}\!\left(u;\frac hk\right)\mathrm du\,\mathrm dt\\
&=\frac1H\left(\widetilde A_{2,j}\!\left(x+H;\frac hk\right)-\vphantom{\Bigg|}\widetilde A_{2,j}\!\left(x;\frac hk\right)\right)-\frac1H\int\limits_x^{x+H}\int\limits_x^t\widetilde A_{0,j}\!\left(u;\frac hk\right)\mathrm du\,\mathrm dt\\
&\ll_\eps\frac1H\,k^{5/2}d(k)\,x^{5/3}+H\,x^{3/4+\varepsilon}\ll_\eps k^{5/4}\,x^{29/24+\varepsilon},
\end{align*}
as desired. 
\end{proof}

\noindent The proof of the case $a=1$ gives rise to sums of the kind treated in the above lemma, but with coefficients $A(m,d)$ instead of $A(m,1)$. The following lemma deals with this minor complication.

\begin{lemma}\label{simple-a1bound2}
Let $x\in\left[1,\infty\right[$, let $j\in\left\{0,1\right\}$, let $h$ and $k$ be coprime integers satisfying $1\leqslant k\ll x^{1/3}$, and let $d\in\mathbb Z_+$ be so that $d|k$. Then
\[\sum_{m\leqslant x}A(m,d)\left(e\!\left(\frac{mhd}k\right)\vphantom{\Bigg|}+\left(-1\right)^je\!\left(-\frac{mhd}k\right)\right)\left(x-m\right)\ll_\varepsilon d^{\vartheta+\varepsilon}\,k^{5/4}\,x^{29/24+\varepsilon}.\]
\end{lemma}

\begin{proof}
We may compute using M\"obius inversion, Lemma \ref{simple-a1bound}, and (\ref{L-bounds}),
\begin{align*}
&\sum_{m\leqslant x}A(m,d)\left(e\!\left(\frac{mhd}k\right)\vphantom{\Bigg|}+\left(-1\right)^je\!\left(-\frac{mhd}k\right)\right)\left(x-m\right)\\
&=\sum_{m\leqslant x}\sum_{\ell\mid\left(d,m\right)}\mu\!\left(\ell\right)A\!\left(1,\frac d\ell\right)A\!\left(\frac{m\vphantom d}\ell,1\right)\left(e\!\left(\frac{mhd}k\right)\vphantom{\Bigg|}+\left(-1\right)^je\!\left(-\frac{mhd}k\right)\right)\left(x-m\right)\\
&=\sum_{\ell\mid d}\mu\!\left(\ell\right)A\!\left(1,\frac d\ell\right)\ell\sum_{m\leqslant x/\ell}A\!\left(m,1\right)\left(e\!\left(\frac{\ell mhd}{k}\right)\vphantom{\Bigg|}+\left(-1\right)^je\!\left(-\frac{\ell mhd}{k}\right)\right)\left(\frac x\ell-m\right)\\
&=\sum_{\ell\mid d}\mu\!\left(\ell\right)A\!\left(1,\frac d\ell\right)\ell\left(\widetilde A_{1,j}\!\left(\frac x\ell;\frac{\ell h d}{k}\right)\vphantom{\Bigg|}+O_\eps\!\left(\left(\frac k{\ell}\right)^{3/2+\varepsilon}\cdot\frac x\ell\right)+O_\eps\!\left(\left(\frac k{\ell}\right)^{9/2+\varepsilon}\right)\right)\\
&\ll_\eps\sum_{\ell\mid d}
\left(\frac d\ell\right)^{\vartheta+\varepsilon}\ell
\left(
k^{5/4}\left(\frac{x\vphantom k}\ell\right)^{29/24+\varepsilon}
+\left(\frac k\ell\right)^{3/2+\varepsilon}\cdot\frac x\ell
+\left(\frac{k}\ell\right)^{9/2+\varepsilon}
\right),
\end{align*}
where we have used the fact that since $k\ll x^{1/3}$, also $k/\ell\ll\left(x/\ell\right)^{1/3}$, so that Lemma \ref{simple-a1bound} is applicable. Finally, since $k\ll x^{1/3}$, we have
\[k^{3/2+\varepsilon}\,x+k^{9/2+\varepsilon}\ll_\eps k^{5/4}\,x^{29/24+\varepsilon},\]
and the total contribution is $\ll_\eps d^{\vartheta+\varepsilon}\,k^{5/4}\,x^{29/24+\varepsilon}$.
\end{proof}

\noindent We will need to handle various sums involving both Fourier coefficients and Kloosterman sums. Taking discrete Fourier transforms of the Kloosterman sums allows us to reduce things back to exponential sums weighted by just the Fourier coefficients.
\begin{lemma}\label{kloosterman-fourier}
Let $h\in\mathbb Z$ and $k\in\mathbb Z_+$ be coprime, and define
\[\widehat S(h,\xi;k):=\frac1k\sum_{\ell=1}^kS(h,\ell;k)\,e\!\left(-\frac{\ell\xi}k\right)\]
for every $\xi\in\mathbb Z$. Then, for any $m,\xi\in\mathbb Z$, we have
\[S(h,m;k)=\sum_{\xi=1}^k\widehat S(h,\xi;k)\,e\!\left(\frac{m\xi}k\right),\]
and
\[\widehat S(h,\xi;k)\ll k^{1/2}\,d(k),\]
as well as
\[\sum_{\xi=1}^k\bigl|\widehat S(h,\xi;k)\bigr|\ll k\,d(k).\]
\end{lemma}

\begin{proof}
The first claim follows directly from the discrete Fourier inversion formula, and the second follows directly from Weil's bound. The third follows from the Cauchy--Schwarz inequality, the discrete Parseval identity and Weil's bound since
\begin{align*}
\sum_{\xi=1}^k\bigl|\widehat S(h,\xi;k)\bigr|
&\leqslant k^{1/2}\sqrt{\sum_{\xi=1}^k\bigl|\widehat S(h,\xi;k)\bigr|^2}\\
&=k^{1/2}\sqrt{\frac1k\sum_{\ell=1}^k\bigl|S(h,\ell;k)\bigr|^2}\\
&\ll k^{1/2}\,\sqrt{\frac1k\cdot k\left(k^{1/2}\,d(k)\right)^2}\ll k\,d(k). \qedhere\end{align*}
\end{proof}

\noindent The argument used in the proof of Proposition \ref{riesz-identity-with-large-exponent} runs into difficulties in the case $a=1$ due to convergence issues, and so we need to proceed differently. Our next goal is to establish the case $a=1$ essentially by differentiating the formula for $a=2$ termwise. The requisite convergence considerations give a practical truncated version of the summation formula as a by-product.
\begin{proposition}\label{riesz-identity-a1}
Let $x\in\left[1,\infty\right[$, $j\in\left\{0,1\right\}$, and let $h$ and $k$ be positive integers with $1\leqslant k\ll x^{1/3}$. Then
\begin{align*}   
&\widetilde A_{1,j}\!\left(x;\frac hk\right)
=i^{-j}\,2^{-1}\,\pi^{-3/2}\,kx\\
&\quad\times\sum_{d\mid k}\sum_{m=1}^\infty\frac{A(d,m)}{dm}\left(S\!\left(\overline h,m;\frac kd\right)\vphantom{\Bigg|}+\left(-1\right)^jS\!\left(\overline h,-m;\frac kd\right)\right)\mathscr J_{1,j}\!\left(\frac{\pi^6\,d^4\,m^2\,x^2}{k^6}\right).
\end{align*}
Here the infinite series converges both boundedly and uniformly when $x$ is restricted to a compact interval in $\mathbb R_+$.  Furthermore, if $N\in\left[1,\infty\right[$, and $N\gg k^{3}$, then
\begin{align*}
&\widetilde A_{1,j}\!\left(x;\frac hk\right)
=i^{-j}\,2^{-1}\,\pi^{-3/2}\,kx\\&\times\sum_{d\mid k}\sum_{m\leqslant N}\frac{A(d,m)}{dm}\left(S\!\left(\overline h,m;\frac kd\right)\vphantom{\Bigg|}+\left(-1\right)^jS\!\left(\overline h,-m;\frac kd\right)\right)\mathscr J_{1,j}\!\left(\frac{\pi^6\,d^4\,m^2\,x^2}{k^6}\right)+\text{error},
\end{align*}
where the error is
\[\ll_\eps k^{5/4+\eps}\,x^{5/3}\,N^{\varepsilon-1/8}.\]
\end{proposition}

\begin{proof}
Our starting point is the formula
\begin{align*}
&\widetilde A_{2,j}\!\left(x;\frac hk\right)
=-i^{-j}\,2^{-2}\,\pi^{-3/2}\,k\,x^2\\&\times
\sum_{d\mid k}\sum_{m=1}^\infty\frac{A(d,m)}{dm}\left(S\!\left(\overline h,m;\frac kd\right)\vphantom{\Bigg|}+\left(-1\right)^jS\!\left(\overline h,-m;\frac kd\right)\right)\mathscr J_{2,j}\!\left(\frac{\pi^6\,d^4\,m^2\,x^2}{k^6}\right),
\end{align*}
and our plan is to differentiate this identity. The derivative of the left-hand side is of course simply $\widetilde A_{1,j}(x;h/k)$. On the other hand, Lemma \ref{derivative-of-j} tells us that
\[\frac{\mathrm d}{\mathrm dx}\left(x^2\vphantom{\Bigg|}\mathscr J_{2,j}\!\left(\frac{\pi^6\,d^4\,m^2\,x^2}{k^6}\right)\right)=-2x\mathscr J_{1,j}\!\left(\frac{\pi^6\,d^4\,m^2\,x^2}{k^6}\right),\]
and using this it is straightforward to check that, formally, the derivative of the right-hand side coincides with the right-hand side of the desired identity for $\widetilde A_{1,j}(x;h/k)$. From the analysis of the previous section it is clear that the above series for $\widetilde A_{2,j}(x;h/k)$ converges for any $x$. Thus, by standard arguments, it only remains to prove that the series for $\widetilde A_{1,j}(x;h/k)$ converges boundedly and uniformly when $x$ is restricted to a bounded interval.

We take arbitrary numbers $a,b\in\left[1,\infty\right[$ satisfying $a<b$ and $a\gg k^{3}$, and consider the partial sum
\begin{align*}
&\Sigma(a,b)\\
&:=kx\sum_{d\mid k}\sum_{a\leqslant m\leqslant b}\frac{A(d,m)}{dm}\left(S\!\left(\overline h,m;\frac kd\right)+\vphantom{\Bigg|}\left(-1\right)^jS\!\left(\overline h,-m;\frac kd\right)\right)\mathscr J_{1,j}\!\left(\frac{\pi^6\,d^4\,m^2\,x^2}{k^6}\right).
\end{align*}
By Lemma \ref{j-asymptotics}, we have          
\begin{align*}
&\mathscr J_{1,j}\!\left(\frac{\pi^6\,d^4\,m^2\,x^2}{k^6}\right)
=-\frac{1}{2\sqrt{3\pi}}\sum_\pm\exp\!\left(\pm\frac{6i\pi\,d^{2/3}\,m^{1/3}\,x^{1/3}}k\pm\frac{\pi i}2\left(1+j\right)\right)\\
&\qquad\qquad\qquad\qquad\qquad\qquad\qquad+O\!\left(\left(d^{2/3}\,m^{1/3}\,x^{1/3}\,k^{-1}\right)^{-1}\right).
\end{align*}
We may estimate the contribution of the $O$-term in $\Sigma(a,b)$ by
\begin{align*}
\ll_\eps kx\sum_{d\mid k}\sum_{a\leqslant m\leqslant b}\frac{\left|A(d,m)\right|}{dm}\left(\frac kd\right)^{1/2+\varepsilon}\left(d^{2/3}\,m^{1/3}\,x^{1/3}\,k^{-1}\right)^{-1}
\ll_\eps k^{5/2+\varepsilon}\,x^{2/3}\,a^{-1/3}.
\end{align*}
We also write in each term of $\Sigma(a,b)$
\[S\!\left(\overline h,m;\frac kd\right)+\left(-1\right)^jS\!\left(\overline h,-m;\frac kd\right)
=\sum_{\xi=1}^{k/d}\widehat S\!\left(\overline h,\xi;\frac kd\right)\left(e\!\left(\frac{m\xi d}k\right)\vphantom{\Bigg|}+\left(-1\right)^je\!\left(-\frac{m\xi d}k\right)\right).\]
Combining all these observations gives
\begin{align*}
&\Sigma(a,b)\\
&=-\frac{1}{2\sqrt{3\pi}}\sum_\pm e\!\left(\pm\frac{\pi i }2\left(1+j\right)\right)\sum_{d\mid k}\frac1d\sum_{\xi=1}^{k/d}\widehat S\!\left(\overline h,\xi;\frac kd\right)
kx\,\Sigma_\pm(a,b)+O(k^{5/2+\varepsilon}\,x^{2/3}\,a^{-1/3}),
\end{align*}
where
\[
\Sigma_\pm(a,b):=\sum_{a\leqslant m\leqslant b}\frac{A(d,m)}m\left(e\!\left(\frac{m\xi d}k\right)\vphantom{\Bigg|}+\left(-1\right)^je\!\left(-\frac{m\xi d}k\right)\right) e\left(\pm\frac{3\,d^{2/3}\,m^{1/3}\,x^{1/3}}k\right).
\]
To simplify the formulae below, we shall write, for each $\ell\in\left\{0,1\right\}$, and all $x\in\left[1,\infty\right[$,
\[A_{\ell,j,d}\!\left(x;\frac \xi k\right):=\sum_{m\leqslant x}A(d,m)\left(e\!\left(\frac{m\xi d}k\right)\vphantom{\Bigg|}+\left(-1\right)^je\!\left(-\frac{m\xi d}k\right)\right)\left(x-m\right)^\ell.\]
We start the estimation of $\Sigma_\pm(a,b)$ by summing by parts getting
\begin{align*}
\Sigma_\pm(a,b)
&=\left.\frac1t\,e\!\left(\pm\frac{3\,d^{2/3}\,t^{1/3}\,x^{1/3}}k\right)A_{0,j,d}\!\left(t;\frac\xi k\right)\right]_a^{t=b}\\
&\quad-\int\limits_a^bA_{0,j,d}\!\left(t;\frac\xi k\right)\left(-\frac1{t^2}\pm\frac{2\pi i\,d^{2/3}\,x^{1/3}}{k\,t^{5/3}}\right)e\!\left(\pm\frac{3\,d^{2/3}\,t^{1/3}\,x^{1/3}}k\right)\mathrm dt.
\end{align*}
By Lemma \ref{miller-improved}, the substitution term is $\ll_\eps d^{\vartheta+\varepsilon}\,a^{\varepsilon-1/4}$, as is the contribution from the term involving $1/t^2$, and it only remains to estimate the term involving $t^{-5/3}$. We do this by integration by parts to get
\begin{align*}
&\int\limits_a^bA_{0,j,d}\!\left(t;\frac\xi k\right)\frac{2\pi i\,d^{2/3}\,x^{1/3}}{k\,t^{5/3}}\,e\!\left(\pm\frac{3\,d^{2/3}\,t^{1/3}\,x^{1/3}}k\right)\mathrm dt\\
&=\left.A_{1,j,d}\!\left(t;\frac\xi k\right)\frac{2\pi i\,d^{2/3}\,x^{1/3}}{k\,t^{5/3}}\,e\!\left(\pm\frac{3\,d^{2/3}\,t^{1/3}\,x^{1/3}}k\right)\right]_a^{t=b}\\
&\qquad-\int\limits_a^bA_{1,j,d}\!\left(t;\frac\xi k\right)\left(-\frac{10\,\pi i\,d^{2/3}\,x^{1/3}}{3\,k\,t^{8/3}}\pm\frac{4\,\pi^2\,d^{4/3}\,x^{2/3}}{k^2\,t^{7/3}}\right)e\!\left(\pm\frac{d^{2/3}\,t^{1/3}\,x^{1/3}}k\right)\mathrm dt.
\end{align*}
By Lemma \ref{simple-a1bound2}, the substitution terms and the term involving $t^{-8/3}$ contribute
\[\ll_\eps d^{2/3+\vartheta+\varepsilon}\,x^{1/3}\,k^{1/4}\,a^{\varepsilon-11/24},\]
whereas the contribution from the term involving $t^{-7/3}$ is
\[\ll_\eps d^{4/3+\vartheta+\varepsilon}\,k^{-3/4}\,x^{2/3}\,a^{\varepsilon-1/8}.\]
Altogether, our estimate for $\Sigma_\pm(a,b)$ reads
\[\Sigma_\pm(a,b)\ll_\varepsilon d^{\vartheta+\varepsilon}\,a^{\varepsilon-1/4}+d^{2/3+\vartheta+\varepsilon}\,x^{1/3}\,k^{1/4}\,a^{\varepsilon-11/24}
+d^{4/3+\vartheta+\varepsilon}\,k^{-3/4}\,x^{2/3}\,a^{\varepsilon-1/8}.\]
Plugging this into our previous expression for $\Sigma(a,b)$ gives
\begin{align*}
\Sigma(a,b)&\ll_\varepsilon\sum_{d\mid k}\frac1d\sum_{\xi=1}^{k/d}\left|\widehat S\!\left(\overline h,\xi;\frac kd\right)\right|kx\\
&\qquad\qquad\times\left(d^{\vartheta+\varepsilon}\,a^{\varepsilon-1/4}+d^{2/3+\vartheta+\varepsilon}\,x^{1/3}\,k^{1/4}\,a^{\varepsilon-11/24}
+d^{4/3+\vartheta+\varepsilon}\,k^{-3/4}\,x^{2/3}\,a^{\varepsilon-1/8}\right)\\
&\qquad+k^{5/2+\varepsilon}\,x^{2/3}\,a^{-1/3}\\
&\ll_\varepsilon k^{2+\varepsilon}\,x\left(a^{\varepsilon-1/4}+x^{1/3}\,k^{1/4}\,a^{\varepsilon-11/24}
+k^{-3/4}\,x^{2/3}\,a^{\varepsilon-1/8}\right)\\
&\qquad+k^{5/2+\varepsilon}\,x^{2/3}\,a^{-1/3}\\
&\ll_\eps k^{2+\varepsilon}\,x\,a^{\varepsilon-1/4}
+k^{9/4+\varepsilon}\,x^{4/3}\,a^{\varepsilon-11/24}
+k^{5/4+\varepsilon}\,x^{5/3}\,a^{\varepsilon-1/8}
+k^{5/2+\varepsilon}\,x^{2/3}\,a^{-1/3},
\end{align*}
and we are done with the first assertion. For the truncated formula we observe that 
\[k^{2+\varepsilon}\,x\,a^{\varepsilon-1/4}
\ll_\eps k^{5/4+\varepsilon}\,x^{5/3}\,a^{\varepsilon-1/8},\]
and that
\[k^{5/2+\varepsilon}\,x^{2/3}\,a^{-1/3}\ll_\eps k^{5/4+\varepsilon}\,x^{5/3}\,a^{\varepsilon-1/8},\]
as well as
\[k^{9/4+\varepsilon}\,x^{4/3}\,a^{\varepsilon-11/24}\ll_\eps k^{5/4+\varepsilon}\,x^{5/3}\,a^{\varepsilon-1/8},\]
since $k\ll x^{1/3}$.
Thus we have obtained 
\[ 
\Sigma(a,b)\ll_\eps k^{5/4+\varepsilon}x^{5/3}a^{\eps-1/8},\]
which immediately gives the truncated identity. 
\end{proof}

\noindent From the above formula we get directly a better upper bound for $\widetilde A_{1,j}(x;h/k)$ compared to Lemma \ref{simple-a1bound}.
\begin{corollary}\label{better-a1bound}
Let $x\in\left[1,\infty\right[$, let $j\in\left\{0,1\right\}$, and let $h$ and $k$ be coprime integers with $1\leqslant k\ll x^{1/3}$. Then
\[\widetilde A_{1,j}\!\left(x;\frac hk\right)\ll_\varepsilon k^{3/2}\,x^{1+\varepsilon}.\]
Furthermore, if $d\in\mathbb Z_+$ and $d\mid k$, then also
\[\sum_{m\leqslant x}A\!\left(m,d\right)\left(e\!\left(\frac{mhd}k\right)\vphantom{\Bigg|}+\left(-1\right)^je\!\left(-\frac{mhd}k\right)\right)\left(x-m\right)\ll_\varepsilon d^{\vartheta+\varepsilon}\,k^{3/2}\,x^{1+\varepsilon}.\]
\end{corollary}

\begin{proof}
For the first estimate we employ the second identity of Proposition \ref{riesz-identity-a1} with the choice $N=\lfloor x^{16/3}k^{-2}\rfloor$, estimating all the terms by absolute values and using Lemma \ref{j-asymptotics}. The second estimate is proved in exactly the same way as Lemma \ref{simple-a1bound2}.
\end{proof}

\noindent We can now feed this back into the estimation of $\Sigma(a,b)$ in the proof of Proposition \ref{riesz-identity-a1}, where we used the weaker estimate from Lemma \ref{simple-a1bound2} (which in turn was based on Lemma \ref{simple-a1bound}). The resulting estimate for $\Sigma_\pm(a,b)$ is
\[\Sigma_\pm(a,b)\ll_\eps d^{\vartheta+\varepsilon}\,a^{\varepsilon-1/4}+d^{2/3+\vartheta+\varepsilon}\,k^{1/2}\,x^{1/3}\,a^{\varepsilon-2/3}+d^{4/3+\vartheta+\varepsilon}\,k^{-1/2}\,x^{2/3}\,a^{\varepsilon-1/3},\]
and the end result is
\begin{align*}
\Sigma(a,b)&\ll_\eps k^{2+\varepsilon}\,x\,a^{\varepsilon-1/4}+k^{5/2+\varepsilon}\,x^{4/3}\,a^{\varepsilon-2/3}+k^{3/2+\varepsilon}\,x^{5/3}\,a^{\varepsilon-1/3}+k^{5/2+\varepsilon}\,x^{2/3}\,a^{-1/3}\\
&\ll_\eps
 k^{2+\varepsilon}\,x\,a^{\varepsilon-1/4}
 +k^{3/2+\varepsilon}\,x^{5/3}\,a^{\varepsilon-1/3}.
\end{align*}
Thus we have obtained the following. 

\begin{corollary}\label{improved-A_1-voronoi}
Let $x\in\left[1,\infty\right[$, $j\in\left\{0,1\right\}$, let $h$ and $k$ be positive integers with $1\leqslant k\ll x^{1/3}$, and let $N\in\left[1,\infty\right[$ with $N\gg k^3$. Then
\begin{align*}
&\widetilde A_{1,j}\!\left(x;\frac hk\right)
=i^{-j}\,2^{-1}\,\pi^{-3/2}\,kx\\&\times\sum_{d\mid k}\sum_{m\leqslant N}\frac{A(d,m)}{dm}\left(S\!\left(\overline h,m;\frac kd\right)\vphantom{\Bigg|}+\left(-1\right)^jS\!\left(\overline h,-m;\frac kd\right)\right)\mathscr J_{1,j}\!\left(\frac{\pi^6\,d^4\,m^2\,x^2}{k^6}\right)+\text{error},
\end{align*}
where the error is
\[\ll_\eps k^{2+\varepsilon}\,x\,N^{\varepsilon-1/4}
 +k^{3/2+\varepsilon}\,x^{5/3}\,N^{\varepsilon-1/3}.
\]
\end{corollary}

\noindent Proposition \ref{riesz-identity-a1} and Corollary \ref{improved-A_1-voronoi} yield the following result when combined with averaging over $j\in\left\{0,1\right\}$ and Lemma~\ref{j-asymptotics}.
\begin{corollary}\label{sharpcutoff-voronoi}
Let $x\in\left[1,\infty\right[$, $j\in\left\{0,1\right\}$, and let $h$ and $k$ be positive integers with $1\leqslant k\ll x^{1/3}$. Then
\begin{align*}   
\widetilde A_{1}\!\left(x;\frac hk\right)
&=\frac{k\,x}{4\,\pi^2\,\sqrt3}\sum_{d\mid k}\frac1{d}\sum_{m=1}^\infty\frac{A(d,m)}{m}\sum_\pm \left(\mp i\right)\,S\!\left(\overline h,\pm m;\frac kd\right)e\!\left(\pm\frac{3\,d^{2/3}\,m^{1/3}\,x^{1/3}}{k}\right)\\
&\qquad\qquad\qquad+O_\eps\left(k^{5/2+\eps}\,x^{2/3}\right).
\end{align*}
Here the infinite series converges both boundedly and uniformly when $x$ is restricted to a compact interval in $\mathbb R_+$. Furthermore, if $N\in\left[1,\infty\right[$, and $N\ll k^{-3}\,x^3$, then
\begin{align*}
\widetilde A_{1}\!\left(x;\frac hk\right)
&=\frac{k\,x}{4\,\pi^2\,\sqrt3}\sum_{d\mid k}\frac1{d}\sum_{m\leqslant N}\frac{A(d,m)}{m}\sum_\pm \left(\mp i\right)\,S\!\left(\overline h,\pm m;\frac kd\right)e\!\left(\pm\frac{3\,d^{2/3}\,m^{1/3}\,x^{1/3}}{k}\right)\\
&\qquad\qquad\qquad+O_\eps\left(k^{3/2+\eps}\,x^{5/3+\eps}N^{\eps-1/3}\right).
\end{align*}
\end{corollary}    

\section{Pointwise $\Omega$-result for long sums on $\mathrm{GL}_3$}

\noindent Let us first gather some auxiliary results which will be useful in what follows. We start with results concerning Kloosterman sums.
\begin{lemma}\label{Kloosterman-correlation}
Let $k$ be a prime. Suppose also that $m$ and $n$ are integers. Then
\[\sum_{a\in\mathbb Z_k^\times}S(a,m;k)=\begin{cases}1&\text{if $m\not\equiv0\pmod k$}\\
-k+1&\text{if $m\equiv0\pmod k$}
\end{cases}\]
as well as
\begin{align*}
\Sigma(m,n,k):=\sum_{a\in\mathbb Z_k^\times} S(a,m;k)\,S(a,n;k)=\begin{cases}
k^2-k-1&\text{if $m\equiv n\not\equiv0\pmod k$}\\
k-1&\text{if $m\equiv n\equiv0\pmod k$}\\
-1&\text{if $m\equiv0$ and $n\not\equiv0\pmod k$}\\
-1&\text{if $m\not\equiv0$ and $n\equiv0\pmod k$}\\
-k-1&\text{if $0\not\equiv m\not\equiv n\not\equiv0\pmod k$}
\end{cases}    
\end{align*}
\end{lemma} 

\begin{proof}
The first cases of both identities are treated in Chapter 4 of \cite{Iwaniec}. The rest are then fairly straightforward to prove simply by expanding the Kloosterman sums in terms of their definitions and taking advantage of the resulting geometric series.
\end{proof}

\noindent We will use the first derivative test repeatedly and record one of its formulations here. Results such as this are discussed for instance in Section 5.1 of \cite{Huxley}.
\begin{lemma}\label{Lemma-Huxley}
Let $a,b\in\mathbb R$ with $a<b$, let $\lambda\in\mathbb R_+$, and let $f$ be a real-valued continuously differentiable function on $\left]a,b\right[$ such that $\left|f'(x)\right|\geqslant\lambda$ for $x\in\left]a,b\right[$. Also, let $g$ be a complex-valued continuously differentiable function on the interval $\left[a,b\right]$, and let $G\in\mathbb R_+$ be such that $g(x)\ll G$ for $x\in\left[a,b\right]$. Then
\begin{align*}
\int\limits_a^bg(x)\,e\!\left(f(x)\right)\,\mathrm dx\ll\frac G{\lambda}+\frac1{\lambda}\int\limits_a^b\left|g'(x)\right|\mathrm d x. 
\end{align*}
\end{lemma}

\noindent Now we are in a position to prove an asymptotic formula for the averaged mean square of the long second order Riesz sums.
\begin{theorem}\label{A_2-meansquare}
Suppose that $X\in\left[1,\infty\right[$ and that $k$ is a prime such that $k\ll X^{1/3-\delta}$ for any sufficiently small fixed $\delta>0$. Then we have
\begin{align*}
\mathop{\text{\LARGE$\E$}}_{h\in\mathbb Z_k^\times}\int\limits_X^{2X}\left|\widetilde A_2\left(x;\frac hk\right)\right|^2\,\mathrm dx=B(k)\,k^5\,X^{13/3}+O\left(k^{6}\,X^4\right),
\end{align*}
where $B(k)\in\mathbb R_+$ and $B(k)\asymp1$. 

In particular, if $x\in\left[1,\infty\right[$ and $k\in\mathbb Z_+$ is a prime so that $k\ll x^{1/3-\delta}$ for any sufficiently small fixed $\delta>0$, then we have
\[\max_{h\in\mathbb Z_k^\times}\left|\widetilde A_2\!\left(x;\frac hk\right)\right|=\Omega\left(k^{5/2}\,x^{5/3}\right).\]
\end{theorem}

\begin{proof}
Note that the $\Omega$-result follows immediately from the moment result. For the first statement, using Corollary \ref{approximate-voronoi-for-large-a} the second moment is
\begin{align*}
&\mathop{\text{\LARGE$\E$}}_{h\in\mathbb Z_k^\times}\int\limits_X^{2X}\left|\widetilde A_2\left(x;\frac hk\right)\right|^2\,\mathrm dx\\
&=\frac{k^4}{192\,\pi^6}\mathop{\text{\LARGE$\E$}}_{h\in\mathbb Z_k^\times}\int\limits_X^{2X}x^{10/3}\left|\sum_{d\mid k}\sum_{m=1}^\infty\frac{A(d,m)}{d^{5/3}\,m^{4/3}}\sum_\pm S\!\left(\overline h,\pm m;\frac kd\right)\,e\!\left(\pm\frac{3\,d^{2/3}\,m^{1/3}\,x^{1/3}}k\right)\right|^2\,\mathrm dx\\
&\qquad-\frac{k^2}{8\pi^3\sqrt 3}\,\Re\mathop{\text{\LARGE$\E$}}_{h\in\mathbb Z_k^\times}\int\limits_X^{2X}x^{5/3}\sum_{d\mid k}\sum_{m=1}^\infty\frac{A(d,m)}{d^{5/3}\,m^{4/3}}\\&\hspace*{4em}\times\sum_\pm S\!\left(\overline h,\pm m;\frac kd\right)e\!\left(\pm\frac{3\,d^{2/3}\,m^{1/3}\,x^{1/3}}k\right)
\cdot O(k^{7/2}\,x^{4/3})\,\mathrm dx+O(k^{7}\,X^{11/3}).
\end{align*}
Once we have proved that the first term on the right-hand side is $\asymp k^{5}\,X^{13/3}$, it immediately follows from the Cauchy--Schwarz inequality that the mixed term integral is
\[\ll\sqrt{k^{5}\,X^{13/3}}\,\sqrt{k^{7}\,X^{11/3}}\ll k^{6}\,X^{4},\] as required. Also, since $k\ll X^{1/3}$, we have $k^{7}\,X^{11/3}\ll k^{6}\,X^4$.

In the first integral, we expand $\left|\Sigma\right|^2$ as $\Sigma\,\overline\Sigma$. The diagonal terms, i.e.\ those terms where $d_1^2\,m_1=d_2^2\,m_2$ and where the signs in sums $\sum_\pm$ are chosen to be equal, give the contribution
\begin{align*}
&\frac{k^4}{192\pi^6}\smash{\sum_{d_1\mid k}\sum_{d_2\mid k}\underset{d_1^2m_1=d_2^2m_2}{\sum_{m_1=1}^\infty\sum_{m_2=1}^\infty}}
\frac{A(d_1,m_1)\,\overline{A(d_2,m_2)}}{d_1^{5/3}\,d_2^{5/3}\,m_1^{4/3}\,m_2^{4/3}}\\
&\qquad\times\mathop{\text{\LARGE$\E$}}_{h\in\mathbb Z_k^\times}\sum_\pm S\!\left(\overline h,\pm m_1;\frac k{d_1}\right)S\!\left(\overline h,\pm m_2;\frac k{d_2}\right)
\int\limits_X^{2X}x^{10/3}\,\mathrm dx.     
\end{align*}
We will examine the contribution of the different pairs $(d_1,d_2)$ separately at first. The terms where $d_1=d_2=1$ contribute
\begin{align}\label{contribution1}
\frac{k^4}{96\pi^6}\sum_{m=1}^\infty\frac{\left|A(1,m)\right|^2}{m^{8/3}}\left(\frac{k^2-k-1}{k-1}1_{m\not\equiv 0\,(k)}+1_{m\equiv 0\,(k)}\right)\int\limits_X^{2X}x^{10/3}\,\mathrm dx
\end{align}
by Lemma \ref{Kloosterman-correlation}.

Similarly, the terms with $d_1=d_2=k$ contribute
\begin{align}\label{contribution2}
\frac{k^4}{96\pi^6}\sum_{m=1}^\infty\frac{\left|A(k,m)\right|^2}{k^{10/3}m^{8/3}}\int\limits_X^{2X}x^{10/3}\,\mathrm d x
\end{align}
and the terms where $d_1d_2=k$ contribute
\begin{align}\label{contribution3}
-\frac{k^4}{48\pi^6}\Re\left(\sum_{m=1}^\infty\frac{A(1,k^2m)\overline{A(k,m)}}{k^{13/3}m^{8/3}}\right)\int\limits_X^{2X}x^{10/3}\,\mathrm d x. 
\end{align}
Our aim is to show that these three terms sum up to $B(k)k^5X^{13/3}$, where $B(k)\in\mathbb R$ with $B(k)\asymp 1$. The upper bound is straightforward to establish just by using the known estimates for the Fourier coefficients ($A(m,n)\ll_\eps (mn)^{\vartheta+\eps}$ and (\ref{Average-estimates})), but the lower bound requires some work. Towards this, observe first that by the arithmetic-geometric mean inequality, the absolute value of (\ref{contribution3}) is bounded from above by
\begin{align}\label{contribution4}
&\frac{k^4}{96\pi^6}\sum_{m=1}^\infty\frac{\left|A(1,k^2m)\right|^2}{k^{16/3}m^{8/3}}\int\limits_X^{2X}x^{10/3}\,\mathrm dx+\frac{k^4}{96\pi^6}\sum_{m=1}^\infty\frac{\left|A(k,m)\right|^2}{k^{10/3}m^{8/3}}\int\limits_X^{2X}x^{10/3}\,\mathrm d x \nonumber\\
&\leq\frac{k^4}{96\pi^6}\sum_{m=1}^\infty\frac{\left|A(1,m)\right|^2}{m^{8/3}}\int\limits_X^{2X}x^{10/3}\,\mathrm dx+\frac{k^4}{96\pi^6}\sum_{m=1}^\infty\frac{\left|A(k,m)\right|^2}{k^{10/3}m^{8/3}}\int\limits_X^{2X}x^{10/3}\,\mathrm d x.
\end{align}
Notice that the latter term in (\ref{contribution4}) is exactly (\ref{contribution2}). Thus, using the identity $1_{m\not\equiv 0(k)}=1-1_{m\equiv 0(k)}$, a simple computation shows that the sum of (\ref{contribution1}), (\ref{contribution2}), and (\ref{contribution3}) is bounded from below by 
\[\frac{k^4}{96\pi^6}\cdot\frac{k^2-2k}{k-1}\sum_{m=1}^\infty\frac{|A(1,m)|^2}{m^{8/3}}\left(1-1_{m\equiv 0\,(k)}\right)\int\limits_X^{2X}x^{10/3}\,\mathrm d x.\]
This is bounded from below by the required lower bound $\gg k^5X^{13/3}$ just by dropping all the terms except the one corresponding to $m=1$ (recall that $A(1,1)=1$) as each of the terms involved is non-negative.  

The contribution coming from the terms $d_1^2m=d_2^2n$, where the signs in $\sum_\pm$ are chosen to be distinct, can be estimated by using the first derivative test together with Weil's bound as
\begin{align*}
&\ll k^4\sum_\pm\sum_{d_1|k}\sum_{d_2|k}\sum_{m=1}^\infty\sum_{\substack{n=1\\d_1^2m=d_2^2n}}^\infty\frac{\left|A(d_1,m)\overline{A(d_2,n)}\right|}{d_1^{5/3}d_2^{5/3}m^{4/3}n^{4/3}}\left|S\left(\overline h,\pm m;\frac k{d_1}\right)S\left(\overline h,\mp n;\frac k{d_2}\right)\right|\int\limits_X^{2X}x^{10/3}\,e\!\left(\pm\frac{6(d_1^2mx)^{1/3}}k\right)\mathrm dx\\
&\ll k^4\cdot k\,X^{10/3}X^{2/3}k\sum_{d_1|k}\sum_{d_2|k}\sum_{m=1}^\infty\sum_{\substack{n=1\\d_1^2m=d_2^2n}}^\infty\frac{\left|A(d_1,m)\overline{A(d_2,n)}\right|}{d_1^{17/6}d_2^{13/6}m^{5/3}n^{4/3}}\\
&\ll X^{4}k^{6}.
\end{align*}    

For the off-diagonal terms we restrict to the case where the signs in $\sum_\pm$ are chosen to be the same as the case of distinct signs can be treated similarly with easier arguments. The prior terms contribute, using again the first derivative test and Weil's bound,
\begin{align*}
&\ll k^4\sum_\pm\sum_{d_1|k}\sum_{d_2|k}\sum_{m=1}^\infty\frac{\left|A(d_1,m)\right|}{m^{4/3}d_1^{5/3}}\sum_{\substack{n=1\\
d_1^2m\neq d_2^2n}}^\infty\frac{\left|A(d_2,n)\right|}{n^{4/3}d_2^{5/3}}\left|S\left(\overline h,\pm m;\frac k{d_1}\right)S\left(\overline h,\pm n;\frac k{d_2}\right)\right|\\
&\qquad\times\int\limits_X^{2X}x^{10/3}\,e\left(\frac{\pm 3(mx d_1^2)^{1/3}\pm 3(nxd_2^2)^{1/3}}k\right)\mathrm dx\\
&\ll k^4\sum_{d_1|k}\sum_{d_2|k}\sum_{m=1}^\infty\frac{\left|A(d_1,m)\right|}{m^{4/3}d_1^{5/3}}\sum_{\substack{n=1\\
d_1^2m\neq d_2^2n}}^\infty\frac{\left|A(d_2,n)\right|}{n^{4/3}d_2^{5/3}}\cdot\frac{X^{4}k}{\left|d_1^{2/3}m^{1/3}-d_2^{2/3}n^{1/3}\right|}\\
&\qquad\qquad\times\left|S\left(\overline h,\pm m;\frac k{d_1}\right)S\left(\overline h,\pm n;\frac k{d_2}\right)\right|\\
&\ll X^{4}k^6\sum_{d_1|k}\sum_{d_2|k}\sum_{n=1}^\infty\frac{\left|A(d_2,n)\right|}{n^{4/3}d_2^{5/3}}\sum_{\substack{m=1\\
d_1^2m\neq d_2^2n}}^\infty\frac{\left|A(d_1,m)\right|}{m^{4/3}d_1^{5/3}}\cdot\frac{1}{\left|d_1^{2/3}m^{1/3}-d_2^{2/3}n^{1/3}\right|}.
\end{align*}
By symmetry it is enough to treat the terms in the $m$-sum with $d_1^2m<d_2^2n$ (if this is not the case, one would consider the terms with $d_2^2n<d_1^2m$ in the $n$-sum). For these we split the $m$-sum into two parts according to whether $d_1^2m<d_2^2n/2$ or $d_2^2n/2\leq d_1^2m<d_2^2n$. Note that for $d_1^2m\ll d_2^2n$ we have 
\begin{align}\label{reciprocal-estimate}
\frac1{d_2^{5/3}}\cdot\frac1{|d_2^{2/3}n^{1/3}-d_1^{2/3}m^{1/3}|}\ll\frac1{d_2^{1/3}}\cdot\frac{n^{2/3}}{|d_2^2n-d_1^2m|}.
\end{align}
Using (\ref{reciprocal-estimate}) together with the pointwise bound $A(d_2,n)\ll_\varepsilon (d_2n)^{\vartheta+\varepsilon}$ the terms with $d_1^2m<d_2^2n/2$ contribute
\begin{align*}
&\ll X^{4}k^6\sum_{d_1|k}\sum_{d_2|k}\sum_{n=1}^\infty\frac{\left|A(d_2,n)\right|}{n^{4/3}d_2^{5/3}}\sum_{\substack{m=1\\
d_1^2m< d_2^2n/2}}^\infty\frac{\left|A(d_1,m)\right|}{m^{4/3}d_1^{5/3}}\cdot\frac{1}{\left|d_1^{2/3}m^{1/3}-d_2^{2/3}n^{1/3}\right|}\\
&\ll_\varepsilon X^4k^6\sum_{d_1|k}\sum_{d_2|k}\sum_{n=1}^\infty\sum_{\substack{m=1\\
d_1^2m<d_2^2n/2}}^\infty\frac{|A(d_1,m)|}{m^{4/3}}\cdot\frac{d_2^{\vartheta+\varepsilon}}{n^{2/3-\vartheta-\varepsilon}d_1^{5/3}d_2^{1/3}(d_2^2n-d_1^2m)}\\
&\ll_\varepsilon X^4k^6\sum_{d_1|k}\sum_{d_2|k}\sum_{n=1}^\infty\sum_{\substack{m=1\\
d_1^2m<d_2^2n/2}}^\infty\frac{|A(d_1,m)|}{m^{4/3}n^{5/3-\vartheta-\varepsilon}d_1^{5/3}d_2^{7/3-\vartheta-\varepsilon}}\\
&\ll X^4k^6,
\end{align*}
where the last estimate follows from (\ref{Average-estimates}) using partial summation, and the estimate $\vartheta\leq 5/14$. 

Recall that $\delta>0$ is sufficiently small, but fixed. For the terms with $d_2^2n/2\leq d_1^2m<d_2^2n$ we again use (\ref{reciprocal-estimate}) and the pointwise estimate $A(d,n)\ll (dn)^{\vartheta+\delta}$ to see that these terms contribute
\[
\ll X^4k^6\sum_{d_1|k}\sum_{d_2|k}\sum_{n=1}^\infty\frac{n^{\vartheta+2\delta}}{d_2^{1/3-\vartheta-\delta}n^{2/3+\delta}}\sum_{\substack{m=1\\ d_2^2n/2\leq d_1^2m< d_2^2n}}^\infty\frac{|A(d_1,m)|}{m^{4/3}d_1^{5/3}}\cdot\frac1{|d_2^2n-d_1^2m|}.\]
Observe that for $d_2^2n>d_1^2m$ we have
\[ 
n^{2/3}=n^{1/6}n^{1/2}\gg\left(\frac{d_1^2m}{d_2^2}\right)^{1/6}n^{1/2}.\]
Combining this with the estimate, which holds as $d_2^2n/2\leq d_1^2m$, 
\[ 
n^{\vartheta+2\delta}\leq\left(\frac{2d_1^2m}{d_2^2}\right)^{5/14+2\delta}\]
shows that the terms with $d_2^2n/2\leq d_1^2m<d_2^2n$ contribute
\begin{align}\label{intermediate-bound}
\ll X^4k^6\sum_{d_1|k}\sum_{d_2|k}\sum_{n=1}^\infty\sum_{\substack{m=1\\
d_2^2n/2\leq d_1^2m<d_2^2n}}^\infty\frac{|A(d_1,m)|}{m^{8/7-2\delta}d_1^{9/7-4\delta}n^{1/2+\delta}d_2^{5/7-\vartheta+3\delta}(d_2^2n-d_1^2m)}.
\end{align}
Next, we exchange the order of $m$- and $n$-sums. After that we apply the Cauchy--Schwarz inequality to the $n$-sum to see it is
\begin{align*}
&\sum_{\substack{n=1\\d_2^2n/2\leq d_1^2m<d_2^2n}}^\infty\frac1{n^{1/2+\delta}(d_2^2n-d_1^2m)}\\
&\leq\left(\sum_{\substack{n=1\\d_2^2n/2\leq d_1^2m<d_2^2n}}^\infty\frac1{n^{1+2\delta}}\right)^{1/2}\left(\sum_{\substack{n=1\\d_2^2n/2\leq d_1^2m<d_2^2n}}^\infty\frac1{(d_2^2n-d_1^2m)^2}\right)^{1/2}\ll 1,
\end{align*}
where the last estimate is uniform in $m,\,d_1,$ and $d_2$. Now using this in (\ref{intermediate-bound}), partial summation and (\ref{Average-estimates}) show that the terms with $d_1^2m\geq d_2^2n/2$ contribute $\ll X^4k^6$. This finishes the proof.  
\end{proof}

\noindent Next we present the first proof for Theorem \ref{Long-Omega}. The idea is to relate upper bounds of Riesz weighted sums with different orders. We first connect the sizes of $\widetilde A_1(x;h/k)$ and $\widetilde A_2(x;h/k)$. 
\begin{proposition}\label{A2toA1}
Let $k$ be a prime number. Let $\gamma\in\left[1,2\right]$ and $\eta\in\left[1,2\right]$ be such that $\gamma/2+\eta/6> 3/4$ and
\[\widetilde A_1\left(x;\frac hk\right)\ll x^\gamma k^\eta\]
for $k\ll x^{1/3-\delta}$ with $\delta>0$ sufficiently small and fixed. Then we have
\[\widetilde A_2\left(x;\frac hk\right)\ll x^{7/6+\gamma/2}k^{7/4+\eta/2}.\]
Furthermore, we have
\[\underset{h\in\mathbb Z_k^\times}{\max}\left|\widetilde A_1\left(x;\frac hk\right)\right|=\Omega\left(xk^{3/2}\right)\]
for any prime $k\ll x^{1/3-\delta}$.
\end{proposition}

\begin{proof}
For any $H>0$ we have by (\ref{Key-relation})
\begin{align}\label{A_2-identity}
\widetilde A_2\left(x;\frac hk\right)=\frac1H\int\limits_x^{x+H}\widetilde A_2\left(t;\frac hk\right)\,\mathrm dt
-\frac1H\int\limits_x^{x+H}\int\limits_x^t\widetilde A_1\left(u;\frac hk\right)\,\mathrm du\,\mathrm dt.
\end{align}
Substituting the Voronoi identity of Corollary \ref{approximate-voronoi-for-large-a} into the first term, we may estimate it by the first derivative test as
\begin{align}\label{first_der_bound}
\frac1H\int\limits_x^{x+H}\widetilde A_2\left(t;\frac hk\right)\,\mathrm dt\ll\frac1H\,x^{7/3}k^{7/2}+k^{7/2}x^{4/3}.
\end{align}
Now, if $\widetilde A_1(x,h/k)\ll x^\gamma k^\eta$, then the second term on the right-hand side of (\ref{A_2-identity}) is $\ll H\,x^\gamma\,k^\eta$, and choosing $H=x^{7/6-\gamma/2}k^{7/4-\eta/2}$ immediately gives
\[\widetilde A_2\left(x;\frac hk\right)\ll x^{7/6+\gamma/2}k^{7/4+\eta/2}\]
using (\ref{first_der_bound}) and the assumption $\gamma/2+\eta/6>3/4$.

To prove the second statement, note that from Theorem \ref{A_2-meansquare}, (\ref{first_der_bound}), and (\ref{A_2-identity}) it follows that
\begin{align*}
\Omega\left(x^{5/3}k^{5/2}\right)&=\max_{h\in\mathbb Z_k^\times}\left|\widetilde A_2\left(x;\frac hk\right)\right|\\
&\ll \frac1H\,x^{7/3}k^{7/2}+k^{7/2}x^{4/3}+H\max_{u\in[x,x+H]}\max_{h\in\mathbb Z_k^\times}\left|\widetilde A_1\left(u;\frac hk\right)\right|.
\end{align*}
Thus for any $H\leq x$ with $x^{2/3}k=o(H)$ we have 
\[
\max_{h\in\mathbb Z_k^\times}\left|\widetilde A_1\left(x;\frac hk\right)\right|=\Omega\left(H^{-1}x^{5/3}k^{5/2}\right),\]
from which the claim follows. 
\end{proof}
       
\begin{lemma}\label{A_1-on-avg}
Let $x\in\left[1,\infty\right[$, $H\in\left[1,x\right]$, and $k\ll x^{1/3-\delta}$ be a prime with $\delta>0$ sufficiently small and fixed. Then
\[\frac1H\int\limits_x^{x+H}\widetilde A_1\left(t;\frac hk\right)\,\mathrm dt\ll H^{-1}\,x^{5/3}k^{5/2}.\]
\end{lemma}

\begin{proof}  
We have
\[\frac1H\int\limits_x^{x+H}\widetilde A_1\left(t;\frac hk\right)\,\mathrm dt
=\frac1H\left(\widetilde A_2\left(x+H;\frac hk\right)-\widetilde A_2\left(x;\frac hk\right)\right),\]
and by Corollary \ref{a2bound}, this is $\ll H^{-1}\,x^{5/3}k^{5/2}$, where there is no $\eps$ in the exponent of $k$ as the denominator is a prime.
\end{proof}

\noindent Proposition \ref{A2toA1} together with the first part of the following result proves Theorem \ref{Long-Omega}, a result which is formulated as the latter part. 

\begin{theorem}\label{long-omega}
Let $\alpha\in\left]0,1\right]$ and $\beta\in]0,1]$ be such that $3\alpha+\beta>3/2$, $\alpha>\vartheta$, and 
\begin{align}\label{exp-sum-bound}
\sum_{m\leqslant x}A(m,1)e\left(\frac{mh}k\right)\ll x^\alpha k^\beta
\end{align}
for $x\in\left[1,\infty\right[$. Suppose that $k$ is a prime so that $k\ll x^{1/3-\delta}$ for any sufficiently small fixed $\delta>0$. Then 
\[\widetilde A_1\left(x;\frac hk\right)\ll x^{5/6+\alpha/2}k^{5/4+\beta/2}.\]
Furthermore, when $x\longrightarrow\infty$, we have
\[\underset{h\in\mathbb Z_k^\times}{\max}\left|\sump_{m\leqslant x}A(m,1)e\!\left(\frac{mh}k\right)\right|=\Omega\left(x^{1/3}k^{1/2}\right)\]
for any prime $k\ll x^{1/3-\delta}$.
\end{theorem}

\begin{proof}
By (\ref{Key-relation}) we have
\begin{align}\label{A1-identity}
\widetilde A_1\left(x;\frac hk\right)=\frac1H\int\limits_x^{x+H}\widetilde A_1\left(t;\frac hk\right)\,\mathrm dt
-\frac1H\int\limits_x^{x+H}\int\limits_x^t \widetilde A_0\left(u;\frac hk\right)\,\mathrm du\,\mathrm dt.
\end{align}
For the first term, we already know by Lemma \ref{A_1-on-avg} that
\[\frac1H\int\limits_x^{x+H}\widetilde A_1\left(t;\frac hk\right)\,\mathrm dt\ll H^{-1}\,x^{5/3}k^{5/2}.\]
On the other hand, recalling that (by (\ref{L-bounds}))
\begin{align}\label{Connecting_A_0_to_A}
\sum_{m\leqslant x}A(m,1)\,e\!\left(\frac{mh}k\right)=\widetilde A_0\!\left(x;\frac hk\right)+O(k^{3/2+\varepsilon})+O(x^{\vartheta+\varepsilon})
\end{align}
it follows from (\ref{exp-sum-bound}) and the assumptions on $\alpha$ and $\beta$ that we certainly have $\widetilde A_0(x;h/k)\ll x^\alpha k^\beta$ and so
\[\frac1H\int\limits_x^{x+H}\int\limits_x^t\widetilde A_0\left(u;\frac hk\right)\,\mathrm du\,\mathrm dt\ll H\,x^\alpha k^\beta.\]
Choosing $H=x^{5/6-\alpha/2}k^{5/4-\beta/2}$ gives
\[\widetilde A_1\left(x;\frac hk\right)\ll x^{5/6+\alpha/2}k^{5/4+\beta/2}.\]

\noindent To prove the second statement, note that from Proposition \ref{A2toA1} and (\ref{A1-identity}) it follows that
\[\Omega\left(xk^{3/2}\right)=\max_{h\in\mathbb Z_k^\times}\left|\widetilde A_1\left(x;\frac hk\right)\right|\ll \frac1H\,x^{5/3}k^{5/2}+H\max_{u\in[x,x+H]}\max_{h\in\mathbb Z_k^\times}\left|\widetilde A_0\left(u;\frac hk\right)\right|.\]
Thus for any $H\leq x$ with $x^{2/3}k=o(H)$ we have 
\[
\max_{h\in\mathbb Z_k^\times}\left|\widetilde A_0\left(x;\frac hk\right)\right|=\Omega\left(H^{-1}xk^{3/2}\right),\]
from which the claim follows as the residue term in $\widetilde A_0(x;h/k)$ has size $O(k^{3/2+\eps})$ by (\ref{L-bounds}).
\end{proof}

\begin{corollary}
The exponent pairs $(\alpha,\beta)=(1,0)$, $(\alpha,\beta)=(3/4+\varepsilon,0)$, $(\alpha,\beta)=(1/3+\varepsilon,1/2)$, and $(\alpha,\beta)=(1/2+\varepsilon,3/4)$, each of which either holds or holds under some assumptions, in the previous estimate lead to 
\[\widetilde A_1\left(x;\frac hk\right)\ll x^{4/3}k^{5/4},\quad
\ll_\eps x^{29/24+\varepsilon}k^{5/4},\quad
\ll_\eps x^{1+\varepsilon}k^{3/2},\quad\text{and}\quad
\ll_\eps x^{13/12+\varepsilon}k^{13/8},\]
respectively. Here the choice $(\alpha,\beta)=(1,0)$ is admissible by the Rankin--Selberg theory, the second one by Lemma \ref{millers-bound}, the third under the conjectural bounds (\ref{conjecture}), and the final possibility by Theorem \ref{improved-upper-bound} (assuming $\vartheta=0$).
\end{corollary}
    
\section{Second moments of long sums}

\noindent Using the truncated Voronoi identity for $\widetilde A_1(x;h/k)$ we deduce the following mean square result. 

\begin{theorem}\label{A1ms}
Let $x\in[1,\infty[$ and suppose that $k$ is a prime so that $k\ll X^{1/3-\delta}$ for any sufficiently small fixed $\delta>0$. Then we have
\begin{align*}
\mathop{\text{\LARGE$\E$}}_{h\in\mathbb Z_k^\times}\int\limits_X^{2X}\left|\widetilde A_1\left(x;\frac hk\right)\right|^2\,\mathrm dx
=C(k)\cdot X^3k^3+O_\eps\left(k^{4}\,X^{8/3+\eps}\right),
\end{align*}  
where $C(k)\in\mathbb R_+$ with $C(k)\asymp 1$.
\end{theorem}

\begin{proof}
We apply Corollary \ref{sharpcutoff-voronoi} with the choice $N= \lfloor k^{-3}X^3\rfloor$ to get
\begin{align*}
&\mathop{\text{\LARGE$\E$}}_{h\in\mathbb Z_k^\times}\int\limits_X^{2X}\left|\widetilde A_1\left(x;\frac hk\right)\right|^2\,\mathrm dx\\
&=\frac{k^2}{48\,\pi^4}\mathop{\text{\LARGE$\E$}}_{h\in\mathbb Z_k^\times}\int\limits_X^{2X}x^{2}\left|\sum_{d\mid k}\sum_{m\leq N}\frac{A(d,m)}{d\,m}\sum_\pm (\mp i)S\!\left(\overline h,\pm m;\frac kd\right)\,e\!\left(\pm\frac{3\,d^{2/3}\,m^{1/3}\,x^{1/3}}k\right)\right|^2\,\mathrm dx\\
&\qquad+\frac{k}{4\pi^2\sqrt 3}\,\Re\mathop{\text{\LARGE$\E$}}_{h\in\mathbb Z_k^\times}\int\limits_X^{2X}x\sum_{d\mid k}\sum_{m\leq N}\frac{A(d,m)}{d\,m}\\&\hspace*{3em}\times\sum_\pm (\mp i) S\!\left(\overline h,\pm m;\frac kd\right)e\!\left(\pm\frac{3\,d^{2/3}\,m^{1/3}\,x^{1/3}}k\right)
\cdot O(k^{5/2}\,x^{2/3})\,\mathrm dx+O_\eps(k^{5}\,X^{7/3+2\eps}).
\end{align*}
Once we have proved that the first term on the right-hand side is $\asymp k^{3}\,X^{3}$, it immediately follows from the Cauchy--Schwarz inequality that the mixed term integral is
\[\ll_\eps\sqrt{k^{3}\,X^{3}}\,\sqrt{k^{5}\,X^{7/3+2\eps}}\ll_\eps k^{4}\,X^{8/3+\eps}.\]
Also, as $k\ll X^{1/3-\delta}$, we have $k^{5}X^{7/3+2\eps}\ll_\eps k^{4}X^{8/3+\eps}$, as required.
  
In the first integral, we expand $\left|\Sigma\right|^2$ as $\Sigma\,\overline\Sigma$. The diagonal terms, i.e.\ those terms where $d_1^2\,m=d_2^2\,n$ and where the signs in sums $\sum_\pm$ are chosen to be equal, give the contribution
\begin{align*}
&\frac{k^2}{48\pi^4}\sum_{d_1\mid k}\sum_{d_2\mid k}\underset{d_1^2m=d_2^2n}{\sum_{m\leq N}\sum_{n\leq N}}
\frac{A(d_1,m)\,\overline{A(d_2,n)}}{d_1\,d_2\,m\,n}\mathop{\text{\LARGE$\E$}}_{h\in\mathbb Z_k^\times}\sum_\pm S\!\left(\overline h,\pm m;\frac k{d_1}\right)S\!\left(\overline h,\pm n;\frac k{d_2}\right)
\int\limits_X^{2X}x^{2}\,\mathrm dx.     
\end{align*}
We will examine the contribution of the different pairs $(d_1,d_2)$ separately at first. The terms where $d_1=d_2=1$ contribute
\begin{align}\label{contribution1-A_1}
\frac{k^2}{24\pi^4}\sum_{m\leq N}\frac{\left|A(1,m)\right|^2}{m^{2}}\left(\frac{k^2-k-1}{k-1}1_{m\not\equiv 0\,(k)}+1_{m\equiv 0\,(k)}\right)\int\limits_X^{2X}x^2\,\mathrm dx
\end{align}
by Lemma \ref{Kloosterman-correlation}.

Similarly, the terms with $d_1=d_2=k$ contribute
\begin{align}\label{contribution2-A_1}
\frac{k^2}{24\pi^4}\sum_{m\leq N}\frac{\left|A(k,m)\right|^2}{k^{2}m^{2}}\int\limits_X^{2X}x^{2}\,\mathrm d x
\end{align}
and the terms where $d_1d_2=k$ contribute
\begin{align}\label{contribution3-A_1}
-\frac{k^2}{12\pi^4}\Re\left(\sum_{m\leq N/k^2}\frac{A(1,k^2m)\overline{A(k,m)}}{k^{3}m^{2}}\right)\int\limits_X^{2X}x^{2}\,\mathrm d x. 
\end{align}
Our aim is to show that these three terms (\ref{contribution1-A_1}), (\ref{contribution2-A_1}), and (\ref{contribution3-A_1}), sum up to $C(k)k^3X^{3}$, where $C(k)\in\mathbb R$ with $C(k)\asymp 1$. The argument is similar as in the proof of Theorem \ref{A_2-meansquare}.

The upper bound of the right order of magnitude follows easily from the averaged estimates for the Fourier coefficients (\ref{Average-estimates}). For the lower bound, observe first that by using the arithmetic-geometric mean inequality the absolute value of (\ref{contribution3-A_1}) is bounded from above by
\begin{align}\label{contribution4-A_1}
&\frac{k^2}{24\pi^4}\sum_{m\leq N/k^2}\frac{\left|A(1,k^2m)\right|^2}{k^{4}m^{2}}\int\limits_X^{2X}x^{2}\,\mathrm dx+\frac{k^2}{24\pi^4}\sum_{m\leq N/k^2}\frac{\left|A(k,m)\right|^2}{k^{2}m^{2}}\int\limits_X^{2X}x^{2}\,\mathrm d x \nonumber\\
&\leq \frac{k^2}{24\pi^4}\sum_{m\leq N}\frac{\left|A(1,m)\right|^2}{m^{2}}\int\limits_X^{2X}x^{2}\,\mathrm dx+\frac{k^2}{24\pi^4}\sum_{m\leq N}\frac{\left|A(k,m)\right|^2}{k^{2}m^{2}}\int\limits_X^{2X}x^{2}\,\mathrm d x .
\end{align}
Notice that the latter term in (\ref{contribution4-A_1}) is exactly (\ref{contribution2-A_1}). Hence the sum of (\ref{contribution1-A_1}), (\ref{contribution2-A_1}), and (\ref{contribution3-A_1}) is bounded from below by 
\[ 
\frac{k^2}{24\pi^4}\cdot\frac{k^2-2k}{k-1}\sum_{m\leq N}\frac{|A(1,m)|^2}{m^2}\left(1-1_{m\equiv 0\,(k)}\right)\int\limits_X^{2X}x^2\,\mathrm d x.\]
This gives the lower bound $\gg k^3X^3$ as in the proof of Theorem \ref{A_2-meansquare}. 

The contribution arising from the terms $d_1^2m=d_2^2n$ where the signs in the sum $\sum_\pm$ are chosen to be distinct is by the first derivative test and Weil's bound given by
\begin{align*}
&\ll k^2\sum_\pm\sum_{d_1|k}\sum_{d_2|k}\sum_{m\leq N}\sum_{\substack{n\leq N\\d_1^2m=d_2^2n}}\frac{|A(d_1,m)\overline{A(d_2,n)|}}{d_1d_2mn}\left|S\left(\overline h,\pm m;\frac k{d_1}\right)S\left(\overline h,\mp n;\frac k{d_2}\right)\right|\\
&\qquad\times\int\limits_X^{2X}x^{2}\,e\!\left(\frac{\pm 6(d_1^2mx)^{1/3}}k\right)\mathrm dx\\
&\ll k^2\cdot k\,X^{2}\,X^{2/3}k\sum_{m\leq N}\frac{\left|A(1,m)\right|^2}{m^{7/3}}\ll X^{8/3}k^{4}
\end{align*} 
as in the proof of Theorem \ref{A_2-meansquare}. 

Next we study the off-diagonal contribution in the same sum, which is given, up to a constant, by 
\begin{align*}
&k^2\sum_\pm\sum_\pm\sum_{d_1|k}\sum_{d_2|k}\underset{d_1^2m\neq d_2^2n}{\sum_{m\leq N}\sum_{n\leq N}}\frac{A(m,d_1)\overline{A(n,d_2)}}{mnd_1d_2}\mathop{\text{\LARGE$\E$}}_{h\in\mathbb Z_k^\times}S\left(\overline h,\pm m;\frac k{d_1}\right)S\left(\overline h,\pm n;\frac k{d_2}\right)\\
&\qquad\qquad\times\int\limits_X^{2X}x^2e\,\left(\pm\frac{3(mxd_1^2)^{1/3}}k\mp\frac{3(nxd_2^2)^{1/3}}k\right)\,\mathrm d x. 
\end{align*}
We restrict to the case where the signs in $\sum_\pm$ are chosen to be the same as the case of distinct signs can be treated similarly with easier arguments. The prior terms contribute, using again the first derivative test and Weil's bound, 
\begin{align*}
\ll k^3\sum_{d_1|k}\sum_{d_2|k}\underset{d_1^2m\neq d_2^2n}{\sum_{m\leq N}\sum_{n\leq N}}\frac{|A(m,d_1)\overline{A(n,d_2)}|}{mn(d_1d_2)^{3/2}}\cdot\frac{X^{8/3}k}{\bigl|d_1^{2/3}m^{1/3}-d_2^{2/3}n^{1/3}\bigr|}.
\end{align*}      
By symmetry it is enough to treat the terms in the $m$-sum with $d_1^2m<d_2^2n$. For these we split the $m$-sum into two parts according to whether $d_1^2m<d_2^2n/2$ or $d_2^2n/2\leq d_1^2m<d_2^2n$. Note that in any case we have
\begin{align}\label{estimate-y} 
\frac1{d_2^{3/2}}\cdot\frac1{|d_2^{2/3}n^{1/3}-d_1^{2/3}m^{1/3}|}\ll\frac{n^{2/3}}{d_2^{1/6}|d_2^2n-d_1^2m|}.
\end{align}
Using this, the terms with $d_1^2m<d_2^2n/2$ contribute
\begin{align*}
&\ll X^{8/3}k^{4}\sum_{d_1|k}\sum_{d_2|k}\sum_{m\leq N}\frac{|A(d_1,m)|}{md_1^{3/2}}\underset{d_1^2m< d_2^2n/2}{\sum_{n\leqslant N}}\frac{|A(d_2,n)|}{nd_2^{1/6}}\cdot\frac{n^{2/3}}{|d_2^2n-d_1^2m|}\\
&\ll X^{8/3}k^{4}\sum_{d_1|k}\sum_{d_2|k}\sum_{m\leq N}\frac{|A(d_1,m)|}{md_1^{3/2}}\underset{d_1^2m< d_2^2n/2}{\sum_{n\leqslant N}}\frac{|A(d_2,n)|}{d_2^{13/6}n^{4/3}}
\ll_\eps X^{8/3+\varepsilon}k^{4}.
\end{align*}
In the last step we have used partial summation together with the averaged bounds (\ref{Average-estimates}) and the fact that $\vartheta\leq 5/14$.

For the terms with $d_2^2n/2\leq d_1^2m<d_2^2n$ we note that  
\[ 
m\geq m^{1/2}\left(\frac{d_2^2n}{2d_1^2}\right)^{1/2}.\]
Combining this with (\ref{estimate-y}) shows that the terms with $d_2^2n/2\leq d_1^2m<d_2^2n$ contribute
\begin{align*}
& \ll X^{8/3}k^4\sum_{d_1|k}\sum_{d_2|k}\sum_{m\leq N}\underset{d_2^2n/2\leq d_1^2m < d_2^2n}{\sum_{n\leqslant N}}\frac{|A(d_1,m)A(d_2,n)|}{m^{1/2}n^{5/6}d_1^{1/2}d_2^{7/6}}\cdot\frac1{|d_2^2n-d_1^2m|}
\end{align*}
Applying the Cauchy--Schwarz inequality shows that this is bounded by
\begin{align}\label{estimate-z}
&\ll k^4X^{8/3}\sum_{d_1|k}\sum_{d_2|k}\frac1{d_1^{1/2}d_2^{7/6}}\left(\sum_{m\leq N}\underset{d_2^2n/2\leq d_1^2m < d_2^2n}{\sum_{n\leqslant N}}\frac{|A(m,d_1)A(n,d_2)|^2}{mn^{5/3}}\right)^{1/2} \nonumber\\
&\qquad\times\left(\sum_{m\leq N}\underset{d_2^2n/2\leq d_1^2m < d_2^2n}{\sum_{n\leqslant N}}\frac1{(d_2^2n-d_1^2m)^2}\right)^{1/2}\\
&\ll_\eps k^4X^{8/3+\varepsilon}\sum_{d_1|k}\sum_{d_2|k}\frac1{d_1^{1/2-\vartheta}d_2^{7/6-\vartheta}} \nonumber\\
&\ll_\eps k^4X^{8/3+\varepsilon} \nonumber.
\end{align}
Above we have used partial summation together with (\ref{Average-estimates}) and the observation that the final double sum in (\ref{estimate-z}) is uniformly bounded with respect to $d_1$ and $d_2$. This finishes the proof.  
\end{proof}
           
\noindent As a consequence of the above theorem we obtain lower bounds for certain averaged mean squares. Theorem \ref{long_mean_square} follows immediately from the latter bound by (\ref{exp-sum}). 
\begin{theorem}\label{lower bounds}
Suppose that $k$ is a prime so that $k\ll X^{1/3-\delta}$ for any sufficiently small fixed $\delta>0$. Then the lower bounds
\[\mathop{\text{\LARGE$\E$}}_{h\in\mathbb Z_k^\times}\int\limits_X^{2X}\left|\widetilde A_1\left(x;\frac hk\right)\right|^2\,\mathrm dx\gg X^3k^3
\quad\text{and}\quad\mathop{\text{\LARGE$\E$}}_{h\in\mathbb Z_k^\times}\int\limits_X^{3X}\left|\widetilde A_0\left(x;\frac hk\right)\right|^2\,\mathrm dx\gg X^{5/3}k\]
hold for sufficiently large $X$.  
\end{theorem}

\begin{proof}
The first bound is immediate from the previous theorem. For the latter bound, the main observation is that we can bound the mean square of $\widetilde A_1(x;h/k)$ from above in terms of mean squares of $\widetilde A_2(x;h/k)$ and $\widetilde A_0(x;h/k)$. Knowing a lower bound for the averaged mean square of $\widetilde A_1(x;h/k)$ and an upper bound for the averaged mean square of $\widetilde A_2(x;h/k)$ then leads to the desired conclusion. 

Indeed, for any $0<H\leq X$ to be chosen later, we compute
\begin{align}\label{A_1toA_0estimate}
&\mathop{\text{\LARGE$\E$}}_{h\in\mathbb Z_k^\times}\int\limits_X^{2X}\left|\widetilde A_1\left(x;\frac hk\right)\right|^2\,\mathrm dx \nonumber\\
&=\frac1{H^2}\mathop{\text{\LARGE$\E$}}_{h\in\mathbb Z_k^\times}\int\limits_X^{2X}\left|\int\limits_x^{x+H}\widetilde A_1\left(x;\frac hk\right)\,\mathrm d t\right|^2\,\mathrm d x \nonumber\\
&=\frac1{H^2}\mathop{\text{\LARGE$\E$}}_{h\in\mathbb Z_k^\times}\int\limits_X^{2X}\left|\int\limits_x^{x+H}\left(\widetilde A_1\left(t;\frac hk\right)-\int\limits_x^t\widetilde A_0\left(u;\frac hk\right)\,\mathrm d u\right)\,\mathrm d t\right|^2\,\mathrm d x \nonumber\\
&\ll\frac1{H^2}\mathop{\text{\LARGE$\E$}}_{h\in\mathbb Z_k^\times}\int\limits_X^{2X}\left|\int\limits_x^{x+H}\widetilde A_1\left(t;\frac hk\right)\,\mathrm d t\right|^2\,\mathrm d x
+\frac1{H^2}\mathop{\text{\LARGE$\E$}}_{h\in\mathbb Z_k^\times}\int\limits_X^{2X}\left|\int\limits_x^{x+H}\int\limits_x^t\widetilde A_0\left(u;\frac hk\right)\,\mathrm d u\,\mathrm d t\right|^2\,\mathrm d x.
\end{align}
By simply integrating, the first term on the previous line is clearly 
\begin{align}\label{first-term-1}
\ll\frac1{H^2}\mathop{\text{\LARGE$\E$}}_{h\in\mathbb Z_k^\times}\int\limits_X^{2X+H}\left|\widetilde A_2\left(x;\frac hk\right)\right|^2\,\mathrm dx.
\end{align}
Using the upper bound $\widetilde A_2(x;h/k)\ll x^{5/3}k^{5/2}$ from Corollary \ref{a2bound} we see that (\ref{first-term-1}) is
\[ 
\ll\frac1{H^2}\mathop{\text{\LARGE$\E$}}_{h\in\mathbb Z_k^\times}\int\limits_X^{2X}\left|\widetilde A_2\left(x;\frac hk\right)\right|^2\,\mathrm dx+\frac1H X^{10/3}k^5.
\]

The latter term in (\ref{A_1toA_0estimate}) can be estimated as follows. First note that
\begin{align*}
\int\limits_x^{x+H}\int\limits_x^t\widetilde A_0\left(u;\frac hk\right)\,\mathrm d u\,\mathrm d t&=\int\limits_x^{x+H}\widetilde A_0\left(u;\frac hk\right)(u-x)\,\mathrm d u
\ll H\int\limits_x^{x+H}\left|\widetilde A_0\left(u;\frac hk\right)\right|\,\mathrm d u.
\end{align*}
Using this together with the Cauchy--Schwarz inequality yields
\begin{align*}
\frac1{H^2}\mathop{\text{\LARGE$\E$}}_{h\in\mathbb Z_k^\times}\int\limits_X^{2X}\left|\int\limits_x^{x+H}\int\limits_x^t\widetilde A_0\left(u;\frac hk\right)\,\mathrm d u\,\mathrm d t\right|^2\,\mathrm d x &\ll\mathop{\text{\LARGE$\E$}}_{h\in\mathbb Z_k^\times}\int\limits_X^{2X}\left(\int\limits_x^{x+H}\left|\widetilde A_0\left(u;\frac hk\right)\right|\mathrm d u\right)^2\,\mathrm d x\\
&\ll H\mathop{\text{\LARGE$\E$}}_{h\in\mathbb Z_k^\times}\int\limits_X^{2X+H}\int\limits_x^{x+H}\left|\widetilde A_0\left(u;\frac hk\right)\right|^2\mathrm d u\,\mathrm d x \\
&\ll H^2\mathop{\text{\LARGE$\E$}}_{h\in\mathbb Z_k^\times}\int\limits_X^{2X+H}\left|\widetilde A_0\left(x;\frac hk\right)\right|^2\,\mathrm d x\\
&\ll H^2\mathop{\text{\LARGE$\E$}}_{h\in\mathbb Z_k^\times}\int\limits_X^{3X}\left|\widetilde A_0\left(x;\frac hk\right)\right|^2\,\mathrm d x. 
\end{align*}
Thus we have shown that
\begin{align*}
&\mathop{\text{\LARGE$\E$}}_{h\in\mathbb Z_k^\times}\int\limits_X^{2X}\left|\widetilde A_1\left(x;\frac hk\right)\right|^2\,\mathrm dx\\
&\ll\frac1{H^2}\mathop{\text{\LARGE$\E$}}_{h\in\mathbb Z_k^\times}\int\limits_X^{2X}\left|\widetilde A_2\left(x;\frac hk\right)\right|^2\,\mathrm dx
+H^2\mathop{\text{\LARGE$\E$}}_{h\in\mathbb Z_k^\times}\int\limits_X^{3X}\left|\widetilde A_0\left(x;\frac hk\right)\right|^2\,\mathrm dx+\frac1H X^{10/3}k^5.
\end{align*}
As the second moment of $\widetilde A_2(x;h/k)$ averaged over $h$ modulo $k$ is $\asymp X^{13/3}k^5$ when $k\ll X^{1/3-\delta}$ by Theorem \ref{A_2-meansquare}, this implies
\begin{align}\label{connecting_ineq_1}
\mathop{\text{\LARGE$\E$}}_{h\in\mathbb Z_k^\times}\int\limits_X^{2X}\left|\widetilde A_1\left(x;\frac hk\right)\right|^2\,\mathrm dx&\ll\frac{X^{13/3}k^5}{H^2}+H^2\mathop{\text{\LARGE$\E$}}_{h\in\mathbb Z_k^\times}\int\limits_X^{3X}\left|\widetilde A_0\left(x;\frac hk\right)\right|^2\,\mathrm dx+\frac1H X^{10/3}k^5 \nonumber \\
&\ll\frac{X^{13/3}k^5}{H^2}+H^2\mathop{\text{\LARGE$\E$}}_{h\in\mathbb Z_k^\times}\int\limits_X^{3X}\left|\widetilde A_0\left(x;\frac hk\right)\right|^2\,\mathrm dx, 
\end{align}
where the last estimate holds as $H\leq X$.

Next we shall optimise the choice of the parameter $H$. Actually, before that we need a preliminary lower bound for the moment 
\begin{align*}
\mathop{\text{\LARGE$\E$}}_{h\in\mathbb Z_k^\times}\int\limits_X^{3X}\left|\widetilde A_0\left(x;\frac hk\right)\right|^2\,\mathrm dx
\end{align*}
we are interested in. This is to guarantee that our optimal choice for $H$ satisfies the constraint $H\leq X$.

We make the initial choice $H=X$ in (\ref{connecting_ineq_1}). This gives
\begin{align}\label{prelim_lower_bound_a}
\mathop{\text{\LARGE$\E$}}_{h\in\mathbb Z_k^\times}\int\limits_X^{3X}\left|\widetilde A_0\left(x;\frac hk\right)\right|^2\,\mathrm dx\gg Xk^3
\end{align}
by Theorem \ref{A1ms} as $k\ll X^{1/3-\delta}$. We now make the final optimal choice for $H$ by choosing
\[ 
H=\frac{X^{13/12}k^{5/4}}{\sqrt[4]{\mathop{\text{\LARGE$\E$}}_{h\in\mathbb Z_k^\times}\int\limits_X^{3X}\left|\widetilde A_0\left(x;\frac hk\right)\right|^2\,\mathrm dx}}.\]
This satisfies the constraint $H\leq X$ for sufficiently large $X$ as by (\ref{prelim_lower_bound_a}) and $k\ll X^{1/3-\delta}$ we have 
\[ 
H\ll\frac{X^{13/12}k^{5/4}}{X^{1/4}k^{3/4}}=X^{5/6}k^{1/2}\ll X^{1-\delta/2}.
\]
Using this choice for $H$ in (\ref{connecting_ineq_1}) gives 
\[ 
\mathop{\text{\LARGE$\E$}}_{h\in\mathbb Z_k^\times}\int\limits_X^{3X}\left|\widetilde A_0\left(x;\frac hk\right)\right|^2\,\mathrm dx\gg X^{5/3}k\]
as, by the previous theorem, the mean-square of $\widetilde A_1(x;h/k)$ averaged over $h$ modulo $k$ is $\asymp X^3k^3$ when $k\ll X^{1/3-\delta}$. Hence, the second bound is also established. 
\end{proof}

\section{An improved pointwise bound for long sums with rational twists}

\noindent In this section we improve the best known upper bound for rationally additively twisted sums with sufficiently small denominators.
\begin{theorem}
Let $x\in\left[1,\infty\right[$, and let $h$ and $k$ be coprime integers with $1\leqslant k\ll x^{1/3}$. Then
\[\sum_{m\leqslant x}A(m,1)\,e\!\left(\frac{mh}k\right)\ll_\eps k^{3/4}\,x^{1/2+\vartheta/2+\varepsilon}.\]
\end{theorem}

\begin{proof}
By the estimate (\ref{Connecting_A_0_to_A}) it is enough to prove the desired bound for $\widetilde A_0(x;h/k)$. Let $H\in\left[1,\infty\right[$ satisfy $H\ll x$, which will be chosen at the end. Estimating a short exponential sum by absolute values yields
\[\int\limits_x^{x+H}\widetilde A_0\!\left(t;\frac hk\right)\,\mathrm dt-H\,\widetilde A_0\!\left(x;\frac hk\right)=\int\limits_x^{x+H}\left(\widetilde A_0\!\left(t;\frac hk\right)-\widetilde A_0\!\left(x;\frac hk\right)\right)\mathrm dt\ll_\varepsilon H^2\,x^{\vartheta+\varepsilon},\]
so that
\[H\,\widetilde A_0\!\left(x;\frac hk\right)=\int\limits_x^{x+H}\widetilde A_0\!\left(t;\frac hk\right)\,\mathrm dt+O(H^2\,x^{\vartheta+\varepsilon}).\]
We know from Corollary \ref{better-a1bound} that
\[\widetilde A_1\!\left(x;\frac hk\right)\ll_\eps k^{3/2}\,x^{1+\varepsilon}.\]
This bound gives
\begin{align*}
\widetilde A_0\!\left(x;\frac hk\right)
&=\frac1{H}\int\limits_{x}^{x+H}\widetilde A_0\!\left(t;\frac hk\right)\,\mathrm dt+O(H\,x^{\vartheta+\varepsilon})\\
&=\frac1{H}\left(\widetilde A_1\!\left(x+H;\frac hk\right)\vphantom{\Bigg|}-\widetilde A_1\!\left(x;\frac hk\right)\right)+O(H\,x^{\vartheta+\varepsilon})\\
&\ll_\eps\frac1H\,k^{3/2}\,x^{1+\varepsilon}+H\,x^{\vartheta+\varepsilon}.
\end{align*}
Choosing $H=k^{3/4}\,x^{(1-\vartheta)/2}$ yields
\[\widetilde A_0\!\left(x;\frac hk\right)\ll_\eps k^{3/4}\,x^{1/2+\vartheta/2+\varepsilon},\]
as desired. 
\end{proof}

\section{Short second moment for Riesz weighted sums}

\noindent The goal of this section is to compute the short second moment of Riesz weighted sums for $a=2$ and $a=3$. Let us start by listing some consequences of the first derivative test.

\begin{lemma}\label{first-derivative1}
Let $\Xi,U,T\in\left[0,X\right]$ and $k,m,n,d_1,d_2\in\mathbb Z_+$. Then
\begin{align*}
&\int\limits_X^{X+\Xi}x^{(4a+2)/3}\,e\!\left(\pm\frac{3\!\left(d_1^{2/3}m^{1/3}\left(x+T\right)^{1/3}+d_2^{2/3}n^{1/3}\left(x+U\right)^{1/3}\right)}k\right)\mathrm dx \ll\frac{X^{(4a+4)/3}k}{d_1^{2/3}m^{1/3}+d_2^{2/3}n^{1/3}}.
\end{align*}
\end{lemma}

\begin{lemma}\label{first-derivative-2}
Let $\Xi,U,T\in\left[0,X\right]$ and $k,m,n,d_1,d_2\in\mathbb Z_+$ with $U\leqslant T$ and $d_1^2m<d_2^2n$. Then
\begin{align*}
&\int\limits_X^{X+\Xi}x^{(4a+2)/3}\,e\!\left(\pm\frac{3\!\left(d_1^{2/3}m^{1/3}\left(x+T\right)^{1/3}-d_2^{2/3}n^{1/3}\left(x+U\right)^{1/3}\right)}k\right)\mathrm dx \ll\frac{X^{(4a+4)/3}k}{d_2^{2/3}n^{1/3}-d_1^{2/3}m^{1/3}}.
\end{align*}
\end{lemma}

\begin{lemma}\label{first-derivative-3}
Let $\Xi,T\in\left]0,X\right]$ and $k,m,n,d_1,d_2\in\mathbb Z_+$ with $d_1^2m<d_2^2n$, $d_1\geqslant d_2$, as well as $d_1^2 m\leqslant X/6T$. Then
\begin{align}\label{final-first-derivative}
&\int\limits_X^{X+\Xi}x^{(4a+2)/3}\,e\!\left(\pm\frac{3\!\left(d_1^{2/3}m^{1/3}\,x^{1/3}-d_2^{2/3}n^{1/3}\left(x+T\right)^{1/3}\right)}k\right)\mathrm dx \nonumber \\
&\qquad\qquad\ll\frac{X^{(4a+4)/3}\,k\,n^{2/3}\,d_2^{-2/3}}{n-m}.
\end{align}
\end{lemma}

\noindent The first two lemmas are straightforward to prove, but the third one requires some explanation. The derivative of the phase function is, up to a sign, given by
\begin{align*}
\frac{d_1^{2/3}m^{1/3}}{x^{2/3}k}-\frac{d_2^{2/3}n^{1/3}}{(x+T)^{2/3}k}
=\frac{d_1^{2/3}m^{1/3}\left(x+T\right)^{2/3}-d_2^{2/3}n^{1/3}\,x^{2/3}}{\left(x+T\right)^{2/3}\,x^{2/3}\,k}.
\end{align*}
The denominator is of course $\asymp X^{4/3}k$, and the numerator is, by using the identity $x^3-y^3=(x-y)(x^2+xy+y^2)$, given by
\begin{align*}
\frac{\left(d_1^2m-d_2^2n\right)x^2+2\,x\,Td_1^2m+d_1^2m\,T^2}{\left(d_1^{2/3}m^{1/3}\left(x+T\right)^{2/3}\right)^2+
d_1^{2/3}m^{1/3}\,d_2^{2/3}n^{1/3}\,x^{2/3}\left(x+T\right)^{2/3}
+\left(d_2^{2/3}n^{1/3}\,x^{2/3}\right)^2}.
\end{align*}
In the numerator, the denominator is given by $\asymp d_2^{4/3}n^{2/3}\,X^{4/3}$ and the numerator is
\[\asymp\left(d_2^2n-d_1^2m\right)X^2.\]
Now (\ref{final-first-derivative}) follows immediately from Lemma \ref{Lemma-Huxley}.

Let us then proceed to the proof of the short mean square result for Riesz weighted exponential sums related to Hecke--Maass cusp forms. 

\begin{theorem}\label{short-mean-square}
Let $X\in\left[1,\infty\right[$, $k$ be a prime, and let $\Delta,\Xi\in\left[1,X\right]$ satisfy\footnote{Observe that the first condition implies $k\ll X^{1/3-\delta}$.} $X^{1/2+\delta}k^{3/2}\ll\Delta\ll X^{2/3-\delta}k$, and $X^{2+\delta}k^3\ll\Xi\,\Delta^2$ for any sufficiently small fixed $\delta>0$.
Then, for $X$ large enough,
\[\mathop{\text{\LARGE$\E$}}_{h\in\mathbb Z_k^\times}\int\limits_X^{X+\Xi}\left|\widetilde A_2\left(x+\Delta;\frac hk\right)-\widetilde A_2\left(x;\frac hk\right)\right|^2\mathrm dx\asymp\Xi\,\Delta^2\,X^2k^3,\]
and
\[\mathop{\text{\LARGE$\E$}}_{h\in\mathbb Z_k^\times}\int\limits_X^{2X}\left|\widetilde A_3\left(x+\Delta;\frac hk\right)-\widetilde A_3\left(x;\frac hk\right)\right|^2\mathrm dx\asymp\Xi\,\Delta^2\,X^{10/3}k^5.\]
\end{theorem}

\begin{proof}   
We shall actually prove a more general result concerning $\widetilde A_a(x;h/k)$ for any integer $a\geqslant2$ satisfying $a\equiv0\pmod3$ or $a\equiv2\pmod3$. Our starting point is the formula
\begin{align*}   
&\widetilde A_a\left(x+\Delta;\frac hk\right)-\widetilde A_a\left(x;\frac hk\right)
=\frac{(-1)^ak^a}{(2\pi)^{a+1}\sqrt3}\sum_\pm i^{\pm a}\sum_{d|k}\sum_{m=1}^\infty\frac{A(d,m)}{m^{(a+2)/3}d^{(2a+1)/3}}S\left(\overline h,\pm m;\frac kd\right)\\
&\qquad\times\left(\left(x+\Delta\right)^{(2a+1)/3}e\!\left(\pm\frac{3(md^2(x+\Delta))^{1/3}}k\right)-x^{(2a+1)/3}\,e\!\left(\pm\frac{3(md^2x)^{1/3}}k\right)\right)\\
&\qquad\qquad+O\,\left(X^{2a/3}k^{(2a+3)/2}\right),
\end{align*}
which follows from Corollary \ref{approximate-voronoi-for-large-a}. Note that $d(k)$ does not appear in the error term as we assume $k$ to be a prime.  

Exchanging here the factor $(x+\Delta)^{(2a+1)/3}$ to $x^{(2a+1)/3}$ costs $\ll\Delta\,X^{(2a-2)/3}k^{(2a+1)/2}$, which is $\ll X^{2a/3}k^{(2a+3)/2}$ provided that $\Delta\ll kX^{2/3-\delta}$. Thus,
\begin{align*}
&\widetilde A_a\left(x+\Delta;\frac hk\right)-\widetilde A_a\left(x;\frac hk\right)\\
&=\frac{(-1)^ax^{(2a+1)/3}k^a}{(2\pi)^{a+1}\sqrt3}\sum_\pm i^{\pm a}\sum_{d|k}\sum_{m=1}^\infty\frac{A(d,m)}{m^{(a+2)/3}d^{(2a+1)/3}}S\left(\overline h,\pm m;\frac kd\right)\\&\qquad
\times\left(e\!\left(\pm\frac{3(md^2(x+\Delta))^{1/3}}k\right)-e\!\left(\pm\frac{3(md^2x)^{1/3}}k\right)\right)+O\,\left(X^{2a/3}k^{(2a+3)/2}\right).
\end{align*}
Now the second moment may be expanded as
\begin{align}\label{second-moment-for-a-sum}
&\mathop{\text{\LARGE$\E$}}_{h\in\mathbb Z_k^\times}\int\limits_X^{X+\Xi}\left|\widetilde A_a\left(x+\Delta;\frac hk\right)-\widetilde A_a\left(x;\frac hk\right)\right|^2\,\mathrm dx \nonumber\\
&=\frac{k^{2a}}{(2\pi)^{2a+2}\cdot3} \nonumber\\
&\times\int\limits_X^{X+\Xi}x^{(4a+2)/3}\mathop{\text{\LARGE$\E$}}_{h\in\mathbb Z_k^\times}\left|\sum_\pm i^{\pm a}\sum_{d|k}\sum_{m=1}^\infty\frac{A(d,m)}{m^{(a+2)/3}d^{(2a+1)/3}}S\left(\overline h,\pm m;\frac kd\right)\left(e\!\left(\ldots\right)-e\!\left(\ldots\right)\right)\right|^2\mathrm dx \nonumber\\
&-\frac{(-1)^ak^a}{2^a\pi^{a+1}\sqrt3}\Re\int\limits_X^{X+\Xi}x^{(2a+1)/3}\sum_\pm i^{\pm a}\sum_{d|k}\sum_{m=1}^\infty\frac{A(d,m)}{m^{(a+2)/3}d^{(2a+1)/3}}
\mathop{\text{\LARGE$\E$}}_{h\in\mathbb Z_k^\times}S\left(\overline h,\pm m;\frac kd\right) \nonumber \\
&\qquad\qquad\qquad\times\left(e\!\left(\ldots\right)-e\!\left(\ldots\right)\right)\cdot O\,\left(X^{2a/3}k^{(2a+3)/2}\right)\,\mathrm dx+O\,\left(\Xi\,X^{4a/3}k^{2a+3}\right).
\end{align}
Once we have proved that the first term is $\asymp\Xi\,\Delta^2\,X^{(4a-2)/3}k^{2a-1}$, then the middle term is by the Cauchy--Schwarz inequality
\begin{align}\label{Estimate-1}
\ll\sqrt{\Xi\,\Delta^2\,X^{(4a-2)/3}k^{2a-1}}\sqrt{\Xi\,X^{4a/3}k^{2a+3}}
\ll\Xi\,\Delta\,X^{(4a-1)/3}\,k^{2a+1}.
\end{align}
Thus, we can focus on the first term on the right-hand side of (\ref{second-moment-for-a-sum}). In the course of examining this term various subterms arise. We will derive an upper bound for each of them and at the end of the proof show that these terms contribute less than the main term that will turn out to be $\asymp\Xi\,\Delta^2\,X^{(4a-2)/3}\,k^{2a-1}$. 

For $\lambda\in\{0,\Delta\}$ we write
\[s_\lambda^+:=i^a\sum_{d|k}\sum_{m=1}^\infty\frac{A(d,m)}{m^{(a+2)/3}d^{(2a+1)/3}}e\!\left(\frac{3(md^2(x+\lambda))^{1/3}}k\right)S\left(\overline h,m;\frac kd\right)\]and\[
s_\lambda^-:=i^{-a}\sum_{d|k}\sum_{m=1}^\infty\frac{A(d,m)}{m^{(a+2)/3}d^{(2a+1)/3}}e\!\left(-\frac{3(md^2(x+\lambda))^{1/3}}k\right)S\left(\overline h,-m;\frac kd\right),\]
so that, up to a constant, the integral under study takes the shape
\begin{align}\label{Expanding-the-integral}
&k^{2a}\int\limits_X^{X+\Xi}x^{(4a+2)/3}\mathop{\text{\LARGE$\E$}}_{h\in\mathbb Z_k^\times}\left|\sum_\pm i^{\pm a}\sum_{d|k}\sum_{m=1}^\infty\frac{A(d,m)}{m^{(a+2)/3}d^{(2a+1)/3}}S\left(\overline h,\pm m;\frac kd\right)\left(e\!\left(\ldots\right)-e\!\left(\ldots\right)\right)\right|^2\mathrm dx \nonumber\\
&=k^{2a}\int\limits_X^{X+\Xi}x^{(4a+2)/3}\mathop{\text{\LARGE$\E$}}_{h\in\mathbb Z_k^\times}\left|\left(s_\Delta^+-s_0^+\right)+\left(s_\Delta^--s_0^-\right)\right|^2\,\mathrm d x \nonumber\\
&= k^{2a}\sum_{\pm}\int\limits_X^{X+\Xi}x^{(4a+2)/3}\mathop{\text{\LARGE$\E$}}_{h\in\mathbb Z_k^\times}\left|s_\Delta^\pm-s_0^\pm\right|^2\mathrm dx \nonumber \\&\qquad\qquad\quad+2k^{2a}\Re\int\limits_X^{X+\Xi}x^{(4a+2)/3}\mathop{\text{\LARGE$\E$}}_{h\in\mathbb Z_k^\times}\left(s_\Delta^+-s_0^+\right)\overline{\left(s_\Delta^--s_0^-\right)}\,\mathrm d x,
\end{align}
where we sum over both choices of the sign $\pm$. Furthermore, we split each of the sums $s_\lambda^\pm$ into parts
\[s_\lambda^{\pm,\leqslant}:=i^{\pm a}\sum_{d|k}\sum_{m\leqslant N_0}\ldots\quad\text{and}\quad
s_\lambda^{\pm,>}:=i^{\pm a}\sum_{d|k}\sum_{m>N_0}\ldots,\]
where we set $N_0:=X/6\Delta$.

Now the first integral on the right-hand side of (\ref{Expanding-the-integral}) is
\begin{align}\label{Expanding-the-integral-2}
&k^{2a}\sum_\pm\int\limits_X^{X+\Xi}x^{(4a+2)/3}\mathop{\text{\LARGE$\E$}}_{h\in\mathbb Z_k^\times}\left|s_\Delta^{\pm,\leqslant}-s_0^{\pm,\leqslant}\right|^2\mathrm dx\\
&\qquad+k^{2a}\sum_\pm\int\limits_X^{X+\Xi}x^{(4a+2)/3}\mathop{\text{\LARGE$\E$}}_{h\in\mathbb Z_k^\times}\left|s_\Delta^{\pm,>}-s_0^{\pm,>}\right|^2\mathrm dx \nonumber\\
&\qquad+2k^{2a}\Re\sum_\pm\int\limits_X^{X+\Xi}x^{(4a+2)/3}\mathop{\text{\LARGE$\E$}}_{h\in\mathbb Z_k^\times}\left(s_\Delta^{\pm,\leqslant}-s_0^{\pm,\leqslant}\right)\overline{\left(s_\Delta^{\pm,>}-s_0^{\pm,>}\right)}\mathrm dx. \nonumber
\end{align}
Of the three new integrals we treat the last one first. It is given by
\begin{align*}
&2k^{2a}\Re\sum_\pm\int\limits_X^{X+\Xi}x^{(4a+2)/3}\left(s_\Delta^{\pm,\leqslant}-s_0^{\pm,\leqslant}\right)\overline{\left(s_\Delta^{\pm,>}-s_0^{\pm,>}\right)}\,\mathrm dx\\
&=k^{2a}\sum_\pm\int\limits_X^{X+\Xi}x^{(4a+2)/3}\left(\overline{\left(s_\Delta^{\pm,\leqslant}-s_0^{\pm,\leqslant}\right)}\left(s_\Delta^{\pm,>}-s_0^{\pm,>}\right)+\left(s_\Delta^{\pm,\leqslant}-s_0^{\pm,\leqslant}\right)\overline{\left(s_\Delta^{\pm,>}-s_0^{\pm,>}\right)}\right)\mathrm dx\\
&\ll k^{2a}\sum_\pm\sum_{d_1|k}\sum_{m\leq N_0}\frac{A\left(d_1,m\right)}{m^{(a+2)/3}d_1^{(2a+1)/3}}S\left(\overline h,\pm m;\frac k{d_1}\right)\\
&\qquad\qquad\times\sum_{d_2|k}\sum_{n> N_0}\frac{\overline{A\left(d_2,n\right)}}{n^{(a+2)/3}d_2^{(2a+1)/3}}S\left(\overline h,\pm n;\frac k{d_2}\right)\\
&\times\int\limits_X^{X+\Xi}x^{(4a+2)/3}\left(e\!\left(\mp\frac{3\left(x+\Delta\right)^{1/3}\left(m^{1/3}d_1^{2/3}-n^{1/3}d_2^{2/3}\right)}k\right) \right. \\
&\left. \qquad\qquad\qquad-e\!\left(\mp\frac{3\left(\left(x+\Delta\right)^{1/3}m^{1/3}d_1^{2/3}-x^{1/3}n^{1/3}d_2^{2/3}\right)}k\right)\right. \\ 
&\left.\qquad\qquad\qquad\qquad \qquad\quad-e\!\left(\mp\frac{3\left(x^{1/3}m^{1/3}d_1^{2/3}-\left(x+\Delta\right)^{1/3}n^{1/3}d_2^{2/3}\right)}k\right) \right. \\
&\left. \qquad\qquad\qquad\qquad\qquad\qquad\qquad\qquad+e\!\left(\mp\frac{3x^{1/3}\left(m^{1/3}d_1^{2/3}-n^{1/3}d_2^{2/3}\right)}k\right)\right)\mathrm d x.
\end{align*}
As $n>m$ and $k$ is a prime, the diagonal contribution from $d_1^2m=d_2^2n$ can only arise when $d_1=k$ and $d_2=1$ in the integrals corresponding to the first and fourth exponential term on the right-hand side of the previous display. This in particular means that $N_0<n=k^2m$ and so we have $m>N_0/k^2$ in the diagonal terms corresponding to the first and fourth exponential term. Hence, their total contribution is, for any sufficiently small fixed $\varepsilon>0$,
\begin{align}\label{Estimate-2}
&\ll k^{2a}\sum_{N_0/k^2<m\leq N_0}\frac{\left|A\left(k,m\right)\right|}{m^{(a+2)/3}k^{(2a+1)/3}}\cdot\frac{\left|A\left(1,k^2m\right)\right|}{\left(k^2m\right)^{(a+2)/3}}\cdot k^{1/2}\Xi X^{(4a+2)/3} \nonumber\\
&\ll_\eps k^{3\vartheta+2a/3-7/6}\Xi X^{(4a+2)/3}\sum_{N_0/k^2<m\leq N_0}\frac1{m^{(2a+4)/3-2\vartheta-2\varepsilon}} \nonumber\\
&\ll_\eps k^{3\vartheta+2a/3-7/6}\Xi X^{(4a+2)/3}\left(\frac{N_0}{k^2}\right)^{-2/3-\varepsilon}\sum_{N_0/k^2<m\leq N_0}\frac1{m^{(2a+2)/3-2\vartheta-3\varepsilon}} \nonumber\\
&\ll_\eps k^{3\vartheta+2a/3+1/6+2\varepsilon}\Delta^{2/3+\varepsilon}\Xi X^{4a/3-\eps},
\end{align}
where we have applied the pointwise bound $\left|A\left(m,n\right)\right|\ll (mn)^{\vartheta+\varepsilon}$ in the second step and in the final step we recall the definition of $N_0=X/6\Delta$. 

The diagonal contribution in the integrals corresponding to the two other exponentials are estimated similarly, noting that by the first derivative test
\begin{align*}
\int\limits_X^{X+\Xi}x^{(4a+2)/3}e\!\left(\mp\frac{3m^{1/3}d^{2/3}\left(\left(x+\Delta\right)^{1/3}-x^{1/3}\right)}k\right)\,\mathrm d x\ll\frac{X^{(4a+7)/3}\,k}{\Delta\,m^{1/3}d^{2/3}}.
\end{align*}
\noindent We may estimate the off-diagonal terms with $d_1^2m\neq d_2^2n$ by Lemma \ref{first-derivative-2} or Lemma \ref{first-derivative-3} depending on the term in question. By symmetry it is enough to consider the terms with $d_1^2m<d_2^2n$. As a model case, we consider the sum
\begin{align*}
&k^{2a}\sum_\pm\sum_{d_1|k}\sum_{m\leq N_0}\frac{A\left(d_1,m\right)}{m^{(a+2)/3}d_1^{(2a+1)/3}}S\left(\overline h,\pm m;\frac k{d_1}\right)\sum_{d_2|k}\sum_{n> N_0}\frac{\overline{A\left(d_2,n\right)}}{n^{(a+2)/3}d_2^{(2a+1)/3}}S\left(\overline h,\pm n;\frac k{d_2}\right)\\
&\qquad\qquad\quad\times\int\limits_X^{X+\Xi}x^{(4a+2)/3}e\!\left(\mp\frac{3\left(x^{1/3}m^{1/3}d_1^{2/3}-\left(x+\Delta\right)^{1/3}n^{1/3}d_2^{2/3}\right)}k\right)\,\mathrm d x
\end{align*}
For simplicity, we further restrict to terms with $d_1=k$, $d_2=1$ (in which case $n=d_2^2n>d_1^2m=k^2m$) as the other terms are easier to deal with. 

Note that when $N_0/k^2<m\leq N_0$ we may argue as in the diagonal case to see that these terms are by Weil's bound 
\begin{align}\label{Estimate-3}
&\ll k^{2a+1/2}\sum_{N_0/k^2\leq m\leq N_0}\frac{|A(k,m)|}{m^{(a+2)/3}k^{(2a+1)/3}}\sum_{n>N_0}\frac{|A(1,n)|}{n^{(a+2)/3}}\cdot X^{(4a+2)/3}\Xi \nonumber\\
&\ll k^{4a/3+1/6+\vartheta}X^{(4a+2)/3}\Xi N_0^{-a/3+1/3}\left(\frac{N_0}{k^2}\right)^{-a/3+1/3} \nonumber\\
&\ll k^{2a-1/2+\vartheta}X^{(2a+4)/3}\Xi\Delta^{(2a-2)/3}.
\end{align}

For the terms with $m\leq N_0/k^2$ we use Lemma \ref{first-derivative-3} and Weil's bound to estimate the sum in question as
\begin{align}\label{off-diagonal-special-case}
&\ll k^{2a+1/2}\sum_{m\leqslant N_0/k^2}\frac{\left|A(k,m)\right|}{m^{(a+2)/3}k^{(2a+1)/3}}\sum_{n>N_0}\frac{\left|A(1,n)\right|}{n^{(a+2)/3}}\cdot\frac{X^{(4a+4)/3}\,n^{2/3}k}{|n-m|} \nonumber\\
&\ll k^{4a/3+7/6}X^{(4a+4)/3}\sum_{m\leq N_0/k^2}\frac{|A(k,m)|}{m^{(a+2)/3}}\sum_{n>N_0}\frac{|A(1,n)|}{n^{a/3}}\cdot\frac1{|n-m|}.
\end{align}
To the inner sum we apply the Cauchy--Schwarz inequality and partial summation (together with the estimate (\ref{Average-estimates})) to see that 
\begin{align*}
\sum_{n>N_0}\frac{|A(1,n)|}{n^{a/3}}\cdot\frac1{|n-m|}&\leq\left(\sum_{n>N_0}\frac{|A(1,n)|^2}{n^{2a/3}}\right)^{1/2}\left(\sum_{n>N_0}\frac1{(n-m)^2}\right)^{1/2} \\
&\ll N_0^{-a/3+1/2},
\end{align*}
where the last estimate holds uniformly in $m\leq N_0$. Using this in (\ref{off-diagonal-special-case}) gives a contribution 
\begin{align}\label{Estimate-4}
&\ll k^{4a/3+7/6}X^{(4a+4)/3}\sum_{m\leq N_0/k^2}\frac{|A(k,m)|^2}{m^{(a+2)/3}}\left(\frac X\Delta\right)^{-a/3+1/2} \nonumber \\
&\ll_\varepsilon k^{4a/3+7/6+2\vartheta+\varepsilon}X^{a+11/6}\Delta^{a/3-1/2},
\end{align} 
where we have in addition used partial summation together with (\ref{Average-estimates}).

Note also that the second integral on the right-hand side of (\ref{Expanding-the-integral}) can be treated using even simpler arguments. Indeed, applying Lemma \ref{first-derivative1} to bound the exponential integrals, and estimating then trivially using the Weil bound shows that these terms contribute
\begin{align}\label{second-integral}
&\ll k^{2a+1}\sum_{d_1|k}\sum_{m=1}^\infty\frac{|A(d_1,m)|}{m^{(a+2)/3}d_1^{(2a+1)/3}}\sum_{d_2|k}\sum_{n=1}^\infty\frac{|A(d_2,n)|}{n^{(a+2)/3}d_1^{(2a+1)/3}}\cdot\frac1{\sqrt{d_1d_2}} \nonumber\\
&\qquad\qquad\qquad\qquad\times\frac{X^{(4a+4)/3}k}{d_1^{2/3}m^{1/3}+d_2^{2/3}n^{1/3}} \nonumber \\
&\ll X^{(4a+4)/3}k^{2a+2}.
\end{align} 
  
Let us next treat the middle term in (\ref{Expanding-the-integral-2}). Using the triangle inequality it is clearly enough to estimate the integrals
\[k^{2a}\int\limits_X^{X+\Xi}x^{(4a+2)/3}\,\bigl|s_\lambda^{\pm,>}\bigr|^2\,\mathrm dx.\]
We expand the square to see that the integral is bounded from above by the expression
\begin{align*}
&k^{2a}\sum_{d_1|k}\sum_{d_2|k}\sum_{m>N_0}\sum_{n>N_0}\frac{\left|A\left(d_1,m\right)A\left(d_2,n\right)\right|}{m^{(a+2)/3}n^{(a+2)/3}d_1^{(2a+1)/3}d_2^{(2a+1)/3}}S\left(\overline h,\pm m;\frac k{d_1}\right)S\left(\overline h,\pm n;\frac k{d_2}\right)\\
&\qquad\qquad\times\int\limits_X^{X+\Xi}x^{(4a+2)/3}e\left(\pm\frac{3\left(x+\lambda\right)^{1/3}\left(m^{1/3}d_1^{2/3}-n^{1/3}d_2^{2/3}\right)}k\right)\mathrm d x.
\end{align*}
We treat the diagonal and off-diagonal terms separately. Let us first concentrate on the diagonal, which correspond to the terms where $d_1^2m_1=d_2^2m_2$. These contribute
\begin{align}\label{Estimate-5}
&\ll k^{2a+1}\sum_{m>N_0}\frac{\left|A(1,m)\right|^2}{m^{(2a+4)/3}}\cdot\Xi\,X^{(4a+2)/3} \nonumber\\
&\ll k^{2a+1}\,\Xi\,X^{(4a+2)/3}\,N_0^{-(2a+1)/3} \nonumber\\
&\ll k^{2a+1}\Xi\,X^{(2a+1)/3}\,\Delta^{(2a+1)/3}
\end{align}
by partial summation coupled with (\ref{Average-estimates}). 

Similarly, the off-diagonal terms contribute, using Lemma \ref{first-derivative-2} and Weil's bound, 
\begin{align*}
&\ll k^{2a+1}\sum_{d_1|k}\sum_{d_2|k}\sum_{m>N_0}\frac{\left|A(d_1,m)\right|}{m^{(a+2)/3}d_1^{(2a+1)/3}}\sum_{n>m}\frac{\left|A(d_2,n)\right|}{n^{(a+2)/3}d_2^{(2a+1)/3}}\cdot\frac{X^{(4a+4)/3}k}{|d_1^{2/3}m^{1/3}-d_2^{2/3}n^{1/3}|}.
\end{align*}
By symmetry it is enough to treat the terms in the $m$-sum with $d_1^2m<d_2^2n$. For these we split the $m$-sum into two parts according to whether $d_1^2m<d_2^2n/2$ or $d_2^2n/2\leq d_1^2m<d_2^2n$. Let $\eta>0$ be arbitrarily small, but fixed. Note that for $d_2^2n/2\leq d_1^2m$ we have, using (\ref{reciprocal-estimate}) and the pointwise bound $A(m,d)\ll (md)^{\vartheta+\eta}$, that 
\begin{align}\label{estimate_X}
&\frac{|A(d_2,n)|}{m^{(a+2)/3}d_1^{(2a+1)/3}n^{(a+2)/3}d_2^{(2a+1)/3}}\cdot\frac1{|d_2^{2/3}n^{1/3}-d_1^{2/3}m^{1/3}|} \nonumber\\
&\ll\frac{n^{\vartheta+\eta}}{m^{(a+2)/3}d_1^{(2a+1)/3}n^{(a+2)/3}d_2^{2a/3-1-\vartheta-\eta}|d_2^2n-d_1^2m|}.
\end{align}
Observe that 
\[ 
n^{a/3}=n^{1/6}n^{a/3-1/6}\gg\left(\frac{d_1^2m}{d_2^2}\right)^{1/6}n^{a/3-1/6}.\]
Using this together with an easy estimate
\[ 
n^\vartheta\ll\left(\frac{d_1^2m}{d_2^2}\right)^\vartheta\leq\left(\frac{d_1^2m}{d_2^2}\right)^{5/14}\]
gives that (\ref{estimate_X}) is
\begin{align*}
\ll \frac1{m^{a/3+10/21}d_1^{2a/3-1/21}n^{a/3-1/6-\eta}d_2^{2a/3-13/21-\vartheta-\eta}|d_2^2n-d_1^2m|}.
\end{align*}
This means that the terms with $d_2^2n/2\leq d_1^2m<d_2^2n$ contribute
\begin{align*}
&\ll k^{2a+2}X^{(4a+4)/3}\\
&\quad\times\sum_{d_1|k}\sum_{d_2|k}\sum_{n>N_0}\sum_{\substack{m>N_0\\
d_2^2n/2\leq d_1^2m<d_2^2n}}\frac{\left|A(d_1,m)\right|}{m^{a/3+10/21}d_1^{2a/3-1/21}n^{a/3-1/6-\eta}d_2^{2a/3-13/21-\vartheta-\eta}(d_2^2n-d_1^2m)}\\
&\ll k^{2a+2}X^{(4a+4)/3}\\
&\quad\times\sum_{d_1|k}\sum_{d_2|k}\sum_{n>N_0}\sum_{m>N_0}\left(\frac{d_1^2m}{d_2^2}\right)^{2\eta}\frac{\left|A(d_1,m)\right|}{m^{a/3+10/21}d_1^{2a/3-1/21}n^{a/3-1/6+\eta}d_2^{2a/3-13/21-\vartheta-\eta}(d_2^2n-d_1^2m)},
\end{align*}
where in the last step we have multiplied both the numerator and denominator by $n^{2\eta}$, used the estimate $n\ll d_1^2m/d_2^2$, and finally dropped the constraint $d_2^2n/2\leq d_1^2m<d_2^2n$ in the $m$-sum. Next, observe that for a fixed $m$ the $n$-sum can be bounded by the Cauchy--Schwarz inequality as
\[ 
\sum_{n>N_0}\frac1{n^{a/3-1/6+\eta}(d_2^2n-d_1^2m)}\leq\left(\sum_{n>N_0}\frac1{n^{2a/3-1/3+2\eta}}\right)^{1/2}\left(\sum_{n>N_0}\frac1{(d_2^2n-d_1^2m)^2}\right)^{1/2}\ll 1,\]
where the final bound is uniform in $m$, $d_1$, and $d_2$. 

Combining everything we have shown that in total the terms with $d_2^2n/2\leq d_1^2m<d_2^2n$ contribute, using partial summation together with (\ref{Average-estimates}),
\begin{align}\label{Estimate-6}
&\ll k^{2a+2}X^{(4a+4)/3}\sum_{d_1|k}\sum_{d_2|k}\sum_{m>N_0}\frac{1}{m^{a/3+10/21-2\eta}d_1^{2a/3-1/21-4\eta-\vartheta}d_2^{2a/3-13/21-\vartheta+3\eta}} \nonumber\\
&\ll k^{2a+2}X^{(4a+4)/3},
\end{align}
where the inner triple sum is $\ll 1$ by the estimate $\vartheta\leq 5/14$. 

For the terms with $d_1^2m<d_2^2n/2$ we note that by (\ref{estimate_X}) one has
\begin{align*}
&\frac{|A(d_2,n)|}{m^{(a+2)/3}d_1^{(2a+1)/3}n^{(a+2)/3}d_2^{(2a+1)/3}}\cdot\frac1{|d_2^{2/3}n^{1/3}-d_1^{2/3}m^{1/3}|} \nonumber\\
&\ll_\eps\frac{n^{\vartheta+\varepsilon}}{m^{(a+2)/3}d_1^{(2a+1)/3}n^{(a+2)/3}d_2^{2a/3-1-\vartheta-\varepsilon}(d_2^2n-d_1^2m)}\\
&\ll_\eps \frac1{m^{(a+2)/3}d_1^{(2a+1)/3}n^{(a+5)/3-\vartheta-\varepsilon}d_2^{(2a+7)/6-\vartheta-\varepsilon}}.
\end{align*}
Hence, the terms with $d_1^2m<d_2^2n/2$ contribute
\begin{align*}
\ll_\eps k^{2a+2}X^{(4a+4)/3}\sum_{d_1|k}\sum_{d_2|k}\sum_{n>N_0}\sum_{m>N_0}\frac{|A(d_1,m)|}{m^{(a+2)/3}d_1^{(2a+1)/3}n^{(a+5)/3-\vartheta-\varepsilon}d_2^{(2a+7)/6-\vartheta-\varepsilon}},
\end{align*}
which is again 
\begin{align}\label{Estimate-7}
\ll k^{2a+2}X^{(4a+4)/3}
\end{align}
using partial summation and (\ref{Average-estimates}).


Now only one integral remains, namely
\[k^{2a}\sum_\pm\int\limits_X^{X+\Xi}x^{(4a+2)/3}\mathop{\text{\LARGE$\E$}}_{h\in\mathbb Z_k^\times}\left|s_\Delta^{\pm,\leqslant}-s_0^{\pm,\leqslant}\right|^2\mathrm dx,\]
where we have for simplicity dropped the constant $(2\pi)^{-(2a+2)}/3$ in front. When we again multiply out $\left|\Sigma\right|^2=\Sigma\,\overline\Sigma$, we obtain from the off-diagonal terms by Lemma \ref{first-derivative-2} or Lemma \ref{first-derivative-3} depending on the term in question, and Weil's bound a contribution
\begin{align*}
&\ll k^{2a+2}\sum_{d_1|k}\sum_{d_2|k}\sum_{m\leqslant N_0}\sum_{m<n\leqslant N_0}\frac{\left|A(d_1,m)\right|\cdot\left|A(d_2,n)\right|}{m^{(a+2)/3}\,n^{(a+2)/3}d_1^{2a/3+5/6}d_2^{2a/3+5/6}}\cdot\frac{X^{(4a+4)/3}}{|d_1^{2/3}m^{1/3}-d_2^{2/3}n^{1/3}|}.
\end{align*}
Estimating this is completely analogous to the computation above and leads to an upper bound 
\begin{align}\label{Estimate-8}
\ll X^{(4a+4)/3}k^{2a+2}.
\end{align} 

To evaluate the diagonal contribution we set $N_1:=\ell\,X^2k^3\,\Delta^{-3}$ with a constant $\ell$ so small that
\[\frac{m^{1/3}\left((x+\Delta)^{1/3}-x^{1/3}\right)}k\leqslant\frac1{12},\]
say, for each $m\leqslant N_1$. Clearly $N_1<N_0$ for $X$ sufficiently large, as $\Delta\gg X^{1/2+\delta}k^{3/2}$.

By estimating trivially, the higher frequencies $m\geq N_1$ in the diagonal terms contribute
\begin{align}\label{Estimate-9}
&\ll k^{2a+1}\sum_{N_1<m\leqslant N_0}\frac{\left|A(1,m)\right|^2}{m^{(2a+4)/3}}\,\Xi\,X^{(4a+2)/3} \nonumber\\
&\ll k^{2a+1} N_1^{-2a/3-1/3}\,\Xi\,X^{(4a+2)/3}
\ll\Xi\,\Delta^{2a+1}.
\end{align}

In the remaining smaller frequency diagonal terms, we see that these terms are given by
\begin{align}\label{short-main-term}        
&k^{2a}\sum_\pm\sum_{d_1|k}\sum_{d_2|k}\sum_{m\leqslant N_1}\sum_{\substack{n\leqslant N_1\\d_1^2m=d_2^2n}}\frac{A(d_1,m)\overline{A(d_2,n)}}{m^{(a+2)/3}n^{(a+2)/3}d_1^{(2a+1)/3}d_2^{(2a+1)/3}}\nonumber \\
&\qquad\qquad\times\mathop{\text{\LARGE$\E$}}_{h\in\mathbb Z_k^\times}S\left(\overline h,\pm m;\frac k{d_1}\right)S\left(\overline h,\pm n;\frac k{d_2}\right) \nonumber\\
&\qquad\qquad\qquad\times\int\limits_X^{X+\Xi}x^{(4a+2)/3}\left|e\!\left(\pm\frac{3m^{1/3}d_1^{2/3}\left((x+\Delta)^{1/3}-x^{1/3}\right)}k\right)-1\right|^2\mathrm dx.
\end{align}    
Using Lemma \ref{Kloosterman-correlation} we note that this equals $S_1+S_2+S_3+S_4$, where
\begin{align*}
&S_1:=k^{2a}\left(k-\frac1{k-1}\right)\sum_\pm\sum_{\substack{m\leq N_1\\
m\not\equiv 0\,(k)}}\frac{|A(1,m)|^2}{m^{(2a+4)/3}}\\
&\qquad\qquad\times\int\limits_X^{X+\Xi}x^{(4a+2)/3}\left|e\left(\pm\frac{3m^{1/3}\left((x+\Delta)^{1/3}-x^{1/3}\right)}k\right)-1\right|^2\,\mathrm d x\\
&S_2:=k^{2a}\sum_\pm\sum_{m\leq N_1}\frac{|A(k,m)|^2}{m^{(2a+4)/3}k^{(4a+2)/3}}\\
&\qquad\qquad\times\int\limits_X^{X+\Xi}x^{(4a+2)/3}\left|e\left(\pm\frac{3m^{1/3}k^{2/3}\left((x+\Delta)^{1/3}-x^{1/3}\right)}k\right)-1\right|^2\,\mathrm d x\\
&S_3:=-2k^{2a}\sum_\pm\Re\sum_{m\leq N_1/k^2}\frac{A(1,k^2m)\overline{A(k,m)}}{(k^2m)^{(a+2)/3}m^{(a+2)/3}k^{(2a+1)/3}}\\
&\qquad\qquad\qquad\times\int\limits_X^{X+\Xi}x^{(4a+2)/3}\left|e\left(\pm\frac{3m^{1/3}k^{2/3}\left((x+\Delta)^{1/3}-x^{1/3}\right)}k\right)-1\right|^2\,\mathrm d x \\
&S_4:=k^{2a}\sum_\pm\sum_{\substack{m\leq N_1\\
m\equiv 0\,(k)}}\frac{|A(1,m)|^2}{m^{(2a+4)/3}}\int\limits_X^{X+\Xi}x^{(4a+2)/3}\left|e\left(\pm\frac{3m^{1/3}\left((x+\Delta)^{1/3}-x^{1/3}\right)}k\right)-1\right|^2\,\mathrm d x.
\end{align*}
Our next objective is to show that 
\begin{align}\label{Estimate-10}
S_1+S_2+S_3+S_4\asymp \Xi\Delta^2 X^{(4a-2)/3}k^{2a-1}.
\end{align}
The upper bound follows simply by estimating trivially using known average growth properties of the Fourier coefficients. Towards the lower bound, note that by the arithmetic-geometric mean inequality we have 
\begin{align*}
2\cdot\frac{|A(1,k^2m)\overline{A(k,m)|}}{(k^2m)^{(a+2)/3}m^{(a+2)/3}k^{(2a+1)/3}}\leq\frac{|A(1,k^2m)|^2}{(k^2m)^{(2a+4)/3}}+\frac{|A(k,m)|^2}{k^{ (4a+2)/3}m^{(2a+4)/3}}
\end{align*}
and so 
\begin{align*}
|S_3|&\leq k^{2a}\sum_\pm\sum_{m\leq N_1/k^2}\frac{|A(1,k^2m)|^2}{(k^2m)^{(2a+4)/3}}\int\limits_X^{X+\Xi}x^{(4a+2)/3}\left|e\left(\pm\frac{3m^{1/3}k^{2/3}\left((x+\Delta)^{1/3}-x^{1/3}\right)}k\right)-1\right|^2\,\mathrm d x\\
&\quad+k^{2a}\sum_\pm\sum_{m\leq N_1}\frac{|A(k,m)|^2}{k^{(4a+2)/3}m^{(2a+4)/3}}\int\limits_X^{X+\Xi}x^{(4a+2)/3}\left|e\left(\pm\frac{3m^{1/3}k^{2/3}\left((x+\Delta)^{1/3}-x^{1/3}\right)}k\right)-1\right|^2\,\mathrm d x\\
&\leq |S_4|+|S_2|.
\end{align*}
From this it follows that 
\begin{align*} 
S_1+S_2+S_3+S_4 &\gg k^{2a+1}\sum_\pm\sum_{m\leq N_1}\frac{|A(1,m)|^2}{m^{(2a+4)/3}}\left(1-1_{m\equiv 0\,(k)}\right)\\
&\qquad\qquad\times\int\limits_X^{X+\Xi} x^{(4a+2)/3}\left|e\left(\pm\frac{3m^{1/3}\left((x+\Delta)^{1/3}-x^{1/3}\right)}{k}\right)-1\right|^2\,\mathrm d x.
\end{align*}
Using the non-negativity of the terms involved and dropping all but the first term in the sum, this is
\[ 
\geq k^{2a+1}\sum_{\pm}\int\limits_X^{X+\Xi}x^{(4a+2)/3}\left|e\left(\pm\frac{3\left((x+\Delta)^{1/3}-x^{1/3}\right)}{k}\right)-1\right|^2\,\mathrm d x.\]   
Using the estimate $\left|e(\alpha)-1\right|\asymp\left|\alpha\right|$, which holds for small $|\alpha|$ (recall that $\Delta\ll X^{2/3-\delta}k$), we can bound this from below by
\begin{align*}
&\gg k^{2a+1}\Xi\,X^{(4a+2)/3}\left(\frac{\Delta}{X^{2/3}k}\right)^2\\
&\asymp\Xi\,\Delta^2\,X^{(4a-2)/3}\,k^{2a-1}.
\end{align*}
We conclude that the expression (\ref{short-main-term}) is $\asymp \Xi\,\Delta^2\,X^{(4a-2)/3}\,k^{2a-1}$. 

Combining estimates (\ref{second-moment-for-a-sum})-(\ref{Estimate-10}) we have shown that
\begin{align*}
\mathop{\text{\LARGE$\E$}}_{h\in\mathbb Z_k^\times}\int\limits_X^{X+\Xi}\left|\widetilde A_a\left(x+\Delta;\frac hk\right)-\widetilde A_a\left(x;\frac hk\right)\right|^2\,\mathrm dx=S+O(E),
\end{align*}
where $S\asymp\Xi\,\Delta^2\,X^{(4a-2)/3}\,k^{2a-1}$ and
\begin{align*}
E:=&\Xi X^{4a/3}k^{2a+3}+\Xi\Delta X^{(4a-1)/3}k^{2a+1}+k^{3\vartheta+2a/3+1/6+2\eps}\Delta^{2/3+\eps}\Xi X^{4a/3-\varepsilon} \\
&+k^{2a+\vartheta-1/2}X^{(2a+4)/3}\Xi\Delta^{(2a-2)/3}+k^{4a/3+7/6+\vartheta+\varepsilon}X^{a+11/6}\Delta^{a/3-1/2}\\
&+X^{(4a+4)/3}k^{2a+2}+k^{2a+1}\Xi X^{(2a+1)/3}\Delta^{(2a+1)/3}+\Xi \Delta^{2a+1}.
\end{align*}
Now we note that $E=o(S)$ by considering the terms in $E$ individually. Below we list which assumptions are needed to bound each of the terms. These claims can be easily verified by straightforward computations:  \\
\begin{itemize}
\item $\Xi X^{4a/3}k^{2a+3}=o(S)$ holds as $\Delta\gg X^{1/2+\delta}k^{3/2}$ and $k\ll X^{1/3-\delta}$.

\item $\Xi\Delta X^{(4a-1)/3}k^{2a+1}=o(S)$ holds as $\Delta\gg X^{1/2+\delta}k^{3/2}$ and $k\ll X^{1/3-\delta}$.

\item $k^{3\vartheta+2a/3+1/6+2\varepsilon}\Delta^{2/3+\varepsilon}\Xi X^{4a/3-\eps}=o(S)$ holds as $\Delta\gg X^{1/2+\delta}k^{3/2}$ and $\vartheta\leq 5/14$.

\item $k^{2a+\vartheta-1/2}X^{(2a+4)/3}\Xi\Delta^{(2a-2)/3}=o(S)$ holds as $\Delta\ll X^{2/3-\delta}k$, $k\ll X^{1/3-\delta}$, and $\vartheta<1/2$ (when $a\geq 4)$ and as $\Delta\gg X^{1/2+\delta}k^{3/2}$, $k\ll X^{1/3-\delta}$, and $\vartheta<1/2$ (when $a<4)$. 

\item $k^{4a/3+7/6+2\vartheta+\varepsilon}X^{a+11/6}\Delta^{a/3-1/2}=o(S)$ holds as $\Xi\Delta^2\gg X^{2+\delta}k^3$, $\Delta\ll X$, and $\vartheta\leq 5/14$.

\item $X^{(4a+4)/3}k^{2a+2}=o(S)$ holds as $X^{2+\delta}k^3\ll\Xi\Delta^2$. 

\item $k^{2a+1}\Xi X^{(2a+1)/3}\Delta^{(2a+1)/3}=o(S)$ holds as $\Delta\gg X^{1/2+\delta}k^{3/2}$ and $k\ll X^{1/3-\delta}$.

\item $\Xi \Delta^{2a+1}=o(S)$ holds as $\Delta\ll X^{2/3-\delta}k$.
\end{itemize} 
\noindent These yield the claimed estimate
\[ 
\mathop{\text{\LARGE$\E$}}_{h\in\mathbb Z_k^\times}\int\limits_X^{X+\Xi}\left|\widetilde A_a\left(x+\Delta;\frac hk\right)-\widetilde A_a\left(x;\frac hk\right)\right|^2\,\mathrm dx\asymp\Xi\,\Delta^2\,X^{(4a-2)/3}\,k^{2a-1}. \]
The statements of the theorem follow by specialising to the cases $a=2$ and $a=3$.

\end{proof}     
  
\section{$\Omega$-result for short sums of Fourier coefficients}

\noindent Choosing $\Xi=X$ (which clearly satisfies the required assumptions) in Theorem \ref{short-mean-square} leads to the bounds
\begin{align}\label{A_2-meansquare-2}
\mathop{\text{\LARGE$\E$}}_{h\in\mathbb Z_k^\times}\int\limits_X^{2X}\left|\widetilde A_2\left(x+\Delta;\frac hk\right)-\widetilde A_2\left(x;\frac hk\right)\right|^2\,\mathrm dx\asymp\Delta^2\,X^3k^3,
\end{align}                  
and
\begin{align}\label{A_3-meansquare}
\mathop{\text{\LARGE$\E$}}_{h\in\mathbb Z_k^\times}\int\limits_X^{2X}\left|\widetilde A_3\left(x+\Delta;\frac hk\right)-\widetilde A_3\left(x;\frac hk\right)\right|^2\,\mathrm dx\asymp\Delta^2\,X^{13/3}k^5
\end{align}
in the range $k^{3/2}X^{1/2+\delta}\ll\Delta\ll X^{2/3-\delta}k$. Using similar ideas as before, these can be used to deduce lower bounds for twisted short sums of the Hecke eigenvalues on average when one also averages over the numerator of the exponential twist, which will in turn yield an $\Omega$-result for short sums. 

\begin{theorem}\label{ShortA_0-lower-bound}
Let $X\in[1,\infty[$, let $k$ be a prime and suppose that $k^{3/2}\,X^{1/2+\delta}\ll\Delta\ll kX^{2/3-\delta}$ for any sufficiently small fixed $\delta>0$. Then, for sufficiently large $X$, we have 
\begin{align*}
\mathop{\text{\LARGE$\E$}}_{h\in\mathbb Z_k^\times}\int\limits_X^{4X}\left|\widetilde A_0\left(x+\Delta;\frac hk\right)-\widetilde A_0\left(x;\frac hk\right)\right|^2\,\mathrm dx\gg\Delta^2\,X^{1/3}k^{-1}.
\end{align*}
\end{theorem}

\noindent Notice that this result immediately implies Theorem \ref{Short-mean-square} (and consequently Theorem \ref{Short-Omega}) as the residue terms in the sums $\widetilde A_0(x;h/k)$ cancel each other. 
  
\begin{proof}
We begin by proving a lower bound for the mean square of a short first order Riesz sum $\widetilde A_1(x+\Delta;h/k)-\widetilde A_1(x;h/k)$. As in the proof of Theorem \ref{lower bounds}, this time using the pointwise bound $\widetilde A_3(x;h/k)\ll k^{7/2}x^{7/3}$ from Corollary \ref{a2bound}, we have for any $0<H\leq X$ that
\begin{align}\label{connecting_ineq_2}
&\mathop{\text{\LARGE$\E$}}_{h\in\mathbb Z_k^\times}\int\limits_X^{2X}\left|\widetilde A_2\left(x+\Delta;\frac hk\right)-\widetilde A_2\left(x;\frac hk\right)\right|^2\,\mathrm dx \nonumber\\
&\ll \frac1{H^2}\mathop{\text{\LARGE$\E$}}_{h\in\mathbb Z_k^\times}\int\limits_X^{2X}\left|\widetilde A_3\left(x+\Delta;\frac hk\right)-\widetilde A_3\left(x;\frac hk\right)\right|^2\,\mathrm dx \nonumber\\
&\qquad\qquad+H^2\mathop{\text{\LARGE$\E$}}_{h\in\mathbb Z_k^\times}\int\limits_X^{3X}\left|\widetilde A_1\left(x+\Delta;\frac hk\right)-\widetilde A_1\left(x;\frac hk\right)\right|^2\,\mathrm dx+\frac1H k^7X^{14/3} \nonumber\\
&\ll\frac1{H^2}\,\Delta^2\,X^{13/3}k^5+H^2\mathop{\text{\LARGE$\E$}}_{h\in\mathbb Z_k^\times}\int\limits_X^{3X}\left|\widetilde A_1\left(x+\Delta;\frac hk\right)-\widetilde A_1\left(x;\frac hk\right)\right|^2\,\mathrm dx+\frac1H k^7X^{14/3},
\end{align}
where the last estimate follows from (\ref{A_3-meansquare}). 

Next we shall optimise the choice of the parameter $H$. Actually, before that we need a preliminary lower bound for the moment 
\begin{align*}
\mathop{\text{\LARGE$\E$}}_{h\in\mathbb Z_k^\times}\int\limits_X^{3X}\left|\widetilde A_1\left(x+\Delta;\frac hk\right)-\widetilde A_1\left(x;\frac hk\right)\right|^2\,\mathrm dx
\end{align*}
we are interested in. This is to guarantee that our optimal choice for $H$ satisfies the constraint $H\ll \Delta^2X^{-1/3}k^{-2}$ (and in particular $H\leq X$ when $X$ is large enough). This is needed so that $H^{-1}X^{14/3}k^7\ll H^{-2}\Delta^2X^{13/3}k^5$.

We make the initial choice $H=\Delta^2X^{-1/3}k^{-2}$ (which satisfies the constraint $H\leq X$ for sufficiently large $X$ by the assumption $\Delta\ll X^{2/3-\delta}k$) in (\ref{connecting_ineq_2}). This gives
\begin{align}\label{prelim_lower_bound_b}
\mathop{\text{\LARGE$\E$}}_{h\in\mathbb Z_k^\times}\int\limits_X^{3X}\left|\widetilde A_1\left(x+\Delta;\frac hk\right)-\widetilde A_1\left(x;\frac hk\right)\right|^2\,\mathrm dx\gg \Delta^{-2}k^7X^{11/3}
\end{align}
using (\ref{A_2-meansquare-2}) as $\Delta\gg k^{3/2}X^{1/2+\delta}$. We now make the final choice for $H$ by choosing
\[ 
H=\frac{\Delta^{1/2}X^{13/12}k^{5/4}}{\sqrt[4]{\mathop{\text{\LARGE$\E$}}_{h\in\mathbb Z_k^\times}\int\limits_X^{3X}\left|\widetilde A_1\left(x+\Delta;\frac hk\right)-\widetilde A_1\left(x;\frac hk\right)\right|^2\,\mathrm dx}}.
\]
This satisfies the constraint $H\ll \Delta^2X^{-1/3}k^{-2}$ thanks to (\ref{prelim_lower_bound_b}) and the assumption $\Delta\gg X^{1/2+\delta}k^{3/2}$. Using this choice in (\ref{connecting_ineq_2}) gives
\begin{align}\label{prelim_lower_bound_29}
\mathop{\text{\LARGE$\E$}}_{h\in\mathbb Z_k^\times}\int\limits_X^{3X}\left|\widetilde A_1\left(x+\Delta;\frac hk\right)-\widetilde A_1\left(x;\frac hk\right)\right|^2\,\mathrm dx\gg\left(\frac{\Delta^2X^3k^3}{\Delta X^{13/6}k^{5/2}}\right)^2=\Delta^2X^{5/3}k
\end{align}
by (\ref{A_2-meansquare-2}), giving a lower bound for the short first order Riesz sum.

We now turn into estimating the sum appearing in the statement of the theorem. Again, arguing as in the proof of Theorem \ref{lower bounds} and using (\ref{A_2-meansquare-2}) together with the bound $\widetilde A_2(x;h/k)\ll x^{5/3}k^{5/2}$ gives
\begin{align}\label{connecting_ineq_3}
&\mathop{\text{\LARGE$\E$}}_{h\in\mathbb Z_k^\times}\int\limits_X^{3X}\left|\widetilde A_1\left(x+\Delta;\frac hk\right)-\widetilde A_1\left(x;\frac hk\right)\right|^2\,\mathrm dx \nonumber\\
&\ll\frac1{H^2}\mathop{\text{\LARGE$\E$}}_{h\in\mathbb Z_k^\times}\left(\int\limits_X^{2X}+\int\limits_{3X/2}^{3X}\right)\left|\widetilde A_2\left(x+\Delta;\frac hk\right)-\widetilde A_2\left(x;\frac hk\right)\right|^2\,\mathrm dx \nonumber \\
&\qquad+H^2\mathop{\text{\LARGE$\E$}}_{h\in\mathbb Z_k^\times}\int\limits_X^{4X}\left|\widetilde A_0\left(x+\Delta;\frac hk\right)-\widetilde A_0\left(x;\frac hk\right)\right|^2\,\mathrm dx+\frac1H X^{10/3}k^5 \nonumber\\
&\ll\frac1{H^2}\,\Delta^2\,X^3\,k^3+H^2\mathop{\text{\LARGE$\E$}}_{h\in\mathbb Z_k^\times}\int\limits_X^{4X}\left|\widetilde A_0\left(x+\Delta;\frac hk\right)-\widetilde A_0\left(x;\frac hk\right)\right|^2\,\mathrm dx+\frac1H X^{10/3}k^5
\end{align}
for any $0<H\leq X$.

Next we shall optimise the choice of the parameter $H$. Actually, before that we need a preliminary lower bound for the moment 
\begin{align*}
\mathop{\text{\LARGE$\E$}}_{h\in\mathbb Z_k^\times}\int\limits_X^{4X}\left|\widetilde A_0\left(x+\Delta;\frac hk\right)-\widetilde A_0\left(x;\frac hk\right)\right|^2\,\mathrm dx
\end{align*}
we are interested in. This is to guarantee that our optimal choice for $H$ satisfies the constraint $H\ll \Delta^2X^{-1/3}k^{-2}$ (and in particular $H\leq X$ when $X$ is large enough in view of the assumption $\Delta\ll kX^{2/3-\delta}$). This is needed so that $H^{-1}X^{10/3}k^5\ll H^{-2}\Delta^2X^3k^3$.

We make the initial choice $H=\Delta^2X^{-1/3}k^{-2}$ (which again satisfies the constraint $H\leq X$ for sufficiently large $X$) in (\ref{connecting_ineq_3}). This gives
\begin{align}\label{prelim_lower_bound_30}
\mathop{\text{\LARGE$\E$}}_{h\in\mathbb Z_k^\times}\int\limits_X^{4X}\left|\widetilde A_0\left(x+\Delta;\frac hk\right)-\widetilde A_0\left(x;\frac hk\right)\right|^2\,\mathrm dx\gg \Delta^{-2}X^{7/3}k^5
\end{align}
using (\ref{prelim_lower_bound_29}) as $\Delta\gg X^{1/2+\delta}k^{3/2}$. We now make the final optimal choice for $H$ by choosing
\[ 
H=\frac{\Delta^{1/2}X^{3/4}k^{3/4}}{\sqrt[4]{\mathop{\text{\LARGE$\E$}}_{h\in\mathbb Z_k^\times}\int\limits_X^{4X}\left|\widetilde A_0\left(x+\Delta;\frac hk\right)-\widetilde A_0\left(x;\frac hk\right)\right|^2\,\mathrm dx}}.
\]
This satisfies the constraint $H\ll \Delta^2X^{-1/3}k^{-2}$ thanks to (\ref{prelim_lower_bound_30}) and the assumption $\Delta\gg X^{1/2+\delta}k^{3/2}$. Using this choice in (\ref{connecting_ineq_3}) gives
 
\begin{align*}
k\,\Delta^2\,X^{5/3}&\ll \mathop{\text{\LARGE$\E$}}_{h\in\mathbb Z_k^\times}\int\limits_X^{3X}\left|\widetilde A_1\left(x+\Delta;\frac hk\right)-\widetilde A_1\left(x;\frac hk\right)\right|^2\,\mathrm dx\\
&\ll k^{3/2}\,X^{3/2}\Delta\sqrt{\mathop{\text{\LARGE$\E$}}_{h\in\mathbb Z_k^\times}\int\limits_X^{4X}\left|\widetilde A_0\left(x+\Delta;\frac hk\right)-\widetilde A_0\left(x;\frac hk\right)\right|^2\,\mathrm dx},
\end{align*}
and hence
\[\mathop{\text{\LARGE$\E$}}_{h\in\mathbb Z_k^\times}\int\limits_X^{4X}\left|\widetilde A_0\left(x+\Delta;\frac hk\right)-\widetilde A_0\left(x;\frac hk\right)\right|^2\,\mathrm dx\gg\left(\frac{\Delta^2 X^{5/3}k}{\Delta X^{3/2}k^{3/2}}\right)^2=\Delta^2\,X^{1/3}\,k^{-1}.\]
This completes the proof. 
\end{proof}

\section*{Appendix: Corrections to \cite{Jaasaari--Vesalainen1}}\label{gl3sums-corrections}

\noindent In \cite{Jaasaari--Vesalainen1} there was a wrong factor $d$ instead of the correct factor $d^{1/3}$ in Theorem 1. This was corrected in \cite{Jaasaari--Vesalainen3}, which also described how the proofs of Theorem 2 and Corollary 3 need to be modified in order to obtain the same conclusions. Alas, the computation of the asymptotics for the Meijer $G$-function revealed another mistake in the proof of Theorem 1 of \cite{Jaasaari--Vesalainen1}. Namely, when $j=1$, the $\Gamma$-quotient expression used to approximate the $\Gamma$-quotient coming from the functional equation of the twisted $L$-function of the underlying Maass cusp form is correct only in the upper half-plane; in the lower half-plane the factor $i^j$ should read $(-i)^j$, thereby changing Theorem 1, though fortunately not really affecting the proofs of Theorem 2 and Corollary 3 apart from obvious cosmetic changes. The corrected formulation of Theorem 1 of \cite{Jaasaari--Vesalainen1} reads as follows.
\begin{theorem}
Let $x,N\in\left[2,\infty\right[$ with $N\ll x$, and let $h$ and $k$ be coprime integers with $1\leqslant k\leqslant x$ and $k\ll\left(Nx\right)^{1/3}$, the latter having a sufficiently small implicit constant depending on the underlying Maass cusp form. Then we have
\begin{multline*}
\sum_{m\leqslant x}A(m,1)\,e\!\left(\frac{mh}k\right)\\
=\frac{x^{1/3}}{2\pi\sqrt3}\sum_{d\mid k}\frac1{d^{1/3}}\sum_{d^2\,m\leqslant N_k}\frac{A(d,m)}{m^{2/3}}\sum_\pm S\!\left(\overline h,\pm m;\frac kd\right)e\!\left(\pm\frac{3\,d^{2/3}\,m^{1/3}\,x^{1/3}}k\right)\\
+O(k\,x^{2/3+\vartheta+\varepsilon}\,N^{-1/3})+O(k\,x^{1/6+\varepsilon}\,N^{1/6+\vartheta}).
\end{multline*}
\end{theorem}

\section*{Acknowledgements}

\noindent During the research the author was supported by the Finnish Cultural Foundation, by the Emil Aaltonen Foundation, and by the Engineering and Physical Sciences Research Council [grant number EP/T028343/1]. The author is grateful to Esa V. Vesalainen for several ideas and useful conversations related to the present work as well as sharing some of his computations, which were helpful in preparing this paper. The author also wishes to thank the anonymous referee for a detailed reading of the paper and many helpful comments which improved this work.

\bibliography{gl3omega}{}
\bibliographystyle{plain}

\end{document}